\RequirePackage{fix-cm}
\RequirePackage{amsmath}
\documentclass[smallextended]{svjour3}       % onecolumn (second format)
\smartqed  % flush right qed marks, e.g. at end of proof
\usepackage{graphicx}
\usepackage{versions}
\usepackage{subcaption}
\usepackage{lipsum}
\usepackage{amsfonts}
\usepackage{graphicx}
\usepackage{epstopdf}
\usepackage{wrapfig}

\usepackage{amssymb,tikz,stmaryrd}
\usepackage{tikz}
\usetikzlibrary{fit, arrows, shapes, positioning, shadows, matrix, calc}

 \tikzset{cross/.style={cross out, draw=black, minimum size=2*(#1-\pgflinewidth), inner sep=0pt, outer sep=0pt},
cross/.default={4pt}}

\usetikzlibrary{decorations.pathreplacing}
\tikzstyle std=[line width=0.7pt]   % line
\tikzstyle stdthin=[line width=0.3pt]   % line thin
\tikzstyle stdthick=[line width=1.0pt]   % line thick

\tikzstyle fwd=[line width=0.7pt, ->]   % forward step
\tikzstyle fwdthin=[line width=0.3pt, ->]   % forward thin
\tikzstyle fwdthick=[line width=1.0pt, ->]   % forward thick
\tikzstyle fwddash=[line width=0.7pt, dashed, ->]   % a step parallel to another

\tikzstyle bwd=[double, line width=0.3pt, ->]  % backward step
\tikzstyle refl=[double,  dashed, line width=0.2pt, ->]    % reflection step

\tikzstyle{every node}=[font=\small] % set "small" to texts

\tikzset{
  >=stealth', % arrow head style
  invisible/.style={opacity=0}, % invisible
  alt/.code args={<#1>#2#3}{\alt<#1>{\pgfkeysalso{#2}}{\pgfkeysalso{#3}}}, % animation
  visible on/.style={alt=#1{}{invisible}}, % simplified animation
  smallnode/.style={circle, fill=black, thick, inner sep=1pt, minimum size=1.5pt}, % the smalldot style for nodes
  punkt/.style={
           rectangle,
           rounded corners,
           draw=black, very thick,
           text width=5.5em,
           minimum height=2em,
           text centered},
  punkt_big/.style={
           rectangle,
           rounded corners,
           draw=black, very thick,
           text width=7em,
           minimum height=2em,
           text centered},
}
\newtheorem{fact}{Fact}

  \usetikzlibrary{positioning}
  \usetikzlibrary{shadows}
  \usetikzlibrary{fit, arrows, shapes, positioning, shadows, matrix, calc}
\usetikzlibrary{decorations.pathreplacing}

\usepackage{graphicx}

\newcommand{\reals}{\mathbb{R}}
\newcommand{\complex}{\mathbb{C}}
\newcommand{\graph}{\mathrm{graph}}

\newcommand{\hilbert}{\mathcal{H}}

\usepackage{mathtools}
\usepackage[normalem]{ulem} % to use \sout

\usepackage{amsmath}
\usepackage{amssymb}
\usepackage{mathtools}
\usepackage{mathrsfs}

\newcommand{\cut}[1]{{}}
\newcommand{\CC}{\mathbb{C}}				% Complex numbers
\newcommand{\dom}[1]{{\mathrm{dom}(#1)}} 	% domain
\newcommand{\binfty}{{\boldsymbol \infty}}
\renewcommand{\Re}{\operatorname{Re}} 	%Real part
\renewcommand{\Im}{\operatorname{Im}}	%imaginary part
\newcommand{\ecomplex}{\overline{\mathbb{C}}}	%Extended complex plane

\newcommand{\DeclareAutoPairedDelimiter}[3]{%
	\expandafter\DeclarePairedDelimiter\csname Auto\string#1\endcsname{#2}{#3}%
	\begingroup\edef\x{\endgroup
		\noexpand\DeclareRobustCommand{\noexpand#1}{%
			\expandafter\noexpand\csname Auto\string#1\endcsname*}}%
	\x}

\DeclareAutoPairedDelimiter{\p}{(}{)} 					%parentheses
\DeclareAutoPairedDelimiter{\sp}{[}{]} 					%square parentheses
\DeclareAutoPairedDelimiter{\abs}{|}{|} 					%square parentheses
\DeclareAutoPairedDelimiter{\cp}{\{}{\}} 				%curly parentheses
\DeclareAutoPairedDelimiter{\dotp}{\langle}{\rangle} 	%square parentheses
\DeclareAutoPairedDelimiter{\n}{\Vert}{\Vert} 			%Norm double bars
\DeclareAutoPairedDelimiter{\cl}{\lceil}{\rceil}

\newcommand{\cA}{{\mathcal{A}}}
\newcommand{\cB}{{\mathcal{B}}}
\newcommand{\cC}{{\mathcal{C}}}

\newcommand{\cF}{{\mathcal{F}}}
\newcommand{\cG}{{\mathcal{G}}}
\newcommand{\cH}{{\mathcal{H}}}
\newcommand{\cI}{{\mathcal{I}}}

\newcommand{\cL}{{\mathcal{L}}}
\newcommand{\cM}{{\mathcal{M}}}
\newcommand{\cN}{{\mathcal{N}}}

\usepackage{times,amsmath,graphicx,amssymb}
\usepackage{url}
\usepackage{tikz-3dplot}
\usetikzlibrary{shapes,calc,positioning}
\usetikzlibrary{calc,intersections}
\usepackage{tkz-euclide}
\makeatletter
\DeclareFontFamily{U}{tipa}{}
\DeclareFontShape{U}{tipa}{m}{n}{<->tipa10}{}
\newcommand{\arc@char}{{\usefont{U}{tipa}{m}{n}\symbol{62}}}%

\newcommand{\arc}[1]{\mathpalette\arc@arc{#1}}

\newcommand{\arc@arc}[2]{%
  \sbox0{$\m@th#1#2$}%
  \vbox{
    \hbox{\resizebox{\wd0}{\height}{\arc@char}}
    \nointerlineskip
    \box0
  }%
}
\makeatother

\usepackage{amsopn}

\newcommand{\rarc}{\mathrm{Arc}^+}
\newcommand{\larc}{\mathrm{Arc}^-}

\definecolor{lightgrey}{gray}{0.8}
\definecolor{medgrey}{gray}{0.6}
\definecolor{darkgrey}{gray}{0.4}
\begin{document}

\title{Scaled Relative Graphs:
\thanks{This work was partially supported by AFOSR MURI FA9550-18-1-0502, NSF Grant DMS-1720237, ONR Grant N000141712162, the New Faculty Startup Fund from Seoul National University, the National Research Foundation of Korea (NRF) Grant funded by the Korean Government (MSIP) [No. 2020R1F1A1A01072877], and the National Research Foundation of Korea (NRF) Grant funded by the Korean Government (MSIP) [No. 2017R1A5A1015626].}
}
\subtitle{Nonexpansive operators via 2D Euclidean Geometry}

\titlerunning{Scaled Relative Graphs}        % if too long for running head

\author{Ernest K. Ryu         \and Robert Hannah \and Wotao Yin
}

\authorrunning{Ryu, Hannah, and Yin} % if too long for running head

\institute{E. K. Ryu \at
              Seoul National University, department of mathematical sciences \\
              \email{ernestryu@snu.ac.kr}           %  \\
           \and
           R. Hannah \at
              UCLA, department of mathematics\\
              \email{RobertHannah89@math.ucla.edu}
              \and
        W. Yin \at
        UCLA, department of mathematics\\
        \email{wotaoyin@math.ucla.edu}
}

\date{Received: date / Accepted: date}

\maketitle

\begin{abstract}
Many iterative methods in applied mathematics can be thought of as fixed-point iterations, and such algorithms are usually analyzed analytically, with inequalities. In this paper, we present a geometric approach to analyzing contractive and nonexpansive fixed point iterations with a new tool called the scaled relative graph (SRG). The SRG provides a correspondence between nonlinear operators and subsets of the 2D plane. Under this framework, a geometric argument in the 2D plane becomes a rigorous proof of convergence.
\end{abstract}

\section{Introduction}
%XXX - reacted to that ``To clarify'' is used a lot.XXX
%XXX
%For readers looking up results in the paper without reading everything, it would be useful to
%(at least implicitly) repeat definitions like M and LL in the theorem statements.
%XXX
Fixed-point iterations abound in applied mathematics and engineering.
This classical technique, dating back to \cite{newton1669,picard1890,lindelof1894}, involves the following two steps.
First, find an operator $T:\mathcal{X}\rightarrow\mathcal{X}$, where $\mathcal{X}$ is some space, such that if $x^\star=T(x^\star)$, i.e., if $x^\star$ is a fixed point, then $x^\star$ is a solution to the problem at hand.
Second, perform the \emph{fixed-point iteration} $x^{k+1}=T(x^k)$.
%and prove, or at least hope for, convergence.
% A well-known approach to establishing convergence of a fixed-point iteration is to show that $T$ is a contraction.
% More precisely, if $\mathcal{X}$ is a nonempty complete metric space with metric $d$ and if the contractive inequality
% \begin{equation*}
% d(T(x),T(y))\le L d(x,y),
% \quad
% \forall\, x,y\in \mathcal{X}
% \label{eq:banach-contraction}
% \end{equation*}
% holds for some $L<1$, then the iteration geometrically converges to the fixed point $x^\star$ with rate $d(x^k,x^\star)\le L^kd(x^0,x^\star)$. %\cite{banach1922}.
% %Much research in applied mathematics analyze iterative method through this approach.
% %is dedicated to showing this contractive inequality for an 
% Proving this contractive inequality is the main challenge, especially when the operator is nonlinear and is constructed with compositions, inversions, and resolvents.
Convergence of such iterative methods is usually proved analytically, through a series of inequalities.

% Another approach is to show $T\colon\hilbert\rightarrow\hilbert$ is a
% \begin{equation*}
% \|T(x)-T(y)\|^2\le \|x-y\|^2-\frac{1-\theta}{\theta}\|x-T(x)-y-T(y)\|^2,
% \quad
% \forall\, x,y\in \hilbert
% \end{equation*}

In this paper, we present a \emph{geometric} approach to analyzing contractive and nonexpansive fixed-point iterations with a new tool called the scaled relative graph (SRG).
We can think of the SRG as a signature of an operator analogous to how eigenvalues are a signature of a matrix.
%The SRG of an operator encodes a lot of important or useful information about this operator: including its real spectrum (if it is a symmetric linear operator); whether it is injective or multi-valued; whether it is $L$-Lipschitz, $\alpha$-averaged, $\beta$-cocoercive, etc. and with what parameters $L$, $\alpha$, and $\beta$ respectively.
The SRG provides a correspondence between algebraic operations on nonlinear operators and geometric operations on subsets of the 2D plane.
Using this machinery and elementary Euclidean geometry, we can establish properties of operators (such as contractiveness) and establish the convergence of fixed-point iterations through showing the SRG, a set in the 2D plane, resides within certain circles.
These geometric arguments form rigorous proofs, not just illustrations.

One advantage of geometric proofs is that a single or a few geometric diagrams concisely capture and communicate the core insight.
In contrast, it is much more difficult to extract a core insight from a classical analytic proof based on inequalities.
Another advantage is that tightness, loosely defined as being unable to improve a stated result without additional assumptions, is often immediate.
In contrast, discerning whether it is possible to make improvements when examining a proof based on inequalities is usually more difficult; providing a matching lower bound is often the only way to establish tightness of such results.

% is contractive by proving its SRG is contained in the circle with radius less than $1$ centered at the origin using elementary Euclidean geometry.
%Many results in linear algebra are stated and proved using eigenvalues of matrices.
%We show that many classical and novel results in operator theory can be proved rigorously with using SRGs of operators and elementary Euclidean geometry.

\subsection{Proving convergence with operator properties}
Given $T\colon\hilbert\rightarrow\hilbert$, where $\hilbert$ is a real Hilbert space with norm $\|\cdot\|$, consider the fixed-point iteration given by 
\[
x^{k+1}=T(x^k)
\]
for $k=0,1,\dots$ where $x^0\in\hilbert$ is a starting point.
We say $x^\star$ is a \emph{fixed point} of $T$ if $x^\star=T(x^\star)$.
We say $T\colon\hilbert\rightarrow\hilbert$ is \emph{nonexpansive} if
\[
\|Tx)-T(y)\|\le \|x-y\|,\qquad\forall x,y\in \hilbert.
\]
In this case, $\|x^{k}-x^\star\|$ is a nonincreasing sequence, but $x^k$ need not converge.
For instance, if $T=-I$, then $x^k$ oscillates between $x^0$ and $-x^0$.
We say $T\colon\hilbert\rightarrow\hilbert$ is \emph{contractive} if
\[
\|T(x)-T(y)\|\le L\|x-y\|,\qquad\forall x,y\in \hilbert
\]
for some $L<1$.
In this case, $x^k\rightarrow x^\star$ strongly with rate $\|x^k-x^\star\|\le L^k\|x^0-x^\star\|$.
%When $T$ is contractive, the iteration geometrically converges to the fixed point $x^\star$ with rate  
This classical argument is the Banach contraction principle \cite{banach1922}.
We say $T\colon\hilbert\rightarrow\hilbert$ is \emph{averaged} if $T=(1-\theta)I+\theta R$ for some nonexpansive operator $R$ and $\theta\in(0,1)$, where $I$ is the identity operator.
In this case,
$x^k\rightarrow x^\star$ weakly for a fixed point $x^\star$ provided that $T$ has a fixed point.
This result is the Krasnosel'ski\u{\i}--Mann theorem \cite{mann1953,krasnoselskii1955}.
%(To find a fixed point of a nonexpansive operator $T$ that is not necessarily averaged, the averaged operator $S=(1-\theta)I+\theta T$ share the same This ensures the iteration converges, with essentially no additional computational cost.)
The assumption of averagedness is stronger than nonexpansiveness and weaker than contractiveness, as illustrated in Figure~\ref{fig:cont-avg-nonexp}.

We now have a general rubric for proving convergence of a fixed-point iteration:
\begin{enumerate}
    \item Prove the operator $T$ is contractive or averaged.
    \item Apply the convergence argument of Banach or Krasnosel'ski\u{\i}--Mann.%, respectively.
\end{enumerate}
Many, although not all, fixed-point iterations are analyzed through this rubric.
%Many, although not all, fixed-point iterations are analyzed by establishing that the underlying operator is contractive or averaged.
Step 2 is routine.
This work presents a geometric approach to step 1, the more difficult step.

%Iterative methods for solving linear systems such as Jacobi and Gauss--Seidel and first-order optimization methods including gradient descent, iterative soft-thresholding, and split Bregman are such examples.

\begin{figure}
    \centering
\begin{tabular}{ccccc}
\raisebox{-.5\height}{
\begin{tikzpicture}[scale=1.2]
%\draw (-0.85,0.5) node {$\cG(\cN_\theta)=$};
\fill[fill=medgrey] (0,0) circle (0.75);
\draw[dashed] (0,0) circle (1);
\draw [<->] (-1.2,0) -- (1.2,0);
\draw [<->] (0,-1.2) -- (0,1.2);
\draw (1,0) node [above right] {$1$};
%\draw (-0.75,0) node [above left] {$-L$};
\filldraw (1,0) circle ({0.6*1.5/1.2pt});
\draw (0.75,0) node [above left] {$L$};
\filldraw (0.75,0) circle ({0.6*1.5/1.2pt});
\draw (0,-1.4) node  {Contractive};
\end{tikzpicture}}
&
$\subset$
&
\raisebox{-.5\height}{
\begin{tikzpicture}[scale=1.2]
%\draw (-0.85,0.5) node {$\cG(\cN_\theta)=$};
\fill[fill=medgrey] (0.125,0) circle (0.875);
\draw[dashed] (0,0) circle (1);
\draw [<->] (-1.2,0) -- (1.2,0);
\draw [<->] (0,-1.2) -- (0,1.2);
\filldraw (.125,0) circle ({0.6*1.5/1.5pt});
\draw (0.575,0.1) node [above] {$\theta$};
%\draw (1,0) node [above right] {$1$};
\filldraw (1,0) circle ({0.6*1.5/1.2pt});
\draw [decorate,decoration={brace,amplitude=4.5pt}] (0.125,0) -- (1,0) ;
\draw (0,-1.4) node  {Averaged};
\end{tikzpicture}}
&
$\subset$&
\raisebox{-.5\height}{
\begin{tikzpicture}[scale=1.2]
%\draw (-1.1,0.5) node {$\cG(\cL_L)=$};
\fill[fill=medgrey] (0,0) circle (1);
\draw[dashed] (0,0) circle (1);
%\draw[line width=1.0pt] (0,0) circle (0.75);
\draw [<->] (-1.2,0) -- (1.2,0);
\draw [<->] (0,-1.2) -- (0,1.2);
\draw (1,0) node [above right] {$1$};
%\draw (-0.75,0) node [above left] {$-L$};
\filldraw (1,0) circle ({0.6*1.5/1.2pt});
%\filldraw (-.75,0) circle ({0.6*1.5/1.2pt});
\draw (0,-1.4) node  {Nonexpansive};
\end{tikzpicture}}
\end{tabular}
    \caption{
The classes of contractive, averaged, and nonexpansive operators represented with the scaled relative graph (SRG).
The notion of the SRG and the precise meaning of these figures will be defined soon in Section~\ref{s:srg}.
}
    \label{fig:cont-avg-nonexp}
\end{figure}

\subsection{Prior work and contribution}
%Using geometric shapes to graphically illustrate dynamical systems and iterative methods is a very natural and classical approach. Phase plane and cobweb diagram are such examples.
Using circles or disks centered at the origin to illustrate contractive mappings is natural and likely common.
Eckstein and Bertsekas's illustration of firm-nonexpansiveness via the disk with radius $1/2$ centered at $(1/2,0)$ \cite{eckstein1989,eckstein1992} was, to the best of our knowledge, the first geometric illustration of notions from fixed-point theory other than nonexpansiveness and Lipschitz continuity.
Since then, Giselsson and Boyd used similar illustrations in earlier versions of the paper \cite{giselsson2017linear}
(the arXiv versions 1 through 3 have the geometric diagrams, but later versions do not)
and more thoroughly in the lecture slides \cite{giselsson_slides}.
Banjac and Goulart also utilize similar illustrations \cite{banjac2018}.

Through personal communication, we are aware that many have privately used geometric illustrations similar to those presented in this paper to initially build intuition, although the actual mathematics and proofs were eventually presented analytically, with inequalities.
To the best of our knowledge, the use of geometry for \emph{rigorous} proofs of results of nonlinear operators is new.

The notion of the SRG was first defined and presented in the authors' unpublished manuscript \cite{hannah2016}. The work shows how transformations of the operator such as inversion, addition of identity, unitary change in coordinates, and composition map to changes in the SRG
and used these transformations to geometrically rigorously prove many standard results.
It furthermore discusses the Baillon--Haddad Theorem and convergence rates for various operator methods.

Throughout this paper, we state known results as ``Facts''. 
Our contributions are the alternative geometric proofs, the novel results stated as ``Propositions'' and ``Theorems'', and the overall geometric approach based on the SRG.

%
%The paper is organized as follows. 
%Section~\ref{s:prelim} discusses general preliminaries and sets up the notation. 
%Section~\ref{s:srg} presents the notion of the SRG.
%%Section~\ref{s:srg-class} presents the notion of SRG-representability, which roughly means a class of operators is equivalent to its SRG, a set on the 2D plane.
%Section~\ref{s:srg-transformation} maps algebraic operations like scaling, inversion, addition, and composition of operators to geometric operations of the SRG.
%% Section~\ref{s:srg-sum} maps addition of operators to addition of the SRGs.
%% Section~\ref{s:srg-composition} maps compositions of operators to complex multiplications of the SRGs.
%%Section~\ref{s:invariant-circle} applies the machinery introduces the notion of invariant circles and presents an impossibility result.
%% Section~\ref{s:metric-subregularity} applies the machinery to understand metric subregularity and presents impossibility results on when metric subregularity is insufficient for establishing linear convergence.
%Section~\ref{s:conclusion} concludes the paper. 

%The use of (univariate polynomials) for the conjugate gradient method and other iterative methods have lead to great progress in the field \cite{lanczos1952,stiefel1954}.

\section{Preliminaries}
\label{s:prelim}
%In this section, we discuss preliminaries and set up the notation.%
We refer readers to standard references for more information on 
convex analysis \cite{hiriarturruty1993,boyd2004,beck_book2017},
nonexpansive and monotone operators \cite{BCBook,ryu2016},
and geometry \cite{wentworth1913,morley1933,pedoe1970}.
Write $\hilbert$ for a real Hilbert space equipped with the inner product $\langle\cdot,\cdot\rangle$ and norm $\|\cdot\|$.
%(Readers unfamiliar with infinite-dimensional analysis can simply consider $\hilbert=\reals^n$.)
% As $\hilbert$ can be infinite dimensional, we distinguish between weak and strong convergence.
% We will furthermore define multi-valued operators and subdifferentials,
% which generalize gradients to nondifferentiable convex functions.
% Readers unfamiliar with such concepts can safely assume all operators are single-valued, assume all functions are differentiable, and assume the space is finite dimensional, i.e., $\hilbert=\reals^n$,
% without loosing much.
%\textbf{Minkowski-type set notation.}
We use Minkowski-type set notation that generalizes operations on individual elements to sets.
For example, given $\alpha\in \reals$ and sets $U,V\subseteq\hilbert$, write
\[
\alpha U=\{\alpha u\,|\,u\in U\},\quad
U+V=\{u+v\,|\,u\in U,\,v\in V\},\quad U-V=U+(-V).
\]
Notice that if either $U$ or $V$ is $\emptyset$, then $U+V=\emptyset$.
In particular, $U+V$ is the Minkowski sum.
We use similar notation for sets of operators and complex numbers.
The meanings should be clear from context, but for the sake of precision, we provide the full definitions in the appendix.

\textbf{Multi-valued operators.}
For convex analytical and operator theoretic notions, we follow standard notation \cite{BCBook}.
In particular, we consider multi-valued operators, which map a point to a set.
The graph of an operator is defined as
\[
\graph(A)=\{(x,u)\,|\, u \in Ax\}.
%\cup_{u\in U} Au.
\]
For convenience, we do not distinguish an operator from its graph, writing $(x,u)\in A$ to mean $u\in Ax$.
Define the inverse operator as
\[
A^{-1}=\{(u,x)\,|\,(x,u)\in A\},
\]
which always exists.
Define the resolvent of $A$ is $J_A=(I+A)^{-1}$.

%We precisely define notion \emph{class of operators} as the notation is somewhat nonstandard.

We say $\cA$ is a \emph{class of operators} if  $\cA$ is a set of operators on Hilbert spaces.
Note that $A_1,A_2\in \cA$ need not be defined on the same Hilbert spaces, i.e., $A_1\colon\hilbert_1\rightrightarrows\hilbert_1$,
$A_2\colon\hilbert_2\rightrightarrows\hilbert_2$,
and $\hilbert_1\ne\hilbert_2$ is possible.
%For example, $\cL_1$ is the class of nonexpansive operators on \emph{all} Hilbert spaces.

Given classes of operators $\cA$ and $\cB$, write
\begin{align*}
\cA+\cB&=\{A+B\,|\,A\in \cA,\,B\in \cB,\,A\colon\hilbert\rightrightarrows\hilbert,\,B\colon\hilbert\rightrightarrows\hilbert\}.
\end{align*}
To clarify, these definitions require that $A$ and $B$ or $A$ and $I$ are operators on the same (but arbitrary) Hilbert space $\hilbert$, as otherwise the operations would not make sense.
We define $\cA\cB$, $I+\alpha \cA$, and $J_{\alpha \cA}$ similarly.
For $L\in(0,\infty)$, define the class of $L$-Lipschitz operators as
\begin{align*}
\cL_{L}&=\big\{A\colon\dom{A}\rightarrow\hilbert\,|\,\| Ax-Ay\|^2\le L^2\|x-y\|^2,\,\forall\, x,y\in\dom{A}\subseteq\cH\big\}.
\end{align*}
%where, to clarify, $\cH$ is an arbitrary Hilbert space.
%We call an operator a contraction if it is $L$-Lipschitz for some $L<1$, and nonexpansive if it is $1$-Lipschitz.
For $\beta\in(0,\infty)$, define the class of $\beta$-cocoercive operators as
\begin{align*}
\cC_{\beta }&=\big\{A\colon\dom{A}\rightarrow\hilbert\,|\,\langle Ax-Ay,x-y\rangle\ge \beta\|Ax-Ay\|^2,\,\forall\, x,y\in\dom{A}\subseteq\cH\big\}.
\end{align*}
Define the class of monotone operators as
\begin{align*}
\cM&=\big\{A\colon\hilbert\rightrightarrows\hilbert\,|\,\langle Ax-Ay,x-y\rangle\ge 0,\,\forall\, x,y\in\hilbert\big\}.
\end{align*}
To clarify, $\langle Ax-Ay,x-y\rangle\ge 0$ means $\langle u-v,x-y\rangle\ge 0$ for all $(x,u),(y,v)\in A$. 
If $x\notin \dom{A}$, then the inequality is vacuous.
%Hence you can assume that $x,y\in\dom{A}$ with no loss in generality.
A monotone operator $A$ is \emph{maximal} if there is no other monotone operator $B$ such that $\graph(B)$ properly contains $\graph(A)$.
For $\mu\in(0,\infty)$, define the class of $\mu$-strongly monotone operators as
\begin{align*}
\cM_{\mu}&=\big\{A\colon\hilbert\rightrightarrows\hilbert\,|\,\langle Ax-Ay,x-y\rangle\ge \mu\|x-y\|^2,\,\forall\, x,y\in\hilbert\}.
\end{align*}
 For $\theta\in(0,1)$, define the class of $\theta$-averaged operators $\cN_\theta$ as
\begin{align*}
\cN_\theta=(1-\theta)I+\theta\cL_1.
\end{align*}
In these definitions, we do not impose any requirements on the domain or maximality of the operators.

Following the notation of \cite{Nesterov2013_introductory}, respectively write $\cF_{\mu,L}$, $\cF_{0,L}$, $\cF_{\mu,\infty}$, and $\cF_{0,\infty}$ for the sets of lower semi-continuous proper functions on all Hilbert spaces that are respectively $\mu$-strongly convex and $L$-smooth, convex and $L$-smooth, $\mu$-strongly convex, and convex, for $0 < \mu  < L < \infty$.
Write
\begin{align*}
\partial \cF_{\mu,L} = \{\partial f \,|\, f \in  \cF_{\mu,L}\},
\end{align*}
where $0 \leq  \mu  < L \leq  \infty$.

\textbf{Inversive geometry.}
We use the \emph{extended complex plane}
\index{extended complex plane}
$\ecomplex=\CC\cup\{\infty\}$ to represent the 2D plane and the point at infinity.
We call $z\mapsto\bar{z}^{-1}$, a one-to-one map from $\ecomplex$ to $\ecomplex$, the \emph{inversion} map. 
%(The inversion map is well-defined at $0$ because we use the extended complex plane $\ecomplex$.)
In polar form, it is $re^{i\varphi}\mapsto (1/r)e^{i\varphi}$ for $0\le r\le \infty$, i.e., inversion preserves the angle and inverts the magnitude.
In complex analysis, the inversion map is known as the M\"obius transformation \cite[p.\ 366]{ablowitz2003complex}.
 In classical Euclidean geometry, \emph{inversive geometry} considers generally the inversion of the 2D plane about any circle \cite[p.\ 75]{pedoe1970}.
Our inversion map $z\mapsto\bar{z}^{-1}$ is the inversion about the unit circle.

 \emph{Generalized circles} consist of (finite) circles and lines with $\{\infty\}$, and the interpretation is that a line is a circle with infinite radius.
Inversion maps generalized circles to generalized circles.
%The inversion map $z\mapsto\bar{z}^{-1}$ maps generalized circles to generalized circles.
Using a compass and straightedge, the inversion of a generalized circle can be constructed fully geometrically.
%For our purposes, the construction is somewhat simplified as we only need to perform the inversion with respect to the unit circle.
In this paper, we use the following semi-geometric construction:
\begin{enumerate}
    \item Draw a line $L$ through the origin orthogonally intersecting the generalized circle.
\item Let $-\infty<x< y\le \infty$ represent the signed distance of the intersecting points from the origin along this line.
If the generalized circle is a line, then $y=\infty$.
%If $x,y$ are on the opposite side of the origin, then $x<0$.
\item Draw a generalized circle orthogonally intersecting $L$ at $(1/x)$ and $(1/y)$.
\item When inverting a region with a generalized circle as the boundary, pick a point on $L$ within the interior of the region to determine on which side of the boundary the inverted interior lies.
\end{enumerate}
Figures~\ref{fig:inversion_example} and \ref{fig:inversion_example2} illustrate these steps.

\begin{figure}
\begin{center}
\begin{tabular}{cccc}
\raisebox{-.5\height}{
\begin{tikzpicture}[scale=1.0]
%\draw[dashed] (0,0) circle (1);
\def\r{0.85}
\fill[fill=medgrey] (0,0) circle (\r);
\draw [<->] (-1.5,0) -- (1.5,0);
\draw (-1,0) node [above left] {$L$};
\draw [<->] (0,-1.5) -- (0,1.5);
\draw (\r,0) node [below left] {$y$};
\filldraw (\r,0) circle  ({0.6*1.5/1.0pt});
\draw (-\r,0) node [below right] {$x$};
\filldraw (-\r,0) circle  ({0.6*1.5/1.0pt});
\filldraw (1,0) circle  ({0.6*1.5/1.0pt});
\draw (1.1,0) node [above] {$1$};
%\draw (0,-1.7) node {$C(e_1,a,b)$};
\draw (.95,1.25) node {\phantom{$\cup \{\binfty\}$}};

\begin{scope}
\clip (0,0) circle (\r);
\coordinate (A) at (-0.4,0) {};
\coordinate (B) at ({-\r-0.01},0) {};
\coordinate (C) at ({-\r-0.01},1) {};
\tkzMarkRightAngle[size=.1](A,B,C);
\end{scope}

\begin{scope}
\clip (0,0) circle (\r);
\coordinate (A) at (0.4,0) {};
\coordinate (B) at ({\r+0.01},0) {};
\coordinate (C) at ({\r+0.01},1) {};
\tkzMarkRightAngle[size=.1](A,B,C);
\end{scope}

\end{tikzpicture}}
&
\raisebox{-.5\height}{
\begin{tikzpicture}[scale=1.0]

\def\r{0.85}
\fill[fill=medgrey] (-1.5,-1.5) rectangle (1.5,1.5);
\fill[fill=white] (0,0) circle (1/\r);
\draw [<->] (-1.5,0) -- (1.5,0);
\draw (-1.1,0) node [above left] {$L$};
\draw [<->] (0,-1.5) -- (0,1.5);
\filldraw (1,0) circle  ({0.6*1.5/1.0pt});
\draw (1,0) node [above left] {$1$};
\begin{scope}
\clip (0,0) circle (1/\r);
%\draw[dashed] (0,0) circle (1);
\draw ({1/\r-0.0},-0.05) node [below left, fill=white] {$y^{-1}$};
\draw ({-1/\r-0.0},-0.05) node [below right, fill=white] {$x^{-1}$};
\end{scope}
\filldraw ({1/\r},0) circle ({0.6*1.5/1.0pt});
\filldraw ({-1/\r},0) circle ({0.6*1.5/1.0pt});
\draw (1.05,1.25) node {$\cup \{\binfty\}$};
%\draw (0,-1.7) node {$C(e_1,b^{-1},a^{-1})$};

\begin{scope}
\clip (0,0) circle ({(1/\r)});
\coordinate (A) at (-0.4,0) {};
\coordinate (B) at ({-(1/\r)-0.01},0) {};
\coordinate (C) at ({-(1/\r)-0.01},1) {};
\tkzMarkRightAngle[size=.1](A,B,C);
\end{scope}

\begin{scope}
\clip (0,0) circle ({(1/\r)});
\coordinate (A) at (0.4,0) {};
\coordinate (B) at ({(1/\r)+0.01},0) {};
\coordinate (C) at ({(1/\r)+0.01},1) {};
\tkzMarkRightAngle[size=.1](A,B,C);
\end{scope}
\end{tikzpicture}}
\end{tabular}
\end{center}
\caption{Illustration of inverting a disk.
In step 1, we choose $L$ to be the $x$-axis (although any line through the origin works).
In steps 2 and 3, we identify $x$ and $y$, invert them to $x^{-1}$ and $y^{-1}$, and draw the generalized circle in the inverted plane to be the new boundary.
%Recall that the unit circle, and in particular the point $1$, is invariant under the inversion map.
In step 4, we determine that the interior of the disk is mapped to the exterior by noting that $1$, a point invariant under the inversion map, is excluded in the original region and therefore is excluded in the inverted region.
%Alternatively, one could use the fact that the origin is mapped to $\infty$.
}
\label{fig:inversion_example}
\end{figure}

\begin{figure}
\begin{center}
\begin{tabular}{ccc}
\raisebox{-.5\height}{
\begin{tikzpicture}[scale=1.0]
\def\a{0.5}
\def\b{1.4}
\fill[fill=medgrey] ({(\a+\b)/2},0) circle ({(\b-\a)/2});
\draw [<->] (-.5,0) -- (2.5,0);
\draw [<->] (0,-1.2) -- (0,1.2);
\filldraw (\a,0) circle  ({0.6*1.5/1.0pt});
\draw (\a,0) node [below left] {$x$};
\filldraw (\b,0) circle  ({0.6*1.5/1.0pt});
\draw (\b,0) node [below right] {$y$};
\filldraw (1,0) circle  ({0.6*1.5/1.0pt});
\draw (1,0) node [above] {$1$};
%\draw (1,-1.7) node {$C(e_1,c,d)$};
\draw ({1/\a+.23},0) node [below] {\phantom{$x^{-1}$}};
\end{tikzpicture}}
&
\raisebox{-.5\height}{
\begin{tikzpicture}[scale=1.0]
\def\m{1.2}
\fill[fill=medgrey] (\m,-1.2) rectangle (2.5,1.2);
\draw [<->] (-.5,0) -- (2.5,0);
\draw [<->] (0,-1.2) -- (0,1.2);
\filldraw ({(\m)},0) circle  ({0.6*1.5/1.0pt});
\draw ({(\m)+0.1},0) node [below left] {$x$};
\filldraw (1,0) circle  ({0.6*1.5/1.0pt});
\draw (1,0) node [above] {$1$};
\draw (2.0,0.95) node {$\cup \{\binfty\}$};
\draw (1.45,0) node [below right] {$y=\infty$};
%\draw (1,-1.7) node {$C(e_1,f,\infty)$};
\draw (0.1,0) node [below left] {\phantom{$y^{-1}$}};
\end{tikzpicture}}
&
\raisebox{-.5\height}{
\begin{tikzpicture}[scale=1.0]
\fill[fill=medgrey] (0,-1.2) rectangle (2,1.2);
\draw [<->] (-0.7,0) -- (2,0);
\draw [<->] (0,-1.2) -- (0,1.2);
\filldraw (1,0) circle  ({0.6*1.5/1.0pt});
\draw (1,0) node [above] {$1$};
\filldraw (0,0) circle  ({0.6*1.5/1.0pt});
\draw (0,0) node [below left] {$x$};
\draw (1.55,.95) node {$\cup \{\binfty\}$};
\draw (2,0) node [below left] {$y=\infty$};
%\draw (0,-1.7) node {$C(e_1,0,\infty)$};
\end{tikzpicture}}
\\
\raisebox{-.5\height}{
\begin{tikzpicture}[scale=1.0]
\def\a{0.5}
\def\b{1.4}
\fill[fill=medgrey] ({(1/\a+1/\b)/2},0) circle ({(1/\a-1/\b)/2});
\draw [<->] (-.5,0) -- (2.5,0);
\draw [<->] (0,-1.2) -- (0,1.2);
\filldraw ({1/\a},0) circle  ({0.6*1.5/1.0pt});
\draw ({1/\a+.23},0) node [below] {$x^{-1}$};
\filldraw ({1/\b},0) circle  ({0.6*1.5/1.0pt});
\draw ({1/\b+.15},0) node [below left] {$y^{-1}$};
\filldraw (1,0) circle  ({0.6*1.5/1.0pt});
\draw (1,0) node [above] {$1$};
%\draw (1,-1.7) node {$C(e_1,d^{-1},c^{-1})$};
\end{tikzpicture}}
&
\raisebox{-.5\height}{
\begin{tikzpicture}[scale=1.0]
\def\m{1.2}
\fill[fill=medgrey] ({(1/\m)/2},0) circle ({(1/\m)/2});
\draw [<->] (-.5,0) -- (2.5,0);
\draw [<->] (0,-1.2) -- (0,1.2);
\filldraw (0,0) circle  ({0.6*1.5/1.0pt});
\draw (0.1,0) node [below left] {$y^{-1}$};
\filldraw ({(1/\m)},0) circle  ({0.6*1.5/1.0pt});
\draw ({(1/\m)+0.2},0) node [below ] {$x^{-1}$};
\filldraw (1,0) circle  ({0.6*1.5/1.0pt});
\draw (1,0) node [above] {$1$};
%\draw (1,-1.7) node {$C(e_1,0,f^{-1})$};
\draw (2.0,0.95) node {\phantom{$\cup \{\binfty\}$}};
\end{tikzpicture}}
&
\raisebox{-.5\height}{
\begin{tikzpicture}[scale=1.0]
\fill[fill=medgrey] (0,-1.2) rectangle (2,1.2);
\draw [<->] (-0.7,0) -- (2,0);
\draw [<->] (0,-1.2) -- (0,1.2);
\filldraw (1,0) circle  ({0.6*1.5/1.0pt});
\filldraw (0,0) circle  ({0.6*1.5/1.0pt});
\draw (0,0) node [below left] {$y^{-1}$};
\draw (1,0) node [above] {$1$};
\draw (1.55,.95) node {$\cup \{\binfty\}$};
\draw (2,0) node [below left] {$x^{-1}=\infty$};
%\draw (0,-1.7) node {$C(e_1,0,\infty)$};
\end{tikzpicture}}
\end{tabular}
\end{center}
\caption{The three vertical pairs illustrate inversion.
%illustrate how regions defined by generalized circles are mapped under the inversion.
In steps 1, we choose $L$ to be the $x$-axis.
In steps 4 we determine the interior by examining point $1$:
if $1$ is included in the original region, it is included in the inverted region, and vice-versa.
}
\label{fig:inversion_example2}
\end{figure}

\section{Scaled relative graphs}
\label{s:srg}
In this section, we define the notion of \emph{scaled relative graphs} (SRG). Loosely speaking, SRG maps the action of an operator to a set on the extended complex plane.

We use the \emph{extended complex plane} $\ecomplex=\CC\cup\cp{\infty}$ to represent the 2D plane and the point at infinity.
Since complex numbers compactly represent rotations and scaling, this choice simplifies our notation compared to using $\reals^2\cup\cp{\infty}$.
%However, we do not use much complex analysis.
%for $z\in \ecomplex$, 
We avoid the operations $\infty+\infty$, $0/0$, $\infty/\infty$, and $0\cdot\infty$.
Otherwise, we adopt the convention of $z+\infty=\infty$, $z/\infty=0$, $z/0=\infty$, and $z\cdot\infty=\infty$.

% Given $\alpha\in \complex$, $\alpha\ne 0$, and $Z\subseteq\ecomplex$, define $\alpha Z = \{ \alpha z\,|\,z\in Z\}$.

\begin{figure}
\centering
\begin{subfigure}[b]{0.3\textwidth}
\raisebox{-.5\height}{
\begin{tikzpicture}[scale=1.5]
\draw(-0.7,0) node {$\cG(P_L)=$};
\draw [line width=1.0pt] (0.5,0) circle (0.5);
\draw [<->] (-.2,0) -- (1.2,0);
\draw [<->] (0,-.7) -- (0,.7);
\draw (1,0)  node [above right] {1};
\filldraw (1,0) circle ({0.6*1.5/1.5pt});
%\draw (0,1.2) node [above] {SRG of $P_{L}$};
\end{tikzpicture}
}
\end{subfigure}
\begin{subfigure}[b]{0.28\textwidth}
\raisebox{-.5\height}{
\begin{tikzpicture}[scale=1.5]
\draw(-0.85,0.3) node {$\cG(A)=$};
\draw [<->] (-.8,0) -- (.8,0);
\draw [<->] (0,-1.2) -- (0,1.2);
%\filldraw (0,1/2) circle ({0.6*1.5/1.5pt});
%\draw (0,1/2)   node [left] {$1/2$};
\filldraw (0,1) circle ({0.6*1.5/1.5pt});
\draw (0,1) node [above left] {$i$};
%\draw (-1pt,1) -- (1pt,1)   node [above left] {$1$};
%\draw (1.25,0) node [below] {$\frac{1}{2}$};
%\draw [decorate,decoration={brace,amplitude=4.5pt}](1.5,0) -- (1,0) ;
\draw [line width=1.0pt] (0,1/2) circle (1/2);
\draw [line width=1.0pt] (0,-1/2) circle (1/2);
\end{tikzpicture}
}
\end{subfigure}
\begin{subfigure}[b]{0.38\textwidth}
\raisebox{-.5\height}{
\begin{tikzpicture}[scale=1.5]
\draw(-0.6,0.4) node {$\cG(\partial \|\cdot\|)=$};
\draw(0.2,-1.1) node {$\{z \,|\,\Re z>0\}\cup\{0,\infty\}$};
\def\m{.9}
\fill [fill=medgrey] (0,{-\m})--(0,{\m})--({\m},{\m})--({\m},{-\m})--cycle;
\draw [<->] (-0.2,0) -- ({\m},0);
\draw [dashed] (0,{-\m}) -- (0,{\m});
\draw ({\m-0.4},{\m-0.2}) node {$\cup \{\binfty\}$};
\filldraw (0,0) circle ({0.6*1.5/1.5pt});
%\draw (-1pt,1) -- (1pt,1)   node [above left] {$1$};
%\draw (1.25,0) node [below] {$\frac{1}{2}$};
%\draw [decorate,decoration={brace,amplitude=4.5pt}](1.5,0) -- (1,0) ;
\end{tikzpicture}
}
\end{subfigure}

\begin{subfigure}[b]{0.45\textwidth}
\raisebox{-.5\height}{
\begin{tikzpicture}[scale=1.5]
\draw(.6,0.4) node {$\cG(B)=$};
\fill [fill=medgrey] (2,0) circle (1);
\fill [fill=white] (1.5,0) circle (.5);
\fill [fill=white] (2.5,0) circle (.5);

\filldraw (1,0) circle ({0.6*1.5/1.5pt});
\draw (1,0)   node [above left] {$1$};
\filldraw (2,0) circle ({0.6*1.5/1.5pt});
\draw (2,0)   node [above left] {$2$};
\filldraw (3,0) circle ({0.6*1.5/1.5pt});
\draw (3,0)   node [above right] {$3$};
%\draw (1.25,0) node [below] {$\frac{1}{2}$};
%\draw [decorate,decoration={brace,amplitude=4.5pt}](1.5,0) -- (1,0) ;
\draw [<->] (-0.2,0) -- (3.2,0);
\draw [<->] (0,-1.) -- (0,1.);
\end{tikzpicture}}
\end{subfigure}
% \begin{subfigure}[b]{0.45\textwidth}
% \raisebox{-.5\height}{
% \begin{tikzpicture}[scale=1.5]
% \draw(-0.8,0) node {$\cG(\nabla f)=$};
% \draw(0.3,-1) node {$\{z\,|\,\Re z>0,\,|\Im z|\le 0.354\Re z\}$};
% \def\m{1.7}
% \fill[fill=medgrey] (0,0)--(\m,{\m*tan(19.4712)})--(\m,{-\m*tan(19.4712)})--cycle;
% \filldraw (1,0.353553) circle ({0.6*1.5/1.5pt});
% \draw (1,0.353553) node[above left] {$(1,0.354)$};
% \draw [<->] (-0.2,0) -- (\m,0);
% \draw [<->] (0,-.8) -- (0,.8);
% %\draw (0,1/2)   node [left] {$1/2$};
% \end{tikzpicture}}
% \end{subfigure}
\caption{
SRGs of the operators:
$P_L\colon\reals^2 \to \reals^2$ is the projection onto an arbitrary line $L$;
$A\colon\mathbb{R}^2 \to \mathbb{R}^2$ is defined as $A(u,v) = (0,u)$;
 $\partial \|\cdot\|$ is the subdifferential of the Euclidean norm on $\reals^n$ with $n\ge 2$;
 $B\colon\mathbb{R}^3\to\mathbb{R}^3$ is defined as $B(u,v,w) = (u,2v,3w)$.
 The shapes were obtained by plugging the operators into the definition of the SRG and performing direct calculations.
 See \cite{huang2019scaled,huang2019matrix} for follow-up work on drawing the SRG of individual operators.
}
\label{fig:srg-exmaples}
\end{figure}
\subsection{SRG of operators}
\label{ss:srg-operator}
Consider an operator $A\colon\hilbert\rightrightarrows\hilbert$.
Let $x,y\in\hilbert$ be a pair of inputs and let $u,v\in \hilbert$ be their corresponding outputs, i.e., $u\in Ax$, and $v\in Ay$.
The goal is to understand the change in output relative to the change in input.

First, consider the case $x\ne y$.
Consider the complex conjugate pair
\[
z=
\frac{\|u-v\|}{\|x-y\|}
\exp\left[\pm i \angle (u-v,x-y)\right],
\]
where given any $a,b\in\hilbert$
\begin{align*}
    \angle(a,b) =
    \left\{
    \begin{array}{ll}
    \arccos\p{\tfrac{\dotp{a,b}}{\n{a}\n{b}}}&\text{ if }a\ne 0,\,b\ne 0\\
    0&\text{ otherwise}
    \end{array}
    \right.
\end{align*}
denotes the angle between them.
 %(note the $\pm$) captures how large $u-v$ is relative to $x-y$ and how aligned $u-v$ is relative to $x-y$.
The absolute value (magnitude) $|z|=\tfrac{\|u-v\|}{\|x-y\|}$ represents
the size of the change in outputs relative to the size of the change in inputs.
The argument (angle) $\angle (u-v,x-y)$ represents how much the change in outputs is aligned with the change in inputs.
Equivalently, $\Re z$ and $\Im z$ respectively represent the
components of $u-v$ aligned with and perpendicular to $x-y$, i.e.,
\begin{align}
\Re z
&=
\textrm{sgn}(\langle u-v,x-y\rangle)\frac{\|P_{\mathrm{span}\{x-y\}}(u-v)\|}{\|x-y\|}
=\frac{\langle u-v,x-y\rangle}{\|x-y\|^2}\nonumber\\
\qquad
\Im z
&=\pm
\frac{\|P_{\{x-y\}^\perp}(u-v)\|}{\|x-y\|}
%=\pm\frac{1}{\|x-y\|}\sqrt{\|u-v\|^2-\frac{(\langle u-v,x-y\rangle)^2}{\|x-y\|^2}},
\label{eq:srg-alternative2}
\end{align}
where $P_{\mathrm{span}\{x-y\}}$ is the projection onto the span of $x-y$
and $P_{\{x-y\}^\perp}$ is the projection onto the subspace orthogonal to $x-y$.

Define the SRG of an operator $A\colon\hilbert\rightrightarrows\hilbert$ as
\begin{align*}
\mathcal{G}(A)&=
\left\{
\frac{\|u-v\|}{\|x-y\|}
\exp\left[\pm i \angle (u-v,x-y)\right]
\,\Big|\,
u\in Ax,\,v\in Ay,\, x\ne y \right\}\\
&\qquad\qquad\qquad\qquad\qquad\qquad\qquad\qquad\qquad
\bigg(\cup \{\infty\}\text{ if $A$ is multi-valued}\bigg).
\end{align*}
We clarify several points:
(i) $\cG(A)\subseteq \ecomplex$.
(ii) $\infty\in \cG(A)$ if and only if there is a point $x\in\hilbert$ such that $Ax$ is multi-valued.
(In this case, there exists $(x,y),(u,v)\in A$ such that $x=y$ and $u\ne v$, and the idea is that $|z|=\|u-v\|/0=\infty$, i.e.,  $u-v$ is infinitely larger than $x-y=0$.)
(iii) the $\pm$ makes $\cG(A)$ symmetric about the real axis.
(We include the $\pm$ because $\angle (u-v,x-y)$ always returns a nonnegative angle.)
See Figure~\ref{fig:srg-exmaples} for examples.

%If $x=y$ and $u=v$, then $(x,u)$ and $(y,v)$ represents the same evaluation and there is no change of input or output to consider.

%\begin{wrapfigure}[16]{r}{0.38\textwidth}
\begin{figure}
%\vspace{-0.4in}
\begin{center}
\begin{tikzpicture}[scale=1.05]
\def\a{1/2};
\def\b{2};
\def\c{1/2};
\fill [fill=medgrey] ({\a},{(\b+\c)/2}) circle ({(\b-\c)/2});

\def\x{0.896284298022351};
\def\y{0.613243566863332};
\def\u{0.852941212550304};
\def\v{1.911764686639837};

\fill[fill = medgrey] plot[line width=1.0pt, domain=0:{acos((-4 + \y^2 + 4 *\x - (\x)^2)/(4 + \y^2 - 4 *(\x) + (\x)^2))},variable=\p] ({( 2-(-4 + (\y)^2 + (\x)^2)/(2 *(-2 + (\x))) )*cos(\p)+(-4 + (\y)^2 + (\x)^2)/(2 *(-2 + (\x)))},{( 2-(-4 + (\y)^2 + (\x)^2)/(2 *(-2 + (\x))) )*sin(\p)}) --  plot[line width=1.0pt, domain={acos((-4 + \v^2 + 4 *\u - (\u)^2)/(4 + \v^2 - 4 *(\u) + (\u)^2)):0},variable=\p] ({( 2-(-4 + (\v)^2 + (\u)^2)/(2 *(-2 + (\u))) )*cos(\p)+(-4 + (\v)^2 + (\u)^2)/(2 *(-2 + (\u)))},{( 2-(-4 + (\v)^2 + (\u)^2)/(2 *(-2 + (\u))) )*sin(\p)});

\fill [fill=medgrey] ({\a},{-(\b+\c)/2}) circle ({(\b-\c)/2});

\def\x{0.896284298022351};
\def\y{0.613243566863332};
\def\u{0.852941212550304};
\def\v{1.911764686639837};

\fill[fill = medgrey] plot[line width=1.0pt, domain=0:{acos((-4 + \y^2 + 4 *\x - (\x)^2)/(4 + \y^2 - 4 *(\x) + (\x)^2))},variable=\p] ({( 2-(-4 + (\y)^2 + (\x)^2)/(2 *(-2 + (\x))) )*cos(\p)+(-4 + (\y)^2 + (\x)^2)/(2 *(-2 + (\x)))},{-( 2-(-4 + (\y)^2 + (\x)^2)/(2 *(-2 + (\x))) )*sin(\p)}) --  plot[line width=1.0pt, domain={acos((-4 + \v^2 + 4 *\u - (\u)^2)/(4 + \v^2 - 4 *(\u) + (\u)^2)):0},variable=\p] ({( 2-(-4 + (\v)^2 + (\u)^2)/(2 *(-2 + (\u))) )*cos(\p)+(-4 + (\v)^2 + (\u)^2)/(2 *(-2 + (\u)))},{-( 2-(-4 + (\v)^2 + (\u)^2)/(2 *(-2 + (\u))) )*sin(\p)});

\draw (0.15,1.05) node[above right]{$1/2+i$};
\draw (0.15,-1.05) node[below right]{$1/2-i$};
\draw (2,0) node[above right]{$2$};
\draw (1/2,1) node[cross, line width=1.5pt]{};
\draw (1/2,-1) node[cross, line width=1.5pt]{};
\draw (2,0) node[cross, line width=1.5pt] {};

\draw (3.25,1) node[]{$=\cG\left(\begin{bmatrix}
1/2&2&0\\
-1/2&1/2&0\\
0&0&2
\end{bmatrix}\right)$};

\draw [<->] (-0.4,0) -- (2.4,0);
\draw [<->] (0,-2.1) -- (0,2.1);
\end{tikzpicture}
\end{center}
\vspace{-0.15in}
%\vspace{-0.3in}
    \captionsetup{singlelinecheck=off}
%$[1/2,2,0;-1/2,1/2,0;0,0,2]$.
\caption[.]{SRG of a $3\times 3$ matrix. Crosses denote the eigenvalues.
}
\label{fig:matrix_SRG}
\end{figure}
%\end{wrapfigure}
%\subsection{Spectrum $\subseteq$ SRG}
For linear operators, the SRG generalizes eigenvalues.
% In this section, we quickly point out that
% the SRG generalizes eigenvalues and
% that the sign of the real part of the SRG implies stability of a non-linear dynamical system.
%analogous to how the sign of the real part of the eigenvalues imply stability of linear dynamical systems.
Given  $A\in \reals^{n\times n}$, write $\Lambda(A)$ for the set of eigenvalues of $A$.
%With some abuse of notation, we identify $A$ with the linear operator from $\reals^n$ to $\reals^n$.

\begin{theorem}
\label{thm:eigen-srg}
If $A\in \reals^{n\times n}$ and $n = 1$ or $n\ge 3$, then $\Lambda(A)\subseteq \cG(A)$.
\end{theorem}
The result fails for $n=2$ because $S^{n-1}$, the sphere in $\reals^n$, is not simply connected for $n=2$; the proof constructs a loop in $S^{n-1}$ and argues the image of the loop on complex plane is nullhomotopic.
Figure~\ref{fig:matrix_SRG} illustrates an SRG of a matrix.
The SRG of a matrix does not seem to be directly related to the numerical range%
\footnote{After the publication of the paper, Pates \cite{pates2021} established a connection of the SRG with the numerical range.}
(field of values) \cite{horn_johnson_1991} or the pseudospectrum \cite{trefethen2005spectra}.

\begin{proof}
%[Proof of Theorem~\ref{thm:eigen-srg}]
If $\lambda$ is a real eigenvalue of $A$, then
considering \eqref{eq:srg-alternative2} with $x$ as the corresponding eigenvector and $y=0$ tells us $\lambda\in \cG(A)$.

Next consider a complex conjugate eigenvalue pair $\lambda,\overline{\lambda}\in \Lambda(A)$, where $\Im\lambda>0$.
(This case excludes $n=1$.)
$A$ has a real Schur decomposition of the form
\[
A=Q^T
\underbrace{\begin{bmatrix}
B_{11}&A_{12}&A_{13}&\cdots\\
0&B_{22}&A_{23}&\cdots\\
&&\ddots\\
\end{bmatrix}}_{=B}
Q,\qquad
B_{11}=\begin{bmatrix}
a&b\\
-c&a
\end{bmatrix}
\in \reals^{2\times 2},
\]
where $b,c>0$, $\lambda=a+i\sqrt{bc}$, and $Q\in \reals^{n\times n}$ is orthogonal.
(To obtain this decomposition, take the construction of \cite{Murnaghan417} and apply a $\pm 45$ degree rotation to the leading $2\times 2$ block.) Since an orthogonal change of coordinates does not change the SRG, we have $\cG(A)=\cG(B)$. 
Write $S^{n-1}$ for the sphere in $\reals^n$.
Consider the continuous map
$z:S^{n-1}\rightarrow\complex$ defined by
$z(x)=\|Bx\|\exp\left[i\angle (Bx,x)\right]$.
Since $B$ is a linear operator, we have $z(S^{n-1})=\cG(B)$.
Consider the curve
$\gamma(t)=\cos(2\pi t)e_1+\sin(2\pi t)e_2\in S^{n-1}$ from $t\in [0,1]$,
where $e_1$ and $e_2$ are the first and second unit vectors in $\reals^n$.
With simple computation, we get
\[
z(\gamma(t))=
a+\frac{1}{2}(b-c)\sin(4\pi t)+ i
\frac{1}{2}(b+c-(b-c)\cos(4\pi t)).
\]
If $b=c$, then $z(\gamma(t))=\lambda$ and we conclude $\lambda\in z(S^{n-1})=\cG(A)$.

Assume $b\ne c$, and assume for contradiction that $\lambda\notin z(S^{n-1})=\cG(A)$.
The curve $z(\gamma(t))$
strictly encloses the eigenvalue $\lambda=a+i\sqrt{bc}$ since $\min(b,c)\leq\sqrt{bc}\leq\max(b,c)$.
Since $S^{n-1}$ is simply connected for $n\ge 3$,
we can continuously contract $\gamma(t)$ to a point in $S^{n-1}$,
and the continuous map $z$ provides a continuous contraction of $z(\gamma(t))$ to a point in $z(S^{n-1})$.
However, $z( \gamma(t))$ has a nonzero winding number\footnote{The winding number of a closed curve in the plane around a given point is an integer representing the total number of times that curve travels counterclockwise around the point.
A definition, based on the complex analysis and the Cauchy residue theorem can be found in Section 4.1 of \cite{abramowitz_stegun}.
} around $\lambda$ and $\lambda\notin z(S^{n-1})$.
Therefore, $z( \gamma(t))$ cannot be continuously contracted to a point in $z(S^{n-1})$.
This is a contradiction and we conclude 
$\lambda\in z(S^{n-1})=\cG(A)$.
\qed\end{proof}

The SRG $\cG(A)$ maps the action of the operator $A$ to points in $\ecomplex$.
In the following sections, we will need to conversely take any point in $\ecomplex$ and find an operator whose action maps to that point.
Lemma~\ref{lem:complex-srg} provides such constructions.
%the construction of operators that correspond a given point in $\ecomplex$.
%We omit the proofs as they follow from definitions and basic computation.
\begin{lemma}
\label{lem:complex-srg}
Take any $z=z_r+z_ii\in \complex$. 
Define
$A_z\colon\reals^2\rightarrow\reals^2$
and $A_\infty\colon\reals^2\rightrightarrows\reals^2$ as
\[
A_z\begin{bmatrix}
\zeta_1\\\zeta_2
\end{bmatrix}
=
\begin{bmatrix}
z_r\zeta_1-z_i\zeta_2\\
z_r\zeta_2+z_i\zeta_1
\end{bmatrix}
\qquad
A_\infty(x)=\left\{
\begin{array}{ll}
\reals^2&\text{if }x=0\\
\emptyset&\text{otherwise.}
\end{array}
\right.
\]
Then, 
\[
\mathcal{G}(A_z)=\{z,\bar{z}\},
\qquad
\mathcal{G}(A_\infty)=\{\infty\}.
\]
% Furthermore, 
% \[
% \frac{\|A_zx-A_zy\|^2}{\|x-y\|^2}=|z|^2,\quad
% \frac{\langle A_zx-A_zy,x-y\rangle}{\|x-y\|^2} =\Re z,\quad
% \forall\,x,y\in \reals^2,\,x\ne y.
% \]
%\label{lem:infinite-srg}
%Consider $\infty\in \ecomplex$.
% Furthermore,
% \[
% \|x-y\|^2=0,\quad
% \langle u-v,x-y\rangle =0,\quad
% \forall\,(x,u),(y,v)\in A.
% \]
\end{lemma}
If we write $\cong$ to identify an element of $\mathbb{R}^2$ with an element in $\mathbb{C}$ in that
\[
\begin{bmatrix}
x\\y
\end{bmatrix}
\cong x+yi,
\]
then we can view $A_z$ as complex multiplication with $z$ in the sense that
\[
A_z\begin{bmatrix}
\zeta_1\\\zeta_2
\end{bmatrix}
\cong z(\zeta_1+\zeta_2i).
\]
%If we identify $\math \cong z(\zeta_1+\zeta_2i),$\cong$ identifies $\reals^2$ with $\complex$, so $A_z$ corresponds to complex multiplication by $z$.)
\begin{proof}
%[Proof of Lemma~\ref{lem:complex-srg}
Again, we write $\cong$ to identify an element of $\mathbb{R}^2$ with an element in $\mathbb{C}$.
Write $z=r_ze^{i\theta_{z}}$.
Consider any $x,y\in \reals^2$ where $x\ne y$ and define $u=A_zx$ and $v=A_zy$.
Then we can write
\[
x-y=r_w
\begin{bmatrix}
\cos(\theta_w)\\
\sin(\theta_w)
\end{bmatrix}
\]
where $r_w>0$, and
\[
u-v=A_z(x-y)\cong r_{z}r_we^{i(\theta_{z}+\theta_w)}.
\]
This gives us
\[
\frac{\|u-v\|}{\|x-y\|}= r_z,\quad
% \frac{1}{\|x-y\|}\langle u-v,x-y\rangle=\Re r_ze^{i\theta_{z}}. XXX
% \quad
\angle (u-v,x-y)=|\theta_z|,
\]
and
\[
\cG(A_z)=\left\{r_{z}e^{i\theta_{z}},r_{z}e^{- i\theta_{z}}\right\}.
\]
% Now we let 
% $x-y=r_we^{i\theta_w}$, and $u-v=r_{z}r_we^{i(\theta_{z}+\theta_w)}$, (note that $x+yi$ is a linear operator) we get

Now consider $A_\infty$.
By definition, $\infty\in \cG(A_\infty)$.
For any $u\in A_\infty x$ and $v\in A_\infty y$, we have $x=y=0$,
and therefore $\cG(A_\infty)$ contains no finite $z\in \complex$.
We conclude $\cG(A_\infty)=\{\infty\}$.
%Since $\dom (A_\infty)=\{0\}$, 
\qed\end{proof}

% \begin{proof}
% Write $z=r_ze^{i\theta_{z}}$.
% Consider any $x,y\in \reals^2$ where $x\ne y$ and define $u=A_zx$ and $v=A_zy$.
% Then we can write
% \[
% x-y=r_w
% \begin{bmatrix}
% \cos(\theta_w)\\
% \sin(\theta_w)
% \end{bmatrix}
% \]
% where $r_w>0$.
% Then 
% \[
% u-v=A_z(x-y)\cong r_{z}r_we^{i(\theta_{z}+\theta_w)}.
% \]
% Now we have
% \[
% \frac{\|u-v\|}{\|x-y\|}= r_z,\quad
% \frac{1}{\|x-y\|}\langle u-v,x-y\rangle=\Re r_ze^{i\theta_{z}}.
% \quad
% \angle (u-v,x-y)=|\theta_z|.
% \]
% This gives us the SRG
% \[
% \cG(A_z)=\left\{r_{z}e^{i\theta_{z}},r_{z}e^{- i\theta_{z}}\right\}.
% \]
% % Now we let 
% % $x-y=r_we^{i\theta_w}$, and $u-v=r_{z}r_we^{i(\theta_{z}+\theta_w)}$, (note that $x+yi$ is a linear operator) we get
% \qed\end{proof}

\subsection{SRG of operator classes}
\label{ss:srg-class}
Let $\cA$ be a collection of operators. We define the SRG of the class $\cA$ as
\[
\mathcal{G}(\cA)=\bigcup_{A\in \cA}\mathcal{G}(A).
\]
%In later sections, we will use the SRG of operator classes, rather than the SRG of individual operators, more directly, because 
We focus more on SRGs of operator classes, rather than individual operators, because theorems are usually stated with operator classes.
For example, one might say ``If $A$ is $1/2$-cocoercive, i.e., if $A\in \cC_{1/2}$, then $I-A$ is nonexpansive.''
We now characterize the SRG of the Lipschitz, averaged, monotone, strongly monotone, and cocoercive operator classes.

\begin{proposition}
\label{prop:monotone-srg}
Let $\mu,\beta,L\in (0,\infty)$ and $\theta\in(0,1)$.
The SRGs of $\cL_L$, $\cN_\theta$, $\cM$, $\cM_\mu$, and $\cC_\beta$ are, respectively, given by
\begin{center}
\begin{tabular}{cc}
\raisebox{-.5\height}{
\begin{tikzpicture}[scale=1.5]
\draw (-1.1,0.5) node {$\cG(\cL_L)=$};
\fill[fill=medgrey] (0,0) circle (0.75);
%\draw[line width=1.0pt] (0,0) circle (0.75);
\draw [<->] (-1,0) -- (1,0);
\draw [<->] (0,-1) -- (0,1);
\draw (0.75,0) node [above right] {$L$};
\draw (-0.75,0) node [above left] {$-L$};
\filldraw (.75,0) circle ({0.6*1.5/1.5pt});
\filldraw (-.75,0) circle ({0.6*1.5/1.5pt});
\draw (0,-1.2) node  {$\left\{z\in \complex\,\big|\,|z|^2\le L^2\right\}$};
\end{tikzpicture}}
&
\raisebox{-.5\height}{
\begin{tikzpicture}[scale=1.5]
\draw (-0.85,0.5) node {$\cG(\cN_\theta)=$};
\fill[fill=medgrey] (0.25,0) circle (0.75);
\draw [<->] (-.8,0) -- (1.2,0);
\draw [<->] (0,-1) -- (0,1);
%\draw [dashed] (0,0) circle (1);
\draw (1,-0) node [above right] {$1$};
\draw (-0.5,-0) node [below] {$1-2\theta$};
\filldraw (1,0) circle ({0.6*1.5/1.5pt});
\filldraw (.25,0) circle ({0.6*1.5/1.5pt});
\filldraw (-.5,0) circle ({0.6*1.5/1.5pt});
\draw (0.625,0.1) node [above] {$\theta$};
\draw [decorate,decoration={brace,amplitude=4.5pt}] (0.25,0) -- (1,0) ;
\draw (0.4,-1.2) node  {$\left\{z\in \complex\,\big|\,|z|^2+(1-2\theta)\le 2(1-\theta)\Re z\right\}$};
\end{tikzpicture}}
\end{tabular}
\end{center}
\begin{center}
\begin{tabular}{ccc}
\raisebox{-.5\height}{
\begin{tikzpicture}[scale=1.5]
\draw (-0.5,0.5) node {$\cG(\cM)=$};
\fill[fill=medgrey] (0,-1) rectangle (1,1);
\draw [<->] (-0.2,0) -- (1,0);
\draw [<->] (0,-1) -- (0,1);
%\draw [line width=1.0pt] (0,-1.2) -- (0,1.2);
%\draw (1,-1pt) -- (1,1pt)   node [above] {$1$};
\draw (0.7,.85) node {$\cup \{\binfty\}$};
\draw (0.4,-1.2) node {$\left\{z\in \complex\,|\,\Re z\ge 0\right\}\cup\{\infty\}$};
\end{tikzpicture}}
&
\!\!\!\!\!\!
\raisebox{-.5\height}{
\begin{tikzpicture}[scale=1.5]
\draw [<->] (0,-1) -- (0,1);

\draw (-0.2,0.5) node [fill=white] {$\cG(\cM_\mu)=$};
\fill[fill=medgrey] (0.3,-1) rectangle (1,1);

\draw [<->] (-0.2,0) -- (1,0);

%\draw [line width=1.0pt] (0.5,-1.2) -- (0.5,1.2);
%\draw (1,-1pt) -- (1,1pt)   node [above] {$1$};
\draw (0.7,.85) node {$\cup \{\binfty\}$};
\filldraw (.3,0) circle ({0.6*1.5/1.5pt});
\draw (0.3,0) node [above left] {$\mu$};
\draw (0.4,-1.2) node {$\left\{z\in \complex\,|\,\Re z\ge \mu\right\}\cup\{\infty\}$};

\end{tikzpicture}}
&
\!\!\!\!\!\!
\raisebox{-.5\height}{
\begin{tikzpicture}[scale=1.5]
\draw (-0.4,0.3) node {$\cG(\cC_\beta)=$};
\fill[fill=medgrey] (0.6,0) circle (0.6);
%\draw[line width=1.0pt] (0.5,0) circle (0.5);
\draw [<->] (-0.5,0) -- (1.5,0);
\draw [<->] (0,-.8) -- (0,.8);
\filldraw (1.2,0) circle ({0.6*1.5/1.5pt});
\draw (1.2,0) node [above right] {\!$1/\beta$};
\draw (0.4,-1) node  {$\left\{z\in \complex\,\big|\,\Re z\ge \beta|z|^2\right\}$};
\end{tikzpicture}}
\end{tabular}
\end{center}
\end{proposition}
\begin{proof}
%We prove the characterizations of $\cG(\cL_L)$ and $\cG(\cM)$.
%Although we can characterize  $\cG(\cM_\mu)$, $\cG(\cC_\beta)$, and $\cG(\cN_\theta)$ with similar direct proofs, we provide simpler proofs later in \S\ref{ss:srg-classes-proofs} that use operator and SRG transformations.
First, characterize $\cG(\cL_L)$.
We have $\cG(\cL_L)\subseteq\left\{z\in \complex\,\big|\,|z|^2\le L^2\right\}$ since
\[
A\in \cL_L
\;\;\Rightarrow\;\;
\frac{\|Ax-Ay\|}{\|x-y\|}\le L,\,\forall\,x,y\in \hilbert,\,x\ne y
\;\;\Rightarrow\;\;
\cG(A)\subseteq\left\{z\in \complex\,\big|\,|z|^2\le L^2\right\}.
\]
Conversely, given any $z\in \complex$ such that $|z|\le L$,
the operator $A_z$ of Lemma~\ref{lem:complex-srg} satisfies $\| A_zx-A_zy\|\le L\|x-y\|$ for any $x,y\in \reals^2$, i.e., $A_z\in \cL_L$, and $\cG(A_z)=\{z,\bar{z}\}$. Therefore
$\cG(\cL_L)\supseteq \left\{z\in \complex\,\big|\,|z|^2\le L^2\right\}$.% and we have equality.

% The proof has 2 parts: that the SRG is contained in the region and that the SRG fills the entire region.
% Since the left-hand side and the right-hand wide both include $\infty$, we only need to consider the finite points.
%First, we show $\cG(\cM)\backslash\{\infty\}\subseteq \{z\,|\,\Re z\ge 0\}$.
Next, characterize $\cG(\cM)$.
For any $A\in \cM$, monotonicity implies
\[
\frac{\langle u-v,x-y\rangle}{\|x-y\|^2}\ge 0,
\quad\forall\,u\in Ax,\,v\in Ay,\,x\ne y.
\]
Considering \eqref{eq:srg-alternative2}, we conclude $\cG(A)\backslash\{\infty\}\subseteq \{z\,|\,\Re z\ge 0\}$.
%Next, we show $\{z\,|\,\Re z\ge 0\}\subseteq \cG(\mathcal{M})$.
On the other hand, given any $z\in \{z\,|\,\Re z\ge 0\}$, the operator $A_z$ of Lemma~\ref{lem:complex-srg} satisfies $\langle A_zx-A_zy,x-y\rangle \ge 0$ for any $x,y\in \reals^2$, i.e., $A_z\in \cM$, and $\cG(A_z)=\{z,\bar{z}\}$.
Therefore, $z\in\cG(A_z)\subset \cG(\cM)$, and we conclude $\{z\,|\,\Re z\ge 0\}\subseteq \cG(\mathcal{M})$.
Finally, note that $\infty\in \cG(\cM)$ is equivalent to saying that there exists a multi-valued operator in $\cM$. The $A_\infty$ of Lemma~\ref{lem:complex-srg} is one such example.
%This proves equality.

The other SRGs $\cG(\cM_\mu)$, $\cG(\cC_\beta)$, and $\cG(\cN_\theta)$ can be characterized with similar direct proofs or by using operator and SRG transformations introduced later in \S\ref{s:srg-transformation}. In particular:
the fact $\cM_\mu=\mu I+\cM$, Theorem~\ref{thm:srg-scaling-translation}, and the characterization of $\cG(\cM)$ prove the characterization $\cG(\cM_\mu)$;
the fact $(\cM_\mu)^{-1}=\cC_{\mu}$, Theorem~\ref{thm:srg-inversion}, and the characterization $\cG(\cM_\mu)$ prove the characterization $\cG(\cC_{\mu})$; and 
the fact $(1-\theta)I+\theta\cL_1=\cN_\theta$, Theorem~\ref{thm:srg-scaling-translation}, and the characterization of $\cG(\cL_1)$ prove the characterization of $\cG(\cN_\theta)$.
Facts $\cM_\mu=\mu I+\cM$, $(\cM_\mu)^{-1}=\cC_{\mu}$, and $(1-\theta)I+\theta\cL_1=\cN_\theta$ are well known \cite{BCBook}.
\qed\end{proof}

\begin{proposition}
\label{proposition:cvx-srg-first}
Let $0<\mu<L<\infty$. Then
\vspace{0.05in}
%Let $\partial \mathcal{F}_{0,\infty}$ be the class of subdifferential operators of CCP functions. Then
\begin{center}
\begin{tabular}{cccc}
\!\!\!\!\!\!
\raisebox{-.5\height}{
\begin{tikzpicture}[scale=1]
\draw (-0.9,0.4) node {$\cG(\partial \mathcal{F}_{0,\infty})=$};
\fill[fill=medgrey] (0,-1) rectangle (1,1);
\draw [<->] (-0.5,0) -- (1,0);
\draw [<->] (0,-1) -- (0,1);
%\draw [line width=1.0pt] (0,-1.2) -- (0,1.2);
%\draw (1,-1pt) -- (1,1pt)   node [above] {$1$};
\draw (0.65,.85) node {$\scriptstyle\cup \{\binfty\}$};
\draw (-0.2,-1.2) node {$\left\{z\,|\,\Re z\ge 0\right\}\cup\{\infty\}$};
\end{tikzpicture}}
&\!\!\!\!\!\!\!\!\!\!\!\!\!\!\!
\raisebox{-.5\height}{
\begin{tikzpicture}[scale=1]
\draw [<->] (0,-1) -- (0,1);
\draw (-0.7,-0.4)  node [fill=white]{$\cG(\partial \mathcal{F}_{\mu,\infty})=$};
\fill[fill=medgrey] (0.2,-1) rectangle (1,1);
\draw [<->] (-0.5,0) -- (1,0);
%\draw [line width=1.0pt] (0.5,-1.2) -- (0.5,1.2);
%\draw (1,-1pt) -- (1,1pt)   node [above] {$1$};
\draw (0.65,.85) node {$\scriptstyle \cup \{\binfty\}$};
\draw (0.2,0pt) node [above right] {$\mu$};
\filldraw (.2,0) circle ({0.6*1.5/1pt});
\draw (0.2,-1.2) node{$\left\{z\,|\,\Re z\ge \mu\right\}\cup\{\infty\}$};
\end{tikzpicture}}
&\!\!\!\!\!\!\!\!\!\!\!\!\!\!\!\!\!\!
\raisebox{-.5\height}{
\begin{tikzpicture}[scale=1]
\draw (-0.9,0.4) node {$\cG(\partial \mathcal{F}_{0,L})=$};
\fill[fill=medgrey] (0.5,0) circle (0.5);
%\draw[line width=1.0pt] (0.5,0) circle (0.5);
\draw [<->] (-0.5,0) -- (1.2,0);
\draw [<->] (0,-.8) -- (0,.8);
\draw (1,0) node [above right] {$L$};
\filldraw (1,0) circle ({0.6*1.5/1pt});
%\draw (0.,-1) node {$\left\{z\,|\,|z|^2\le L\Re z\right\}$};
\end{tikzpicture}}
&\!\!\!\!\!\!\!\!\!\!\!\!
\raisebox{-.5\height}{
\begin{tikzpicture}[scale=1]
\draw [<->] (0,-.8) -- (0,.8);
\draw (-0.5,-0.3)  node [fill=white] {$\cG(\partial \mathcal{F}_{\mu,L})=$};
\fill[fill=medgrey] (0.65,0) circle (0.35);
%\draw[line width=1.0pt] (0.65,0) circle (0.35);
\draw [<->] (-0.5,0)-- (1.2,0);
\filldraw (1,0) circle ({0.6*1.5/1pt});
\draw (1,0pt) node [above right] {$L$};
\filldraw (.3,0) circle ({0.6*1.5/1pt});
\draw (0.4,0pt) node [above left] {$\mu$};
%\draw (0.,-1) node {$\left\{z \,|\,|z|^2+\mu L\le(L+\mu)\Re z\right\}$};
\end{tikzpicture}
}
\end{tabular}
\end{center}
\vspace{0.05in}
\end{proposition}
\begin{proof}
%The proof has 2 parts: that the SRG is contained in the region and that the SRG fills the entire region.
Since $\partial \mathcal{F}_{0,\infty}\subset \mathcal{M}$, we  have $\cG(\partial \mathcal{F}_{0,\infty})\subseteq\cG(\cM)= \{z\in\complex\,|\,\Re z\ge 0\}\cup\{\infty\}$ by Proposition~\ref{prop:monotone-srg}.
We claim $f\colon\reals^2\rightarrow\reals$ defined by $f(x,y)=|x|$ satisfies $\cG(\partial f)=\{z\in\complex\,|\,\Re z\ge 0\}\cup\{\infty\}$. This tells us $\{z\in\complex\,|\,\Re z\ge 0\}\cup\{\infty\}\subseteq\cG(\partial \mathcal{F}_{0,\infty})$.

We prove the claim with basic computation.
Let $f(x,y)=|x|$.
The subgradient has the form $\partial f(x,y)=(h(x),0)$ for $h$ defined by:
\[
h(x) =
\left\{
\begin{array}{ll}
\{-1\}&\text{for }x<0\\
\{u\,|\,-1\le u\le 1\}&\text{for }x=0\\
\{1\}&\text{for }x>0.
\end{array}\right.
\]
Since $\partial f$ is multi-valued at $(0,0)$, we have $\infty\in \cG(\partial f)$. 
Since $\partial f(1,0)=\partial f(2,0)$, we have $0\in \cG(\partial f)$. 
The input-output pairs
$(0,0)\in \partial f(0,0)$ and $(h(R\cos(\theta)),0)\in\partial f(R\cos(\theta),R\sin(\theta))$ map to the point $R^{-1}(\abs{\cos(\theta)},\pm\sin(\theta))\in\CC$. Clearly the image of this map over the range $R\in(0,\infty)$, $\theta\in[0,2\pi)$ is the right-hand plane except the origin. Hence $\cG(\partial f)=\{z\in\complex\,|\,\Re z\ge 0\}\cup\{\infty\}$.

The SRGs $\cG(\partial \cF_{\mu,\infty})$, $\cG(\partial \cF_{0,L})$, and $\cG(\partial \cF_{\mu,L})$ can be characterized with similar direct proofs or by using operator and SRG transformations introduced later in \S\ref{s:srg-transformation}. In particular:
the fact $\partial \cF_{\mu,\infty}=\mu I+\partial \cF_{0,\infty}$, Theorem~\ref{thm:srg-scaling-translation}, and the characterization of $\cG(\partial \cF_{0,\infty})$ prove the characterization of $\cG(\partial \cF_{\mu,\infty})$;
the fact $\partial \cF_{0,L}=\left(\partial \cF_{1/L,\infty}\right)^{-1}$, Theorem \ref{thm:srg-inversion}, and the characterization of $\cG(\partial \cF_{1/L,\infty})$ prove the characterization of $\cG(\partial \cF_{0,L})$; and
the fact $\partial \cF_{\mu,L}= \mu I+\partial \cF_{0,L-\mu}$, Theorem~\ref{thm:srg-scaling-translation}, and the characterization of $\cG(\partial \cF_{0,L-\mu})$ prove the characterization of $\cG(\partial \cF_{\mu,L})$.
Facts $\partial \cF_{\mu,\infty}=\mu I+\partial \cF_{0,\infty}$, $\partial \cF_{0,L}=\left(\partial \cF_{1/L,\infty}\right)^{-1}$, and $\partial \cF_{\mu,L}= \mu I+\partial \cF_{0,L-\mu}$ are well known \cite{Taylor2017}.
\qed\end{proof}
% First consider purely imaginary points of the form $z=z_ii\in \{z\,|\,\Re z= 0\}$.
% Define $f(w_1,w_2)=\delta_{\{0\}}(w_2)$. 
% Then 
% \[
% (0,0)\in \partial f(0,0),\quad
% (0,z_i)\in \partial f(1,0)
% \]
% for any $z_i\in \reals$,
% and this gives us
% \[
% z_ii=
% \frac{\|(0,z_i)-(0,0)\|}{\|(1,0)-(0,0)\|}\exp\left[ i\angle ((0,z_i)-(0,0),(1,0)-(0,0))\right]\in 
% \cG(\partial f)
% \subset\cG(\partial \mathcal{F}_{0,\infty}).
% \]
% Next, consider the points $z=z_r+z_ii\in \{z\,|\,\Re z> 0\}$.
% Define
% \[
% f(w_1,w_2)=
% \frac{1}{2}
% \begin{bmatrix}
% w_1\\w_2
% \end{bmatrix}^T
% \begin{bmatrix}
% z_r&z_i\\
% z_i&z_i^2/z_r
% \end{bmatrix}
% \begin{bmatrix}
% w_1\\w_2
% \end{bmatrix}.
% \]
% Since
% \[
% \begin{bmatrix}
% z_r&z_i\\
% z_i&z_i^2/z_r
% \end{bmatrix}
% =
% \begin{bmatrix}
% \sqrt{z_r}\\z_i/\sqrt{z_r}
% \end{bmatrix}^T
% \begin{bmatrix}
% \sqrt{z_r}\\z_i/\sqrt{z_r}
% \end{bmatrix}
% \]
% is positive semidefinite, $f$ is convex.
% Then 
% \[
% (0,0)=\nabla f(0,0),\quad(z_r,z_i)=\nabla f(1,0),
% \]
% and, with \eqref{eq:srg-alternative}, we have
% \[
% z_r+z_ii\in 
% \cG(\nabla f)
% \subset\cG(\partial \mathcal{F}_{0,\infty}).
% \]

% Finally, we note that $\infty\in \cG(\partial \cF_{0\,\infty})$ is equivalent to the statement that there exists a multi-valued operator in $\partial \cF_{0\,\infty}$.
% The subdifferential operator of $\|x\|$ is one such example.

%%%%%%%%%%%%%%%%%%%%%%%%%%%%%%%%%%%%%%%%
\subsection{SRG-full classes} \label{ss:srg-full}
%%%%%%%%%%%%%%%%%%%%%%%%%%%%%%%%%%%%%%%%
Section~\ref{ss:srg-operator} discussed how given an operator we can draw its SRG.
Conversely, can we examine the SRG and conclude something about the operator?
To perform this type of reasoning, we need further conditions.

We say class of operators $\cA$ is \emph{SRG-full} if
\begin{align*}
A\in \cA
\quad\Leftrightarrow\quad
\cG(A)\subseteq\cG(\cA).
\end{align*}
Since the implication $A\in \cA\Rightarrow\cG(A)\subseteq \cG(\cA)$ already follows from the SRG's definition, the substance of this definition is the implication $\cG(A)\subseteq \cG(\cA)\Rightarrow A\in \cA$.
Essentially, a class is SRG-full if it can be fully characterized by its SRG; given an SRG-full class $\cA$ and an operator $A$, we can check membership $A\in \cA$ by verifying (through geometric arguments) the containment $\cG(A)\subseteq \cG(\cA)$ in the 2D plane. 

SRG-fullness assumes the desirable property $\cG(A)\subseteq \cG(\cA)\Rightarrow A\in \cA$.
We now discuss which classes possess this property.

%%%%%%%%%%%%%%%%%%%%%%%%%%%%%%%%%%%%%%%%
\begin{theorem} \label{thm:srg-full}
%%%%%%%%%%%%%%%%%%%%%%%%%%%%%%%%%%%%%%%%
An operator class $\cA$ is SRG-full if it is defined by
\[
A\in \cA\quad\Leftrightarrow\quad
h\left(\|u-v\|^2,\|x-y\|^2,\langle u-v,x-y\rangle\right)\le 0,
\quad \forall u\in Ax,\,v\in Ay
\]
for some nonnegative homogeneous function $h\colon\reals^3\rightarrow\reals$.
\end{theorem}
%%%%%%%%%%%%%%%%%%%%%%%%%%%%%%%%%%%%%%%%
To clarify, $h$ is nonnegative homogeneous if $\theta h(a,b,c)= h(\theta a,\theta b,\theta c)$ for all $\theta\ge 0$. (We do not assume $h$ is smooth.)
When a class $\cA$ is defined by $h$ as in  Theorem~\ref{thm:srg-full}, we say $h$ \emph{represents} $\cA$.
For example,
%the class of $\mu$-strongly monotone operators operators, 
the $\mu$-strongly monotone class $\cM_\mu$ is represented by $h(a,b,c)=\mu b-c$,
since
\[
A\in \cM_{\mu}
\quad\Leftrightarrow\quad
\mu\|x-y\|^2\le \langle u-v,x-y\rangle,\quad\forall u\in Ax,\,v\in Ay.
\]
As another example, 
%the class of firmly-nonexpansive operators,
firmly-nonexpansive class $\cN_{1/2}$  is represented by $h(a,b,c)=a-c$,
since
\[
A\in \cN_{1/2}
\quad\Leftrightarrow\quad
\|u-v\|^2\le \langle u-v,x-y\rangle,\quad\forall u\in Ax,\,v\in Ay.
\]
By Theorem~\ref{thm:srg-full}, the classes $\cM$, $\cM_\mu$, $\cC_\beta$, $\cL_L$, and $\cN_\theta$ are all SRG-full.
Respectively, 
% $h(a,b,c)=-c$ represents $\cM$,
% $h(a,b,c)=\mu b-c$ represents $\cM_\mu$,
% $h(a,b,c)=\beta a-c$ represents $\cC_\beta$,
% $h(a,b,c)= a-Lb$ represents $\cL_L$,
% and
% $h(a,b,c)=a+(1-2\theta)b-2(1-\theta)c$ represents $\cN_{\theta}$.
\begin{itemize}
\item $\cM$ is represented by $h=-c$,
\item $\cM_\mu$   is represented by $h=\mu b-c$,
\item $\cC_\beta$ is represented by $h=\beta a-c$,
\item $\cL_L$ is represented by $h= a-Lb$,
\item $\cN_{\theta}$ is represented by $h=a+(1-2\theta)b-2(1-\theta)c$.
\end{itemize}
If $h$ and $g$ represent SRG-full classes $\cA$ and $\cB$, then $\max\{h,g\}$ represents $\cA\cap \cB$
%If $h$ and $g$ represent SRG-full classes $\cA$ and $\cB$, then 
and $\min\{h,g\}$ represents $\cA\cup \cB$.

On the other hand, the classes $\partial \cF_{0,\infty}$, $\partial \cF_{\mu,\infty}$, $\partial \cF_{0,L}$, and $\partial \cF_{\mu,L}$ are not SRG-full.
For example, the operator
\[
A(z_1,z_2)=
\begin{bmatrix}
0&-1\\
1&0
\end{bmatrix}
\begin{bmatrix}
z_2\\z_2
\end{bmatrix}
\]
satisfies $\cG(A)=\{-i,i\}\subseteq \cG( \partial \cF_{0,\infty})$.
%However, $A\notin \partial \cF_{0,\infty}$ because if there were a convex function $f$ such that $\nabla f=A$, then $DA=\nabla^2f$ must be symmetric, which is a contradiction.
However, $A\notin \partial \cF_{0,\infty}$ because there is no convex function $f$ for which $\nabla f$ = $DA$.
\begin{proof}
%[Proof of Theorem~\ref{thm:srg-full}]
Since $A\in \cA\Rightarrow \cG(A)\subseteq \cG(\cA)$ always holds, we show $\cG(A)\subseteq \cG(\cA)\Rightarrow A\in \cA$.
Assume $\cA$ is represented by $h$ and an operator $A\colon\hilbert\rightrightarrows\hilbert$ satisfies $\cG(A)\subseteq\cG(\cA)$.
Let $u_A\in Ax_A$ and $v_A\in Ay_A$ represent distinct evaluations, i.e., $x_A\ne y_A$ or $u_A\ne v_A$.

First consider the case $x_A\ne y_A$.
Then \[
z=(\|u_A-v_A\|/\|x_A-y_A\|)\exp[i\angle(u_A-v_A,x_A-y_A)]
\]
satisfies $z\in \cG(A)\subseteq \cG(\cA)$.
Since $z\in \cG(\cA)$, there is an operator $B\in \cA$ such that $u_B\in Bx_B$ and $v_B\in By_B$ with
\[
\frac{\|u_B-v_B\|^2}{\|x_B-y_B\|^2}=|z|^2,\quad
\frac{\langle u_B-v_B,x_B-y_B\rangle}{\|x_B-y_B\|^2}=\Re z.
\]
Since $h$ represents $\cA$, we have
\[
0\ge 
h\left(\|u_B-v_B\|^2,\|x_B-y_B\|^2,\langle u_B-v_B,x_B-y_B\rangle\right),
\]
and homogeneity gives us
\begin{align*}
0&\ge
h\left(\frac{\|u_B-v_B\|^2}{\|x_B-y_B\|^2},1,\frac{\langle u_B-v_B,x_B-y_B\rangle}{\|x_B-y_B\|^2}\right)\\
&=
h\left(|z|^2,1,\Re z\right)=
h\left(\frac{\|u_A-v_A\|^2}{\|x_A-y_A\|^2},1,\frac{\langle u_A-v_A,x_A-y_A\rangle}{\|x_A-y_A\|^2}\right).
\end{align*}
Finally, by homogeneity we have
\[
h\left(\|u_A-v_A\|^2,\|x_A-y_A\|^2,\langle u_A-v_A,x_A-y_A\rangle\right)\le 0.
\]

Now consider the case $x_A=y_A$ and $u_A\ne v_B$.
Then $A$ is multi-valued and $\infty\in\cG(A)\subseteq \cG(\cA)$.
Since $\infty\in \cG(\cA)$, there is a multi-valued operator $B\in \cA$ such that $u_B\in Bx_B$ and $v_B\in Bx_B$ with $u_B\ne v_B$. This implies
$h(\|u_B-v_B\|^2,0,0)\le 0$. Therefore, $h(\|u_A-v_A\|^2,0,0)\le 0$. 

In conclusion, $(x_A,u_A)$ and $(y_A,v_A)$, which represent arbitrary evaluations of $A$, satisfy the inequality defined by $h$, and we conclude $A\in \cA$.
\qed\end{proof}

\section{Operator and SRG transformations}
\label{s:srg-transformation}
% In this section, we present Theorems~\ref{thm:srg-scaling-translation} and \ref{thm:srg-inversion}, which describe how certain transformations of operators map to changes in their SRGs.
In this section, we show how transformations of operators map to changes in their SRGs.
%Under suitable conditions, we have
%  \begin{itemize}
%  \item $\cG(\cA\cap\cB)=\cG(\cA)\cap\cG(\cB)$,
%  \item $\cG(\alpha\cA)=\alpha\cG(\cA)$,
%  \item $\cG(I+\cA)=1+\cG(\cA)$,
%  \item $\cG(\cA^{-1})=(\cG(\cA))^{-1}$,
%%  \item $\cG(\cA\cup\cB)=\cG(\cA)\cup\cG(\cB)$,
%  \item $\cG(\cA+\cB)=\cG(\cA)+\cG(\cB)$, and
%  \item $\cG(\cA\cB)=\cG(\cA)\cG(\cB)$,
%  \end{itemize}
%where the precise notation and meaning is provided soon.
We then use these results and geometric arguments to analyze convergence of various fixed-point iterations.
% We then use these results and geometric arguments to prove contraction factors of various fixed-point iterations used in convex optimization and monotone operator theory.
The convergence analyses are tight in the sense that they cannot be improved without additional assumptions.

% The results we state as ``Fact'' are known, or mostly known in the sense that only a minor part of the statement is new.
% % In this section, we provide convergence analyses of various methods using the SRG machinery and geometric arguments.
% The published proofs of these ``Facts'' are based on analytic arguments and inequalities.
% Such proofs are mechanically verifiable yet unintuitive.
% Tightness of such results is usually established separately with matching lower bounds.
% In contrast, the geometric proofs we present are more intuitive, and tightness is immediate.

\subsection{SRG intersection}
%%%%%%%%%%%%%%%%%%%%%%%%%%%%%%%%%%%%%%%%
\begin{theorem} \label{thm:srg-intersection}
%%%%%%%%%%%%%%%%%%%%%%%%%%%%%%%%%%%%%%%%
If $\cA$ and $\cB$ are SRG-full classes, then  $\cA\cap \cB$ is SRG-full, and
\[
\cG(\cA\cap \mathcal{B})=\cG(\cA)\cap \cG(\mathcal{B}).
\]
\end{theorem}
The containment $\cG(\cA\cap \mathcal{B})\subseteq \cG(\cA)\cap \cG(\mathcal{B})$ holds regardless of SRG-fullness since, by definition,
$\cG(\cA \cap \cB)= \{\cG(A)\,|\, A\in \cA,\, A\in \cB\}$
and 
$\cG(\cA) \cap \cG(\cB)= \{\cG(A) \cap \cG(B)\,|\, A\in \cA,\, B\in \cB\}$.
Therefore, the substance of Theorem~\ref{thm:srg-intersection} is $\cG(\cA\cap \mathcal{B})\supseteq\cG(\cA)\cap \cG(\mathcal{B})$.
%Theorem~\ref{thm:srg-intersection} when we make more than one assumption on the operator class.
%The empty operator and the operator $A=\{(x,u)\}$ are in every SRG representable class.
This result is useful for setups with multiple assumptions on a single operator such as
Facts~\ref{thm:cont-sm-Lipschitz}, \ref{prop:sm-coco-forward}, \ref{prop:refl-sm-coco}.
A similar result holds with the union.

\begin{proof}
%[Proof of Theorem~\ref{thm:srg-intersection}]
Since $\cA$ and $\cB$ are SRG-full
\begin{align*}
\cG(C)\subseteq\cG(\cA\cap\cB)\subseteq \cG(\cA)\cap \cG(\cB)
\quad&\Rightarrow\quad \cG(C)\subseteq \cG(\cA)\text{ and }\cG(C)\subseteq \cG(\cB)\\
\quad&\Rightarrow\quad C\in \cA\text{ and }C\in \cB\\
\quad&\Rightarrow\quad C\in \cA\cap \cB
\end{align*}
for an operator $C$, and we conclude $\cA\cap\cB$ is SRG-full.

%Since $A\in \cA$ if and only if $\cG(A)\subseteq \cG(\cA)$ and $B\in \cB$ if and only if $\cG(B)\subseteq \cG(\cB)$, we have
Assume $z\in\complex$ satisfies $\{z,\bar{z}\}\subseteq \cG(\cA)\cap \cG(\cB)$.
Then $A_z$ of Lemma~\ref{lem:complex-srg} satisfies $\cG(A_z)=\{z,\bar{z}\}\subseteq \cG(\cA)\cap \cG(\cB)$.
Since $\cA$ and $\cB$ are SRG-full, $A_z\in \cA$ and $A_z\in \cB$
and $\{z,\bar{z}\}=\cG(A_z)\subseteq \cG(\cA\cap \cB)$.
If $\infty\in \cG(\cA)\cap \cG(\cB)$, then
a similar argument using $A_\infty$ of Lemma~\ref{lem:complex-srg} proves $\infty \in \cG(\cA\cap \cB)$.
Therefore $\cG(\cA)\cap \cG(\mathcal{B})\subseteq \cG(\cA\cap \mathcal{B})$.
Since the other containment $\cG(\cA\cap \mathcal{B})\subseteq \cG(\cA)\cap \cG(\mathcal{B})$ holds by definition, we have the equality.
\qed\end{proof}

\subsection{SRG scaling and translation}
\begin{theorem}
%\emph{Scaling and translation.}
\label{thm:srg-scaling-translation}
Let $\alpha\in \reals$ and $\alpha\ne 0$.
If $\cA$ is a class of operators, then 
\[
\cG(\alpha \cA)=\cG(\cA\alpha )=\alpha\cG(\cA),\qquad\cG(I+ \cA)=1+\cG(\cA).
\]
If $\cA$ is furthermore SRG-full, then $\alpha \cA$, $\cA\alpha$, and $I+ \cA$ are SRG-full.
\end{theorem}
\begin{proof}
%The proof of the first part follows directly from the definition of the SRG and \eqref{eq:srg-alternative2}
%XXX it might be helpful to elaborate on the first sentence of the proof of Theorem 4. XXX
 $\cG(\alpha A)=\alpha\cG(A)$ and $\cG(A\alpha )=\alpha\cG(A)$ follow from the definition of the SRG, and 
 $\cG(I+ A)=1+\cG(A)$ follows from  \eqref{eq:srg-alternative2}.
% The statements with the operator classes follow from the statements with the individual operators and from the definition of the SRG of operator classes.
The scaling and translation operations are reversible and $\cG((1/\alpha)\cA)=\cG(\cA(1/\alpha))=(1/\alpha)\cG(\cA)$ and $\cG(\cA-I)=\cG(A)-1$.
For any $B\colon\hilbert\rightrightarrows\hilbert$,
\[
\cG(B)\subseteq \cG(\alpha \cA)
\quad\Rightarrow\quad
\cG((1/\alpha)B)\subseteq \cG(\cA)
\quad\Rightarrow\quad
(1/\alpha)B\in \cA
\quad\Rightarrow\quad
B\in\alpha \cA,
\]
and we conclude $\alpha \cA$ is SRG-full.
By a similar reasoning,  $\cA\alpha$ and $I+ \cA$ are SRG-full.
\qed\end{proof}
Since a class of operators can consist of a single operator,
%Theorem~\ref{thm:srg-scaling-translation} implies that
if $A\colon\hilbert\rightrightarrows\hilbert$, then
\[
\cG(\alpha A)=\cG(A\alpha )=\alpha\cG(A),\qquad\cG(I+ A)=1+\cG(A).
\]
To clarify, $\alpha\cG(A)$ corresponds to scaling $\cG(A)\subseteq \ecomplex$ by $|\alpha|$ and reflecting about the vertical axis (imaginary axis) if $\alpha<0$.
%Remember that $\cG(A)$ is symmetric about the horizontal axis (real axis).
%Again, $A\alpha $ is the operator defined by $x\mapsto A(\alpha x)$.
%We require $\alpha\ne 0$ to avoid $0\cdot\infty$.
%To clarify, $1\in \complex$ and $1+\cG(A)$ corresponds to shifting $\cG(A)$ to the right by one unit.

\subsubsection{Convergence analysis: gradient descent}
Consider the optimization problem
\begin{equation}
\begin{array}{ll}
\underset{x\in \cH}{\mbox{minimize}}& f(x),
\end{array}
\end{equation}
where $f$ is a differentiable function with a minimizer.
Consider gradient descent \cite{cauchy1847}
\begin{equation}
x^{k+1}=x^k-\alpha \nabla f(x^k)
\tag{GD}
\label{eq:gd}
\end{equation}
where $\alpha>0$ and $x^0\in\hilbert$ is a starting point.
We can use the Krasnosel'ski\u{\i}--Mann theorem to establish convergence of \eqref{eq:gd}.

\begin{fact}
\label{fact:gd-conv}
Assume $f$ is convex and $L$-smooth with $L>0$.
For $\alpha\in (0,2/L)$,
the iterates of \eqref{eq:gd} converge in that $x^k\rightarrow x^\star$ weakly for some $x^\star$ such that $\nabla f(x^\star)=0$.
\end{fact}
\begin{proof}
By Propositions~\ref{prop:monotone-srg} and \ref{proposition:cvx-srg-first} and Theorem~\ref{thm:srg-scaling-translation}, we have the geometry
\begin{center}
\raisebox{-.5\height}{
\begin{tikzpicture}[scale=1.2]
%\draw (-0.85,0.5) node {$\cG(\cN_\theta)=$};
\draw[dashed] (0,0) circle (1);
\draw (-1.65,-.25) node [fill=white] {$\cG\left(I-\alpha \nabla f \right)\subseteq $};
\fill[fill=medgrey] (0.125,0) circle (0.875);
\draw [<->] (-1.2,0) -- (1.2,0);
\draw [<->] (0,-1.2) -- (0,1.2);
\filldraw (-0.75,0) circle ({0.6*1.5/1.5pt});
\draw (-0.0,0.05) node [fill=medgrey, above] {\phantom{$1-\alpha L$}};
\draw (-0.25,0.05) node [above] {$1-\alpha L$};
%\draw (1,0) node [above right] {$1$};
\filldraw (1,0) circle ({0.6*1.5/1.2pt});
%\draw [decorate,decoration={brace,amplitude=4.5pt}] (0.125,0) -- (1,0) ;
\def\w{-0.8};
\draw[->] (0.9,-.9)--({1/8*(1+sqrt(49-64*(\w)^2))},\w);
%\draw (0.9,-.8) node [right] {$\cG\left(I-\alpha\partial \cF_{0,L}\right)=\cG(\cN_{1-\alpha L/2})$};
\draw (0.9,-.9) node [right] {$\cG(\cN_{\alpha L/2})$};
\end{tikzpicture}}
\end{center}
Since $\cN_\theta$ is SRG-full by Theorem~\ref{thm:srg-full}, the containment of the SRG in $\ecomplex$ equivalent to the containment of the class.
Therefore $I-\alpha \nabla f$ is averaged, and the iteration converges by the Krasnosel'ski\u{\i}--Mann theorem.
\qed\end{proof}

With stronger assumptions, we can establish an exponential rate of convergence for \eqref{eq:gd}.

\begin{fact}
Assume $f$ is $\mu$-strongly convex and $L$-smooth with $0<\mu<L<\infty$.
For $\alpha\in (0,2/L)$, the iterates of \eqref{eq:gd} converge exponentially to the minimizer $x^\star$ with rate 
\[
\|x^k-x^\star\|\le
\left(\max\{|1-\alpha\mu|,|1-\alpha L|\}\right)^k
\|x^0-x^\star\|
.
\]
\end{fact}
\begin{proof}
This follows from Fact~\ref{fact:forward-step}, which we state and prove below.
\qed\end{proof}

%%%%%%%%%%%%%%%%%%%%%%%%%%%%%%%%%%%%%%%%
\begin{fact} \label{fact:forward-step}
%%%%%%%%%%%%%%%%%%%%%%%%%%%%%%%%%%%%%%%%
%If $f\in \mathcal{F}_{\mu,L}$,then 
Let $0<\mu<L<\infty$ and $\alpha\in (0,\infty)$.
If $\cA= \partial \cF_{\mu,L}$, then
$I-\alpha\cA\subseteq \cL_R$
for 
\[
R= \max\{|1-\alpha\mu|,|1-\alpha L|\}.
\]
This result is tight in the sense that $I-\alpha\cA\nsubseteq \cL_R$ for any smaller value of $R$.
% If $f\in \cF_{\mu,L}$, then 
% \[
% I-\alpha \nabla f\in  \cL\left({\max\{|1-\alpha\mu|,|1-\alpha L|\}}\right).
% \]
\end{fact}
% XXX comment this out? XXX
% Fact~\ref{fact:forward-step} proves the gradient method $x^{k+1}=x^k-\alpha \nabla f(x^k)$ converges linearly at the rate $\cO(R^k)$ when $f$ is $\mu$-strongly convex, $\nabla f$ is $L$-Lipschitz, and $\alpha\in (0,2/L)$.
\begin{proof}
By Proposition~\ref{proposition:cvx-srg-first} and Theorem~\ref{thm:srg-scaling-translation}, we have the geometry
\begin{center}
\begin{tikzpicture}[scale=1]
\def\x{-0.6};
\def\w{0.8};
\fill [fill=lightgrey] (0,0) circle (1);

\fill[fill=medgrey] (-0.35,0) circle (0.65);
%\draw[line width=1.0pt] (-0.35,0) circle (0.65);
\draw [<->] (-1.2,0) -- (1.2,0);
\draw [<->] (0,-1.2) -- (0,1.2);
\draw (-1,0)node [above left] {$1-\alpha L$};
\filldraw (-1,0) circle ({0.6*1.5/1pt});
\filldraw (.3,0) circle ({0.6*1.5/1pt});
\draw (0.3,0) node [above right] {$1-\alpha \mu$};
\draw[->] (-0.8,-1.)--(\x,{-sqrt(3-7*\x-10*((\x)^2))/sqrt(10)});
%\draw (-0.8,-1.) node [below] {$\cG\left(I-\alpha \partial \mathcal{F}_{\mu,L}\right)$};
\draw (-0.8,-1.) node [below] {$\cG\left(I-\alpha \cA\right)$};
\draw[->] (1.2,-.6)--(\w,{-sqrt(1-\w^2)});
\draw (1.2,-.6) node [right] {$\cG\left(\cL_R\right)$};
\draw (.6,-.8) node [below right] {$R={\max\{|1-\alpha\mu|,|1-\alpha L|\}}$};
\end{tikzpicture}
\end{center}
The containment of $\cG(I-\alpha\cA)$ holds for $R$ and fails for smaller $R$.
Since $\cL_R$ is SRG-full by Theorem~\ref{thm:srg-full}, the containment of the SRG in $\ecomplex$ equivalent to the containment of the class.
% Since $\cL\left({\max\{|1-\alpha\mu|,|1-\alpha L|\}}\right)$ is SRG-representable, 
% the inclusion of the SRG in $\ecomplex$ implies the inclusion of the class by Theorem~\ref{thm:srg-representation}.
\qed\end{proof}

% Using a similar line of reasoning, 
% We can show that 
% $I-\alpha \partial \cF_{0,L}\subseteq \cN_{\alpha L/2}$. This proves convergence of the gradient method by Fact XX.

\subsubsection{Convergence analysis: forward step method}
Consider the monotone inclusion problem
\[
\mbox{find $x\in \cH$ \quad such that \quad }0\in Ax
\]
where $A$ is a maximal monotone operator with a zero.
Consider the forward step method \cite{bruck1977}
\begin{equation}
x^{k+1}=x^k-\alpha A(x^k)
\tag{FS}
\label{eq:fs}
\end{equation}
where $\alpha>0$ and $x^0\in\hilbert$ is a starting point.
The forward step method is analogous to gradient descent.
Under the following two setups, \eqref{eq:fs} converges exponentially.

\begin{fact}
Assume $A$ is $\mu$-strongly monotone and $L$-Lipschitz with $0<\mu<L<\infty$.
For $\alpha\in (0,2\mu/L)$, the iterates of \eqref{eq:fs} converge exponentially to the zero $x^\star$ with rate 
\[
\|x^k-x^\star\|\le
\left(1-2\alpha\mu+\alpha^2 L^2\right)^{k/2}
\|x^0-x^\star\|
.
\]
\end{fact}
\begin{proof}
This follows from Fact~\ref{thm:cont-sm-Lipschitz}, which we state and prove below.
\qed\end{proof}

\begin{fact}[Proposition 26.16 \cite{BCBook}]
\label{thm:cont-sm-Lipschitz}
Let $0<\mu<L<\infty$ and $\alpha\in (0,\infty)$.
If $\cA=\mathcal{M}_\mu\cap \mathcal{L}_L$,
then $I-\alpha \cA\subseteq \cL_R$
for
\[
R= \sqrt{1-2\alpha\mu+\alpha^2 L^2}.
\]
This result is tight in the sense that $I-\alpha\cA\nsubseteq \cL_R$ for any smaller value of $R$.
% If $A\in \mathcal{M}_\mu\cap \mathcal{L}_L$,
% then $I-\alpha A\in \cL\left(\sqrt{1-2\alpha\mu+\alpha^2 L^2}\right)$.
\end{fact}
% Fact~\ref{thm:cont-sm-Lipschitz} proves the ``forward-step method'' $x^{k+1}=x^k-\alpha Ax^k$ converges linearly at the rate $\cO(R^k)$ if $A$ is $\mu$-strongly monotone, $L$-Lipschitz, and $\alpha\in (0,2\mu/L^2)$.
\begin{proof}
First consider the case $\alpha\mu>1$.
By Proposition~\ref{prop:monotone-srg} and Theorem~\ref{thm:srg-scaling-translation}, we have the geometry
\begin{center}
\begin{tabular}{ccc}
\raisebox{-.5\height}{
\begin{tikzpicture}[scale=.8]
\def\m{1.55};
\def\L{1.8};

\begin{scope}
\clip (\m,-2) rectangle (2,2);
\fill[fill=medgrey] (0,0) circle (\L);
\end{scope}

\begin{scope}
\clip (-2,-2) rectangle (\m,2);
\draw[dashed] (0,0) circle (\L);
\end{scope}

\draw[dashed] (\m,{sqrt(\L^2-\m^2)})--(\m,2);
\draw[dashed] (\m,-{sqrt(\L^2-\m^2)})--(\m,-2);

\draw [<->] (0,-2) -- (0,2);
\draw [<->] (-2,0) -- (2,0);
\filldraw (\m,-0.0) circle ({0.6*1.5/.8pt});
\draw (\m,-0.05) node [below left] {$\alpha \mu$};
\filldraw (\L,0.0) circle ({0.6*1.5/.8pt});
\draw (\L,0.05) node [below right] {$\alpha L$};

\draw[->] (.9,.6)--(\m,.2);
\draw (.7,.5) node [above] {$\cG\left(\alpha\cA\right)$};
\end{tikzpicture}
}
&
\!\!\!\!\!\!\!\!\!
\raisebox{-.5\height}{
\begin{tikzpicture}[scale=.8]
\def\m{1.55};
\def\L{1.8};

\begin{scope}
\clip ({-\m},-2) rectangle (-2,2);
\fill[fill=medgrey] (0,0) circle (\L);
\end{scope}

\begin{scope}
\clip (2,-2) rectangle ({-\m},2);
\draw[dashed] (0,0) circle (\L);
\end{scope}

\draw[dashed] ({-\m},{sqrt(\L^2-\m^2)})--({-\m},2);
\draw[dashed] ({-\m},-{sqrt(\L^2-\m^2)})--({-\m},-2);

\draw [<->] (0,-2) -- (0,2);
\draw [<->] (-2,0) -- (2,0);
\filldraw ({-\m},-0.0) circle ({0.6*1.5/.8pt});
\draw ({-\m},-0.05) node [below right] {$-\alpha \mu$};
\filldraw ({-\L},0.0) circle ({0.6*1.5/.8pt});
\draw ({-\L},0.05) node [below left] {$-\alpha L$};

\draw (0,.5) node [above, fill=white] {$\cG\left(-\alpha\cA\right)$};
\draw[->] (-.9,.75)--({-\m},.4);
\end{tikzpicture}
}
&
\!\!\!\!\!\!\!\!\!
\raisebox{-.5\height}{
\begin{tikzpicture}[scale=.8]
\def\m{1.55};
\def\L{1.8};

\begin{scope}
\clip ({1-\m},-2) rectangle (-1,2);
\fill[fill=medgrey] (1,0) circle (\L);
\end{scope}

\begin{scope}
\clip (3.0,-2) rectangle ({1-\m},2);
\draw[dashed] (1,0) circle (\L);
\end{scope}

\draw[dashed] ({1-\m},{sqrt(\L^2-\m^2)})--({1-\m},2);
\draw[dashed] ({1-\m},-{sqrt(\L^2-\m^2)})--({1-\m},-2);

\draw [<->] (0,-2) -- (0,2);
\draw [<->] (-1,0) -- (3.0,0);
\filldraw ({1-\m},-0.0) circle ({0.6*1.5/.8pt});
\draw ({1-\m},-0.05) node [below right] {$1-\alpha \mu$};
\filldraw ({1-\L},0.0) circle ({0.6*1.5/.8pt});
\draw ({1-\L},0.05) node [below left] {$1-\alpha L$};

\draw[->] (0.2,.75)--({1-\m},.4);
\draw (1.3,.5) node [above] {$\cG\left(I-\alpha\cA\right)$};
\end{tikzpicture}
}
\end{tabular}
\end{center}
\begin{center}
\begin{tabular}{cc}
\raisebox{-.5\height}{
\begin{tikzpicture}[scale=1.5]
\def\m{1.55};
\def\L{1.8};
\fill[fill=lightgrey] (0,0) circle ({sqrt(1-2*\m+\L^2)});

\begin{scope}
\clip ({1-\m},-1.3) rectangle (-1.3,1.3);
\fill[fill=medgrey] (1,0) circle (\L);
\end{scope}
%\begin{scope}
%\clip ({1-\m},-1.5) rectangle (2.5,1.5);
%\draw [dashed] (1,0) circle (\L);
%\end{scope}
%\draw[line width=1.0pt] (0.35,0) circle (1);
%\draw [dashed] (0,0) circle (.35);
%\draw [line width=1.0pt] (-.35,-0.714143) -- (-.35,0.714143);
\draw [<->] (-1.3,0) -- (1.3,0);
\draw [<->] (0,-1.3) -- (0,1.3);
%\draw [dashed] ({1-\m},{sqrt(\L^2-\m^2)}) -- (0,{sqrt(\L^2-\m^2)});
\filldraw ({1-\m},-0.0) circle ({0.6*1.5/1.5pt});
\draw ({1-\m},-0.05) node [below right, fill=lightgrey] {$1-\alpha \mu$};
\filldraw ({1-\L},0.0) circle ({0.6*1.5/1.5pt});
\draw ({1-\L},0.05) node [below left] {$1-\alpha L$};
\filldraw (1,0) circle ({0.6*1.5/1.5pt})  node [below] {$1$};
%\filldraw (0,{sqrt(\L^2-\m^2)}) circle ({0.6*1.5/1.5pt}) node [below right] {$\alpha\sqrt{L^2-\mu^2}$} ;
%\filldraw (0,{sqrt(1-2*\m+\L^2)}) circle ({0.6*1.5/1.5pt}) node [above right] {$\sqrt{1-2\alpha\mu+\alpha^2 L^2}$};
%\coordinate (I1) at (1,0);
%\coordinate (I2) at ({1-\m},0);
%\coordinate (I3) at ({1-\m},{sqrt(\L^2-\m^2)});
%\draw [] (I1)-- (I3);
%\tkzMarkRightAngle[size=.053](I1,I2,I3);

\def\w{-.72};

\draw[->] (-1.2,-.6)--(\w,{-sqrt(\L^2-(\w-1)^2)});
\draw (-1.2,-.6) node [left] {$\cG\left(I-\alpha\cA\right)$};

\def\y{.85};
\draw[->] (1.1,-.73)--(\y,{-sqrt(1-2*\m+\L^2-\y^2)});
\draw (1.1,-.73) node [right] {$\cG\left(\cL_R\right)$};
\draw (1.2,-1.) node [below] {$R=\sqrt{1-2\alpha\mu+\alpha^2 L^2}$};
\end{tikzpicture}
}
&\!\!\!\!\!\!\!\!\!\!\!\!\!\!\!
\raisebox{-.5\height}{
\begin{tikzpicture}[scale=1]
\def\m{1.55};
\def\L{1.8};
\fill[fill=lightgrey] (0,0) circle ({sqrt(1-2*\m+\L^2)});

\begin{scope}
\clip ({1-\m},-1.5) rectangle (-1.5,1.5);
\fill[fill=medgrey] (1,0) circle (\L);
\end{scope}
\begin{scope}
\clip ({1-\m},-2) rectangle (3,2);
\draw[dashed] (1,0) circle (\L);
\end{scope}
%\begin{scope}
%\clip ({1-\m},-1.5) rectangle (2.5,1.5);
%\draw [dashed] (1,0) circle (\L);
%\end{scope}
%\draw[line width=1.0pt] (0.35,0) circle (1);
%\draw [dashed] (0,0) circle (.35);
%\draw [line width=1.0pt] (-.35,-0.714143) -- (-.35,0.714143);
\draw [<->] (-1.5,0) -- (3,0);
\draw [<->] (0,-2) -- (0,2);
%\draw [dashed] ({1-\m},{sqrt(\L^2-\m^2)}) -- (0,{sqrt(\L^2-\m^2)});
%\filldraw ({1-\m},0) circle ({0.6*1.5/1pt}) node [below right] {$1-\alpha \mu$};
%\filldraw ({1-\L},0) circle ({0.6*1.5/1pt}) node [below left] {$1-\alpha L$};
%\filldraw (0,{sqrt(\L^2-\m^2)}) circle ({0.6*1.5/1pt}) node [below right] {$\alpha\sqrt{L^2-\mu^2}$} ;
%\filldraw (0,{sqrt(1-2*\m+\L^2)}) circle ({0.6*1.5/1pt}) node [above right] {$\sqrt{1-2\alpha\mu+\alpha^2 L^2}$};
\coordinate (I1) at (1,0);
\coordinate (I2) at ({1-\m},0);
\coordinate (I3) at ({1-\m},{sqrt(\L^2-\m^2)});
\coordinate (I4) at ({1-\m},{-sqrt(\L^2-\m^2)});
\coordinate (I5) at ({1-\L},0);
\draw [] (I1)-- (I3);
\tkzMarkRightAngle[size=.1](I1,I2,I3);
\draw (I1) node[below right] {$A$};
\draw (I2) node[below right] {$B$};
\draw (I3) node[above left] {$C$};
\draw (I4) node[below left] {$C'$};
\draw (I5) node[below left] {$D$};
\filldraw (I1) circle ({0.6*1.5/1pt});
\filldraw (I2) circle ({0.6*1.5/1pt});
\filldraw (I3) circle ({0.6*1.5/1pt});
\filldraw (I4) circle ({0.6*1.5/1pt});
\filldraw (I5) circle ({0.6*1.5/1pt});
\filldraw (0,0) circle ({0.6*1.5/1pt});
\draw (0,0) node[below right] {$O$};
\draw [] (0,0)-- (I3);

%
%
%\def\w{-.72};
%
%\draw[->] (-1.5,-.6)--(\w,{-sqrt(\L^2-(\w-1)^2)});
%\draw (-1.5,-.6) node [left] {$\cG\left(I-\alpha(\cM_\mu\cap \cL_L)\right)$};
%
%\def\y{.85};
%\draw[->] (1.1,-.73)--(\y,{-sqrt(1-2*\m+\L^2-\y^2)});
%\draw (1.1,-.73) node [right] {$\cG\left(\cL\left(\sqrt{1-2\alpha\mu+\alpha^2 L^2}\right)\right)$};

\end{tikzpicture}
}
\end{tabular}
\end{center}
To clarify, $O$ is the center of the circle with radius $\overline{OC}$ (lighter shade) and $A$ is the center of the circle with radius $\overline{AC}=\overline{AD}$ defining the inner region (darker shade).
%The inner region is contained within the disk with radius $\sqrt{1-2\alpha\mu+\alpha^2 L^2}$ centered at the origin by the following geometric reasoning.
With $2$ applications of the Pythagorean theorem, we get 
\begin{align*}
\overline{OC}^2&=\overline{CB}^2+\overline{BO}^2
=
\overline{AC}^2-\overline{BA}^2+
\overline{BO}^2\\
&=(\alpha L)^2-(\alpha\mu)^2+(1-\alpha\mu)^2=1-2\alpha\mu+\alpha^2L^2.
\end{align*}
Since $\overline{C'C}$ is a chord of circle $O$, it is within the circle.
Since $2$ non-identical circles intersect at no more than 2 points, and since $D$ is within circle $O$, arc
$\arc{CDC'}$ is within circle $O$.
Finally, the region bounded by $\overline{C'C}\cup \arc{CDC'}$ (darker shade) is within circle $O$ (lighter shade).

The previous diagram illustrates the case $\alpha \mu>1$. 
In the cases $\alpha\mu=1$ and $\alpha\mu< 1$, we have a slightly different geometry, but the same arguments and calculations hold.
\begin{center}
\begin{tabular}{cc}
\raisebox{-.5\height}{
\begin{tikzpicture}[scale=1.5]
\def\m{1};
\def\L{1.2};
\fill[fill=lightgrey] (0,0) circle ({sqrt(1-2*\m+\L^2)});

\begin{scope}
\clip ({1-\m},-1.5) rectangle (-1.,1.5);
\fill[fill=medgrey] (1,0) circle (\L);
\end{scope}
\begin{scope}
\clip ({1-\m},-1.5) rectangle (2.4,1.5);
\draw[dashed] (1,0) circle (\L);
\end{scope}
%\begin{scope}
%\clip ({1-\m},-1.5) rectangle (2.5,1.5);
%\draw [dashed] (1,0) circle (\L);
%\end{scope}
%\draw[line width=1.0pt] (0.35,0) circle (1);
%\draw [dashed] (0,0) circle (.35);
%\draw [line width=1.0pt] (-.35,-0.714143) -- (-.35,0.714143);
\draw [<->] (-1.,0) -- (2.4,0);
\draw [<->] (0,-1.5) -- (0,1.5);
%\draw [dashed] ({1-\m},{sqrt(\L^2-\m^2)}) -- (0,{sqrt(\L^2-\m^2)});
%\filldraw ({1-\m},0) circle ({0.6*1.5/1.5pt}); node [below right] {$1-\alpha \mu$};
%\filldraw ({1-\L},0) circle ({0.6*1.5/1.5pt}); node [below left] {$1-\alpha L$};
%\filldraw (0,{sqrt(\L^2-\m^2)}) circle ({0.6*1.5/1.5pt}); node [below right] {$\alpha\sqrt{L^2-\mu^2}$} ;
%\filldraw (0,{sqrt(1-2*\m+\L^2)}) circle ({0.6*1.5/1.5pt}); node [above right] {$\sqrt{1-2\alpha\mu+\alpha^2 L^2}$};
\coordinate (I1) at (1,0);
\coordinate (I2) at ({1-\m},0);
\coordinate (I3) at ({1-\m},{sqrt(\L^2-\m^2)});
\coordinate (I4) at ({1-\m},{-sqrt(\L^2-\m^2)});
\coordinate (I5) at ({1-\L},0);
\draw [] (I1)-- (I3);
\tkzMarkRightAngle[size=.066666](I1,I2,I3);
\draw (I1) node[below right] {$A$};
\draw (I3) node[above left] {$C$};
\draw (I4) node[below left] {$C'$};
\draw (I5) node[below left] {$D$};
\filldraw (I1) circle ({0.6*1.5/1.5pt});
\filldraw (I3) circle ({0.6*1.5/1.5pt});
\filldraw (I4) circle ({0.6*1.5/1.5pt});
\filldraw (I5) circle ({0.6*1.5/1.5pt});
\filldraw (0,0) circle ({0.6*1.5/1.5pt});
\draw (-0.05,0) node[below right] {$B=O$};
\draw [] (0,0)-- (I3);
\draw (1.4,-1.2) node[fill=white] {Case $\alpha\mu=1$};

%
%
%\def\w{-.72};
%
%\draw[->] (-1.5,-.6)--(\w,{-sqrt(\L^2-(\w-1)^2)});
%\draw (-1.5,-.6) node [left] {$\cG\left(I-\alpha(\cM_\mu\cap \cL_L)\right)$};
%
%\def\y{.85};
%\draw[->] (1.1,-.73)--(\y,{-sqrt(1-2*\m+\L^2-\y^2)});
%\draw (1.1,-.73) node [right] {$\cG\left(\cL\left(\sqrt{1-2\alpha\mu+\alpha^2 L^2}\right)\right)$};

\end{tikzpicture}
}
&
\raisebox{-.5\height}{
\begin{tikzpicture}[scale=1.5]
\def\m{.5};
\def\L{1.2};
\fill[fill=lightgrey] (0,0) circle ({sqrt(1-2*\m+\L^2)});

\begin{scope}
\clip ({1-\m},-1.5) rectangle (-1.4,1.5);
\fill[fill=medgrey] (1,0) circle (\L);
\end{scope}
\begin{scope}
\clip ({1-\m},-1.5) rectangle (2.4,1.5);
\draw[dashed] (1,0) circle (\L);
\end{scope}
%\begin{scope}
%\clip ({1-\m},-1.5) rectangle (2.5,1.5);
%\draw [dashed] (1,0) circle (\L);
%\end{scope}
%\draw[line width=1.0pt] (0.35,0) circle (1);
%\draw [dashed] (0,0) circle (.35);
%\draw [line width=1.0pt] (-.35,-0.714143) -- (-.35,0.714143);
\draw [<->] (-1.4,0) -- (2.4,0);
\draw [<->] (0,-1.5) -- (0,1.5);
%\draw [dashed] ({1-\m},{sqrt(\L^2-\m^2)}) -- (0,{sqrt(\L^2-\m^2)});
%\filldraw ({1-\m},0) circle ({0.6*1.5/1.5pt}); node [below right] {$1-\alpha \mu$};
%\filldraw ({1-\L},0) circle ({0.6*1.5/1.5pt}); node [below left] {$1-\alpha L$};
%\filldraw (0,{sqrt(\L^2-\m^2)}) circle ({0.6*1.5/1.5pt}); node [below right] {$\alpha\sqrt{L^2-\mu^2}$} ;
%\filldraw (0,{sqrt(1-2*\m+\L^2)}) circle ({0.6*1.5/1.5pt}); node [above right] {$\sqrt{1-2\alpha\mu+\alpha^2 L^2}$};
\coordinate (I1) at (1,0);
\coordinate (I2) at ({1-\m},0);
\coordinate (I3) at ({1-\m},{sqrt(\L^2-\m^2)});
\coordinate (I4) at ({1-\m},{-sqrt(\L^2-\m^2)});
\coordinate (I5) at ({1-\L},0);
\draw [] (I1)-- (I3);
\tkzMarkRightAngle[size=.066666](I1,I2,I3);
\draw (I1) node[below right] {$A$};
\draw (I2) node[below right] {$B$};
\draw (I3) node[above] {$C$};
\draw (I4) node[below] {$C'$};
\draw (I5) node[below left] {$D$};
\filldraw (I1) circle ({0.6*1.5/1.5pt});
\filldraw (I2) circle ({0.6*1.5/1.5pt});
\filldraw (I3) circle ({0.6*1.5/1.5pt});
\filldraw (I4) circle ({0.6*1.5/1.5pt});
\filldraw (I5) circle ({0.6*1.5/1.5pt});
\filldraw (0,0) circle ({0.6*1.5/1.5pt});
\draw (0,0) node[below right] {$O$};
\draw [] (0,0)-- (I3);
\draw (1.4,-1.2) node[fill=white] {Case $\alpha\mu>1$};

%
%
%\def\w{-.72};
%
%\draw[->] (-1.5,-.6)--(\w,{-sqrt(\L^2-(\w-1)^2)});
%\draw (-1.5,-.6) node [left] {$\cG\left(I-\alpha(\cM_\mu\cap \cL_L)\right)$};
%
%\def\y{.85};
%\draw[->] (1.1,-.73)--(\y,{-sqrt(1-2*\m+\L^2-\y^2)});
%\draw (1.1,-.73) node [right] {$\cG\left(\cL\left(\sqrt{1-2\alpha\mu+\alpha^2 L^2}\right)\right)$};

\end{tikzpicture}
}
\end{tabular}
\end{center}

The containment holds for $R$ and fails for smaller $R$.
Since $\cL_R$ is SRG-full by Theorem~\ref{thm:srg-full}, the containment of the SRG in $\ecomplex$ equivalent to the containment of the class.
% Since $\cL\left(\sqrt{1-2\alpha\mu+\alpha^2 L^2}\right)$ is SRG-representable, the inclusion of the SRG in $\ecomplex$ implies the inclusion of the class by Theorem~\ref{thm:srg-representation}.
\qed\end{proof}

% The Lipschitz constant of Theorem~\ref{thm:cont-Lip-monotone} never provides a contraction as $\sqrt{1+\alpha^2L^2}\ge 1$ for all $\alpha\in \reals$,

\begin{fact}
Assume $A$ is $\mu$-strongly monotone and $\beta$-cocoercive with $0<\mu<1/\beta<\infty$.
For $\alpha\in (0,2\beta)$, the iterates of \eqref{eq:fs} converge exponentially to the zero $x^\star$ with rate 
\[
\|x^k-x^\star\|\le
\left(1-2\alpha\mu+\alpha^2\mu/\beta\right)^{k/2}
\|x^0-x^\star\|
.
\]
\end{fact}
\begin{proof}
This follows from Fact~\ref{prop:sm-coco-forward} below.
\qed\end{proof}

\begin{fact}
\label{prop:sm-coco-forward}
Let $0<\mu<1/\beta<\infty$ and $\alpha\in (0,2\beta)$.
If $\cA=\mathcal{M}_\mu\cap \mathcal{C}_\beta$, then $I-\alpha \cA\subseteq \cL_R$ for
\[
R= \sqrt{1-2\alpha\mu+\alpha^2\mu/\beta}.
\]
This result is tight in the sense that $I-\alpha\cA\nsubseteq \cL_R$ for any smaller value of $R$.
% If $A\in \mathcal{M}_\mu\cap \mathcal{C}_\beta$ and if $0<\alpha\le 2\beta$, then $I-\alpha A\in \cL\left(\sqrt{1-2\alpha\mu+\alpha^2\mu/\beta}\right)$.
\end{fact}
% Fact~\ref{prop:sm-coco-forward} proves the ``forward-step method'' $x^{k+1}=x^k-\alpha Ax^k$ converges linearly at the rate $\cO(R^k)$ if $A$ is $\mu$-strongly monotone, $\beta$-cocoercive, and $\alpha\in (0,2\beta)$.
%\begin{proof}
\emph{Proof outline}\,
We quickly outline the geometric insight while deferring the full proof with precise geometric arguments to the Section~\ref{s:geo-proofs} in the appendix.
For the case $\mu<1/(2\beta)$, we have the geometry 
\begin{center}
\begin{tabular}{cc}
\raisebox{-.5\height}{
\begin{tikzpicture}[scale=1.5]
\def\m{.5};
\def\b{0.6};
\fill [fill=lightgrey] (0,0) circle ({sqrt(1-2*\m+\m/\b)});
\begin{scope}
\clip ({1-\m},-1.1) rectangle (-1.1,1.1);
\fill[fill=medgrey] ({1-1/2/\b},0) circle ({1/2/\b});
\end{scope}
%\draw [dashed] (0,0) circle ({sqrt(1-2*\m+\L^2)});
\draw [<->] (-1.2,0) -- (1.2,0);
\draw [<->] (0,-1.2) -- (0,1.2);
%\draw [dashed] ({1-\m},{sqrt(\L^2-\m^2)}) -- (0,{sqrt(\L^2-\m^2)});

%\filldraw ({1-\m},{sqrt(1/4/\b^2-(\m-1/2/\b)^2)}) circle ({0.6*1.5/1.5pt});

%Draw braced note
%\draw [decorate,decoration={brace,amplitude=4.5pt}] ({sqrt(1-2*\m+\m/\b)*cos(240)},{sqrt(1-2*\m+\m/\b)*sin(240)}) -- (0,0);
%\draw [line width=1.0pt]({sqrt(1-2*\m+\m/\b)*cos(240)},{sqrt(1-2*\m+\m/\b)*sin(240)}) -- (0,0);
%\draw[->] (-1.1,-.77)--({sqrt(1-2*\m+\m/\b)*cos(240)/2-0.07},{sqrt(1-2*\m+\m/\b)*sin(240)/2+.03});
%\draw (-1.1,-.77) node[below] {$\sqrt{1-2\alpha\mu+\alpha^2\mu/\beta}$};

\coordinate (I1) at  ({1-1/\b/2},0) ;
\coordinate (I2) at ({1-\m},0);
\coordinate (I3) at ({1-\m},{sqrt(1/4/\b^2-(\m-1/2/\b)^2)});
%\draw [] (I1)-- (I3);

\filldraw (I2) circle ({0.6*1.5/1.5pt});

%\draw  ({1-1/\b/2},0.05) node [above left, fill=medgrey] {$1-1/(2\beta)$};
\filldraw ({1-\m},0) node [below ] {$1-\alpha \mu$};

\filldraw ({1-1/\b},0) circle ({0.6*1.5/1.5pt}) ;
\begin{scope}
\clip ({1-\m},-1.1) rectangle (-1.1,1.1);
\clip ({1-1/2/\b},0) circle ({1/2/\b});
\draw ({1-1/\b},0.03) node [above right, fill=medgrey] {$1-\alpha/\beta$};
\end{scope}
\filldraw (1,0) circle ({0.6*1.5/1.5pt})  node [above right] {$1$};
%\filldraw (0,{sqrt(\L^2-\m^2)}) circle ({0.6*1.5/1.5pt}) node [below right] {$\alpha\sqrt{L^2-\mu^2}$} ;
%\filldraw (0,{sqrt(1-2*\m+\L^2)}) circle ({0.6*1.5/1.5pt}) node [above right] {$\sqrt{1-2\alpha\mu+\alpha^2 L^2}$};

\coordinate (O1) at ({1-1/\b},0);
\coordinate (O2) at ({1-1/\b/2},0);
\coordinate (O3) at ({1-\m},{sqrt(1/4/\b^2-(\m-1/2/\b)^2)});

\draw[->] (1.1,.8)--({1-\m,0.4});
\draw(1.1,.8) node[above] {$\cG\left(I-\alpha\cA \right)$};

\def\x{-.8};
\draw[->] (-.9,-.8)--(\x,{-sqrt(1-2*\m+\m/\b-(\x)^2)});
\draw(-.9,-.8) node[below ] {$\cG\left(\cL_R\right)$};
\draw(1.2,-1.05) node[] {$R=\sqrt{1-2\alpha\mu+\alpha^2\mu/\beta}$};
\end{tikzpicture}
}
&
\!\!\!\!\!\!\!\!\!\!\!\!\!\!\!\!\!\!\!\!\!\!\!\!\!\!\!
\raisebox{-.5\height}{
\begin{tikzpicture}[scale=1.5]
\def\m{.5};
\def\b{0.6};
\fill [fill=lightgrey] (0,0) circle ({sqrt(1-2*\m+\m/\b)});

\begin{scope}
\clip ({1-\m},-1.1) rectangle (-1.1,1.1);
\fill[fill=medgrey] ({1-1/2/\b},0) circle ({1/2/\b});
\end{scope}

\begin{scope}
\clip ({1-\m},-1.1) rectangle (1.1,1.1);
\draw[dashed] ({1-1/2/\b},0) circle ({1/2/\b});
\end{scope}

%\draw [dashed] (0,0) circle ({sqrt(1-2*\m+\L^2)});
\draw [<->] (-1.2,0) -- (1.2,0);
\draw [<->] (0,-1.2) -- (0,1.2);
%\draw [dashed] ({1-\m},{sqrt(\L^2-\m^2)}) -- (0,{sqrt(\L^2-\m^2)});

%\filldraw ({1-\m},{sqrt(1/4/\b^2-(\m-1/2/\b)^2)}) circle ({0.6*1.5/1.5pt});

%\draw [dashed] (0,0) circle ({sqrt(1-2*\m+\m/\b)});

%\draw [decorate,decoration={brace,amplitude=4.5pt}] ({sqrt(1-2*\m+\m/\b)*cos(240)},{sqrt(1-2*\m+\m/\b)*sin(240)}) -- (0,0);
%\draw [line width=1.0pt]({sqrt(1-2*\m+\m/\b)*cos(240)},{sqrt(1-2*\m+\m/\b)*sin(240)}) -- (0,0);
%\draw[->] (-1.1,-.77)--({sqrt(1-2*\m+\m/\b)*cos(240)/2-0.07},{sqrt(1-2*\m+\m/\b)*sin(240)/2+.03});
%\draw (-1.1,-.77) node[below] {$\sqrt{1-2\alpha\mu+\alpha^2\mu/\beta}$};

\coordinate (I1) at  ({1-1/\b/2},0) ;
\coordinate (I2) at ({1-\m},0);
\coordinate (I3) at ({1-\m},{sqrt(1/4/\b^2-(\m-1/2/\b)^2)});
\coordinate (I4) at ({1-\m},{-sqrt(1/4/\b^2-(\m-1/2/\b)^2)});
%\draw [] (I1)-- (I3);

\filldraw (I1) circle ({0.6*1.5/1.5pt});
\filldraw (I2) circle ({0.6*1.5/1.5pt});
\filldraw (I3) circle ({0.6*1.5/1.5pt});
\filldraw (I4) circle ({0.6*1.5/1.5pt});
\draw (I3) -- (I1);
\tkzMarkRightAngle[size=.06666](I1,I2,I3);

\draw (I1) node[below] {$C$};
\draw (I3) node[above] {$B$};
\draw (I4) node[below] {$B'$};
\draw (0,0) -- (I3);

%\draw  ({1-1/\b/2},0.05) node [above left, fill=medgrey] {$1-1/(2\beta)$};
\filldraw ({1-\m},0) node [below right ] {$D$};

\filldraw ({1-1/\b},0) circle ({0.6*1.5/1.5pt}) ;
\draw ({1-1/\b},0.03) node [above left,] {$A$};
\filldraw (1,0) circle ({0.6*1.5/1.5pt})  node [above right] {$1$};
%\filldraw (0,{sqrt(\L^2-\m^2)}) circle ({0.6*1.5/1.5pt}) node [below right] {$\alpha\sqrt{L^2-\mu^2}$} ;
%\filldraw (0,{sqrt(1-2*\m+\L^2)}) circle ({0.6*1.5/1.5pt}) node [above right] {$\sqrt{1-2\alpha\mu+\alpha^2 L^2}$};

\filldraw (0,0) circle ({0.6*1.5/1.5pt}) ;

\draw (0,0) node[above left] {$O$};
%\filldraw (I3) -- (0,0);

\coordinate (O1) at ({1-1/\b},0);
\coordinate (O2) at ({1-1/\b/2},0);
\coordinate (O3) at ({1-\m},{sqrt(1/4/\b^2-(\m-1/2/\b)^2)});

% \tkzMarkSegment[pos=.5,mark=||](O1,O2) ;
% \tkzMarkSegment[pos=.5,mark=||](O2,O3) ;
\end{tikzpicture}
}
\end{tabular}
\end{center}
where the calculations involve the use of the Pythagorean theorem.
In the cases $\mu=1/(2\beta)$ and $\mu>1/(2\beta)$, we have a slightly different geometry, but the same arguments and calculations hold.
\qed%\end{proof}

\subsection{SRG inversion}
\label{ss:inversion}
In this subsection, we relate inversion of operators with inversion (reciprocal) of complex numbers.
This operation is intimately connected to inversive geometry.
%that allows us to analyze the inversion using geometric arguments.

%present Theorem~\ref{thm:srg-inversion}

\subsubsection{Operator inversion}
% Theorem~\ref{thm:srg-inversion} relates operator inversion to geometric inversion.
% Use this result to geometrically prove convergence of methods involving resolvents and proximal mappings.

\begin{theorem}
%\emph{Inversion.}
\label{thm:srg-inversion}
If $\cA$ is a class of operators, then 
\[
\cG(\cA^{-1})=\left(\cG(\cA)\right)^{-1}.
\]
If $\cA$ is furthermore SRG-full, then $\cA^{-1}$ is  SRG-full.
\end{theorem}
Since a class of operators can consist of a single operator,
%Theorem~\ref{thm:srg-inversion} implies that
if $A\colon\hilbert\rightrightarrows\hilbert$, then
$\cG(A^{-1})=(\cG(A))^{-1}$.
To clarify, $(\cG(\cA))^{-1}=\{z^{-1}\,|\,z\in \cG(\cA)\}\subseteq\ecomplex$.
Note that $(\cG(\cA))^{-1}=(\overline{\cG(\cA)})^{-1}$, since $\cG(\cA)$ is symmetric about the real axis, so we write the simpler $(\cG(\cA))^{-1}$ even though the inversion map we consider is $z\mapsto \bar{z}^{-1}$.

%Section~\ref{ss:circle_inversion}

\begin{proof}
The equivalence of non-zero finite points, i.e.,
\[
\cG(A^{-1})\backslash\{0,\infty\}=\left(\cG(A)\right\backslash\{0,\infty\})^{-1},
\]
follows from
\begin{align*}
\mathcal{G}(A)\backslash\{0,\infty\}&=
\left\{
\frac{\|u-v\|}{\|x-y\|}
\exp\left[\pm i \angle (u-v,x-y)\right]
\,\Big|\,
(x,u),(y,v)\in A,\, x\ne y,\,u\ne v \right\}
\end{align*}
and
\begin{align*}
&\mathcal{G}(A^{-1})\backslash\{0,\infty\}\\
&\qquad=
\left\{
\frac{\|x-y\|}{\|u-v\|}
\exp\left[\pm i \angle (x-y,u-v)\right]
\,\Big|\,
(u,x),(v,y)\in A^{-1},\, x\ne y,\,u\ne v \right\}\\
&\qquad=
\left\{
\frac{\|x-y\|}{\|u-v\|}
\exp\left[\pm i \angle (u-v,x-y)\right]
\,\Big|\,
(x,u),(y,v)\in A,\, x\ne y,\,u\ne v \right\}\\
&\qquad=\left(\cG(A)\backslash\{0,\infty\}\right)^{-1}
\end{align*}
where we use the fact that $\angle(a,b)=\angle(b,a)$.
% $(x,u)\in A$ if and only if $(u,x)\in A^{-1}$
% and the fact that

The equivalence of the zero and infinite points follow from
\begin{align*}
\infty\in \cG(A)
&\quad\Leftrightarrow\quad
\exists\, (x,u),(x,v)\in A,\,u\ne v\\
&\quad\Leftrightarrow\quad
\exists\, (u,x),(v,x)\in A^{-1},\,u\ne v\\
&\quad\Leftrightarrow\quad
0\in \cG(A^{-1}).
\end{align*}
With the same argument, we have
$0\in \cG(A) \Leftrightarrow \infty\in \cG(A^{-1})$.

The inversion operation is reversible.
For any $B\colon\hilbert\rightrightarrows\hilbert$,
\[
\cG(B)\subseteq \cG( \cA^{-1})
\quad\Rightarrow\quad
\cG(B^{-1})\subseteq \cG(\cA)
\quad\Rightarrow\quad
B^{-1}\in \cA
\quad\Rightarrow\quad
B\in\cA^{-1},
\]
and we conclude $\cA^{-1}$ is SRG-full.
\qed\end{proof}

\subsubsection{Convergence analysis: proximal point}
Consider the monotone inclusion problem
\[
\mbox{find $x\in \cH$ \quad such that \quad }0\in Ax
\]
where $A$ is a maximal monotone operator with a zero.
Consider the proximal point method \cite{martinet1970,martinet1972,rockafellar1976,brezis1978}
\begin{equation}
x^{k+1}=J_{\alpha A}x^k,\tag{PP}
\label{eq:pp}
\end{equation}
where $\alpha>0$ and $x^0\in \hilbert$ is a starting point. 
Since $J_{\alpha A}$ is $1/2$-averaged, we can use the Krasnosel'ski\u{\i}--Mann theorem to establish convergence of \eqref{eq:pp}.
%Fact~\ref{fact:gd-conv}.
Under stronger assumptions, \eqref{eq:pp} converges exponentially.

\begin{fact}
Assume $A$ is $\mu$-strongly monotone with $\mu>0$.
For $\alpha>0$, the iterates of \eqref{eq:pp} converge exponentially to the zero $x^\star$ with rate 
\[
\|x^k-x^\star\|\le
\left(\frac{1}{1+\alpha\mu}\right)^k
\|x^0-x^\star\|
.
\]
\end{fact}
\begin{proof}
This follows from Fact~\ref{prop:sm-res}, which we state and prove below.
\qed\end{proof}

\begin{fact}[Proposition 23.13 \cite{BCBook}]
\label{prop:sm-res}
Let $\mu\in(0,\infty)$ and $\alpha\in(0,\infty)$.
If $\cA=\mathcal{M}_\mu$, then 
$J_{\alpha \cA}\subseteq \cL_R$
for
\[
R= \frac{1}{1+\alpha\mu}.
\]
This result is tight in the sense that $J_{\alpha \cA}\nsubseteq \cL_R$ for any smaller value of $R$.
% If $A\in \mathcal{M}_\mu$, then 
% \[
% J_{\alpha A}\in \cL\left(\frac{1}{1+\alpha\mu}\right).
% \]
\end{fact}
\begin{proof}
By Proposition~\ref{prop:monotone-srg} and Theorems~\ref{thm:srg-scaling-translation} and \ref{thm:srg-inversion}, we have the geometry 
\begin{center}
\begin{tabular}{ccc}
\raisebox{-.5\height}{
\begin{tikzpicture}[scale=1.5]
\fill[fill=medgrey] (1.2,-1.2) rectangle (2,1.2);
\draw (1.2,0.) node [above right] {$1+\alpha\mu$};
\draw (1.6,1) node {$\cup \{\binfty\}$};
\draw [<->] (-1.2,0) -- (2,0);
\draw [<->] (0,-1.2) -- (0,1.2);
\draw [dashed] (0,0) circle (1);
\draw (1,-0) node [above left] {$1$};
\filldraw (1,0) circle ({0.6*1.5/1.5pt});
\filldraw (1.2,0)  circle ({0.6*1.5/1.5pt});
\draw [->] (0.7,.5) -- (1.2,0.5);
\draw (0.7,.5) node [left, fill=white] {$\cG\left(I+\alpha \cM_\mu\right)$};
\end{tikzpicture}}
&
$\stackrel{\bar{z}^{-1}}{\longrightarrow}$
&
\raisebox{-.5\height}{
\begin{tikzpicture}[scale=1.5]
\fill [fill=lightgrey] (0,0) circle (0.83333333333);
\fill[fill=medgrey] (0.41666666666,0) circle (0.41666666666);
\draw [<->] (-1.2,0) -- (1.2,0);
\draw [<->] (0,-1.2) -- (0,1.2);
\draw [dashed] (0,0) circle (1);
\draw (1,-0pt) node [above right] {$1$};
\filldraw (1,0) circle ({0.6*1.5/1.5pt});
\filldraw ((0.83333333333,0)  circle ({0.6*1.5/1.5pt});
\draw (0.83,-.05) node [above left] {$\frac{1}{1+\alpha\mu}$};
\def\x{0.35}
\draw [->] (1.,.85) -- (\x,{sqrt((0.833333/2)^2-(\x-(0.833333/2))^2)});
\draw (1.,.85) node [above] {$\cG\left(J_{\alpha \cA}\right)$};
%\draw (-1.2,1.) node [above, fill=white] {$\cG\left(\left(I+\alpha \cM_\mu\right)^{-1}\right)$};

\def\y{0.65}
\draw [->] (.85,-.5) -- (\y,{-sqrt((0.833333)^2-(\y)^2)});
\draw (.85,-.5) node [right, fill=white] {$\cG\left(\cL_R\right)$};
\draw (.9,-.95) node [fill=white] {$R=\frac{1}{1+\alpha\mu}$};

\end{tikzpicture}}
\end{tabular}
\end{center}

The containment holds for $R$ and fails for smaller $R$.
Since $\cL_R$ is SRG-full by Theorem~\ref{thm:srg-full}, the containment of the SRG in $\ecomplex$ equivalent to the containment of the class.
%Since $\cL\left(\frac{1}{1+\alpha\mu}\right)$ is SRG-representable, the inclusion of the SRG in $\ecomplex$ implies the inclusion of the class by Theorem~\ref{thm:srg-representation}.
\qed\end{proof}

\subsubsection{Convergence analysis: Douglas--Rachford}
Consider the monotone inclusion problem
\[
\mbox{find $x\in \cH$ \quad such that \quad }0\in (A+B)x,
\]
where $A$ and $B$ are operators and $A+B$ has a zero.
Consider Douglas--Rachford splitting \cite{douglas1956,lions1979}
\begin{equation}
z^{k+1}=\left(\tfrac{1}{2}I+\tfrac{1}{2}(2J_{\alpha A}-I)(2J_{\alpha B}-I)\right)z^k,\tag{DR}
\label{eq:dr}
\end{equation}
where $\alpha>0$ and $z^0\in \hilbert$ is a starting point. 
If $z^\star$ is a fixed point, then $J_{\alpha B}(z^\star)$ is a zero of $A+B$
(see tutorial \cite[p.\ 28]{ryu2016} or textbook  \cite[Proposition 26.1]{BCBook}).
% Again, convergence follows from the Krasnosel'ski\u{\i}--Mann theorem when $A$ and $B$ are maximal monotone.
% When we have further assumptions, we can provide a stronger rate of convergence for \eqref{eq:dr}.
We can use the Krasnosel'ski\u{\i}--Mann theorem to establish convergence of \eqref{eq:dr}.

\begin{fact}[Theorem 1~\cite{lions1979}]
Assume $A$ and $B$ are maximal monotone. 
For $\alpha>0$,
the iterates of \eqref{eq:dr} converge in that $z^k\rightarrow z^\star$ weakly for some fixed point $z^\star$.
\end{fact}
\begin{proof}
By Proposition~\ref{prop:monotone-srg} and  Theorems~\ref{thm:srg-scaling-translation} and \ref{thm:srg-inversion}, we have the geometry
\begin{center}
\begin{tabular}{ccccccc}
\raisebox{-.5\height}{
\begin{tikzpicture}[scale=1]
\fill[fill=medgrey] (1,-1.2) rectangle (2.8,1.2);
\draw (2.35,1) node {$\cup \{\binfty\}$};
\draw [<->] (-1.2,0) -- (2.8,0);
\draw [<->] (0,-1.2) -- (0,1.2);
\draw [dashed] (0,0) circle (1);
\draw (1,-0) node [above left] {$1$};
\filldraw (1,0) circle ({0.6*1.5/1pt});
%\draw [->] (0.7,.5) -- (1.,0.5);
\draw (1.875,-.5) node  {$\cG\left(I+ \alpha \cM\right)$};
\end{tikzpicture}}
&\!\!\!\!\!\!
$\stackrel{\bar{z}^{-1}}{\longrightarrow}$
&\!\!\!\!\!\!
\raisebox{-.5\height}{
\begin{tikzpicture}[scale=1]
\fill[fill=medgrey] (0.5,0) circle (0.5);
\draw [<->] (-1.2,0) -- (1.2,0);
\draw [<-] (0,-1.2) -- (0,1.);
\draw [dashed] (0,0) circle (1);
\draw (1,-0pt) node [above right] {$1$};
\filldraw (1,0) circle ({0.6*1.5/1pt});
\def\x{0.15}
\draw [->] (0.1,.7) -- (\x,{sqrt((1/2)^2-(\x-(1/2))^2)});
\draw (0.1,.7) node [above, fill=white] {$\cG\left(J_{\alpha \cM}\right)$};
\end{tikzpicture}}
&\!\!\!\!\!\!\!\!\!\!\!
&$\stackrel{2z-1}{\longrightarrow}$&
&\!\!\!\!\!\!\!\!\!\!\!\!\!\!\!\!\!\!\!\!\!\!\!\!\!\!\!\!\!\!\!\!\!\!\!
\raisebox{-.5\height}{
\begin{tikzpicture}[scale=1]
\fill[fill=medgrey] (0,0) circle (1);
\draw [<->] (-1.2,0) -- (1.2,0);
\draw [<->] (0,-1.2) -- (0,1.2);
\draw [dashed] (0,0) circle (1);
\draw (1,-0pt) node [above right] {$1$};
\filldraw (1,0) circle ({0.6*1.5/1pt});
\def\x{-0.25}
\draw [->] (-1.4,1.2) -- (\x,{sqrt(1-(\x)^2)});
\draw (-1.6,1.1) node[above]  {$\cG\left(2J_{\alpha \cM}-I\right)$};
\end{tikzpicture}}
\end{tabular}
\end{center}
\vspace{0.1in}
Since  $\cL_1$ is SRG-full, Theorem~\ref{thm:srg-full} implies $(2J_{\alpha A}-I)$ is nonexpansive.
By the same reasoning, $(2J_{\alpha B}-I)$ is nonexpansive, and, since the composition of nonexpansive operators is nonexpansive, $(2J_{\alpha A}-I)(2J_{\alpha B}-I)$ is nonexpansive.
So \eqref{eq:dr} is a fixed-point iteration with a $1/2$-averaged operator, and the iteration converges by the Krasnosel'ski\u{\i}--Mann theorem.
\qed\end{proof}

When we have further assumptions, we can provide a stronger rate of convergence.

\begin{fact}
\label{fact:tight_example}
Assume $A$ or $B$ is $\mu$-strongly monotone and $\beta$-cocoercive with $0<\mu<1/\beta<\infty$.
For $\alpha>0$, the iterates of \eqref{eq:dr} converge exponentially to the fixed point $z^\star$ with rate 
\[
\|z^k-z^\star\|\le
\left(
\frac{1}{2}+\frac{1}{2}
\sqrt{1-\frac{4\alpha\mu}{1+2\alpha\mu+\alpha^2\mu/\beta}}\right)^k
\|z^0-z^\star\|.
\]
\end{fact}
\begin{proof}
If $S_1$ is $R_1$-Lipschitz continuous and $S_2$ is $R_2$-Lipschitz continuous, then $S_1S_2$ is $(R_1R_2)$-Lipschitz continuous.
If $S$ is $R$-Lipschitz continuous, then $\frac{1}{2}I+\frac{1}{2}S$ is $\left(\frac{1}{2}+\frac{R}{2}\right)$-Lipschitz continuous.
The result follows from these observations and Fact~\ref{prop:refl-sm-coco}, which we state and prove below.
\qed\end{proof}

\begin{fact}[Theorem 7.2 \cite{giselsson20152}]
\label{prop:refl-sm-coco}
Let $0<\mu<1/\beta<\infty$ and $\alpha\in(0,\infty)$.
If $\cA= \cM_\mu\cap \cC_\beta$, then $2J_{\alpha \cA}-I\subseteq \cL_R$ for
\[
R=\sqrt{1-\frac{4\alpha\mu}{1+2\alpha\mu+\alpha^2\mu/\beta}}.
\]
This result is tight in the sense that $2J_{\alpha \cA}-I\nsubseteq \cL_R$ for any smaller value of $R$.
\end{fact}
\emph{Proof outline}\,
We quickly outline the geometric insight while deferring the full proof with precise geometric arguments to the Section~\ref{s:geo-proofs} in the appendix.
%By Proposition~\ref{prop:monotone-srg} and Theorems~\ref{thm:srg-scaling-translation} and \ref{thm:srg-inversion}, 
We have the geometry
\begin{center}
\begin{tabular}{ccccc}
\raisebox{-.5\height}{
\begin{tikzpicture}[scale=1.8]
\def\m{0.5};
\def\b{.8};
\begin{scope}
\clip ({1+\m},-1.2) rectangle (-1.2,1.2);
\draw[dashed] ({1+1/2/\b},0) circle ({1/2/\b});
\end{scope}
\draw [dashed] ({1+\m},{sqrt((1/2/\b)^2-(1/2/\b-\m)^2)}) -- ({1+\m},1.2);
\draw [dashed] ({1+\m},{-sqrt((1/2/\b)^2-(1/2/\b-\m)^2)}) -- ({1+\m},-1.2);

\begin{scope}
\clip ({1+\m},-1.2) rectangle (2.5,1.2);
\fill[fill=medgrey] ({1+1/2/\b},0) circle ({1/2/\b});
%\draw[line width=1.0pt] ({1+1/2/\b},0) circle ({1/2/\b});
% \fill[fill=medgrey] (1.2,-1.2) rectangle (2,1.2);
% \draw [line width=1.0pt] (1.2,-1.2) -- (1.2,1.2);
% \fill[fill=black!60] (1.3,0) circle (0.3);
% \draw[line width=1.0pt] (1.3,0) circle (0.3);
\end{scope}
\begin{scope}
\clip ({1+1/2/\b},0) circle ({1/2/\b});
%\draw [line width=1.0pt] ({1+\m},-1.2) -- ({1+\m},1.2);
\end{scope}
\filldraw ({1+\m},0)circle ({0.6*1.5/1.8pt});
\filldraw ({1+1/\b},0)  circle ({0.6*1.5/1.8pt});

\draw ({1+1/\b},0.0) node [above left] {$\scriptstyle1+\frac{\alpha}{\beta}$};
\draw [<->] (-1.2,0) -- (2.5,0);
\draw [<->] (0,-1.2) -- (0,1.2);
\draw [dashed] (0,0) circle (1);
\draw ({1+\m},0.05) node [above left,fill=white] {$\scriptstyle 1+\alpha\mu$};
\draw (1,-0pt) node [below left] {$\scriptstyle 1$};
\filldraw (1,0)circle ({0.6*1.5/1.8pt});
%\draw (.2,-.8) node [below,fill=white] {$\cG\left(I+\alpha(\cM_\mu\cap \cC_\beta)\right)$};
\draw (.1,-.6) node [below,fill=white] {$\scriptstyle\cG\left(I+\alpha\cA\right)$};
\draw [->] (.1,-.6) -- ({1+\m},-0.2);
\end{tikzpicture}}
&\!\!\!\!\!\!\!
$\stackrel{\bar{z}^{-1}}{\longrightarrow}$
&\!\!\!\!\!\!\!\!\!\!
\raisebox{-.5\height}{
\begin{tikzpicture}[scale=1.8]
\def\m{0.5};
\def\b{.8};

\begin{scope}
\clip ({1/2*(-\m/(1+\m)+sqrt(1/(1 + \m)/(1 + \m) + (4*\b*\m*(-1 + \b*\m))/(\b + \m + 2*\b*\m)/(\b + \m + 2*\b*\m)))+1/2},-1.2)  rectangle (-1.2,1.2);
\draw[dashed] ({1/(1+\m)/2},0) circle ({1/(1+\m)/2});
\end{scope}

\begin{scope}
\clip ({1/2*(-\m/(1+\m)+sqrt(1/(1 + \m)/(1 + \m) + (4*\b*\m*(-1 + \b*\m))/(\b + \m + 2*\b*\m)/(\b + \m + 2*\b*\m)))+1/2},-1.2)  rectangle (1.2,1.2);
\draw[dashed] ({(1+2*\b)/(2+2*\b)},0) circle ({(1)/(2+2*\b)});
\end{scope}

\begin{scope}
\clip ({1/(1+\m)/2},0) circle ({1/(1+\m)/2});
\fill[fill=medgrey] ({(1+2*\b)/(2+2*\b)},0) circle ({(1)/(2+2*\b)});
%\draw[line width=1.0pt] ({(1+2*\b)/(2+2*\b)},0) circle ({(1)/(2+2*\b)});
\end{scope}
\begin{scope}
\clip ({(1+2*\b)/(2+2*\b)},0) circle ({(1)/(2+2*\b)});
%\draw[line width=1.0pt] ({1/(1+\m)/2},0) circle ({1/(1+\m)/2});
\end{scope}
\draw [<->] (-1.2,0) -- (1.2,0);
\draw [<->] (0,-1.2) -- (0,1.2);
\draw [dashed] (0,0) circle (1);
\draw (1,-0pt) node [below right] {$\scriptstyle 1$};
\filldraw (1,0)circle ({0.6*1.5/1.8pt});

\filldraw ({1/(1+\m)},0)circle ({0.6*1.5/1.8pt});
\draw[->] (1.05,.4)--({1/(1+\m)},0);
\draw (1.15,.4) node [above] {$\scriptstyle\frac{1}{1+\alpha\mu}$};

\filldraw ({1/(1+1/\b)},0)circle ({0.6*1.5/1.8pt});
\draw[->] (.3,.5)--({1/(1+1/\b)},0);
\draw (.3,.45) node [above] {$\scriptstyle\frac{1}{1+\alpha/\beta}$};

\def\x{.5}
%\draw (.1,-.7) node [below,fill=white] {$\cG\left(\left(I+\alpha(\cM_\mu\cap \cC_\beta)\right)^{-1}\right)$};
%\draw (.1,-.7) node [below,fill=white] {$\cG\left(\left(I+\alpha\cA\right)^{-1}\right)$};
\draw (.1,-.6) node [below,fill=white] {$\scriptstyle\cG\left(J_{\alpha \cA}\right)$};
\draw [->] (.1,-.7) -- (\x,{-sqrt(((1)/(2+2*\b))^2-(\x-(1+2*\b)/(2+2*\b))^2)});
\end{tikzpicture}}\\
&\!\!\!\!\!\!\!\!\!\!\!\!\!\!\!\!
$\stackrel{2z-1}{\longrightarrow}$
&\!\!\!\!\!\!\!\!\!\!\!\!\!\!\!\!\!\!\!\!\!\!\!\!\!\!\!\!\!
\raisebox{-.5\height}{
\begin{tikzpicture}[scale=1.8]
\def\m{0.5};
\def\b{.8};
%\fill [fill=lightgrey] (0,0) circle ({sqrt((\b + \m - 2*\b*\m)/   (\b + \m + 2*\b*\m))});

\fill [fill=lightgrey] (0,0) circle ({sqrt((\b + \m - 2*\b*\m)/   (\b + \m + 2*\b*\m))});

\begin{scope}
\clip ({-\m/(1+\m)+sqrt(1/(1 + \m)/(1 + \m) + (4*\b*\m*(-1 + \b*\m))/(\b + \m + 2*\b*\m)/(\b + \m + 2*\b*\m))},-1.2)  rectangle (-1.2,1.2);
\draw[dashed] ({-\m/(1+\m)},0) circle ({1-\m/(1+\m)});
\end{scope}

\begin{scope}
\clip ({-\m/(1+\m)+sqrt(1/(1 + \m)/(1 + \m) + (4*\b*\m*(-1 + \b*\m))/(\b + \m + 2*\b*\m)/(\b + \m + 2*\b*\m))},-1.2)  rectangle (1.2,1.2);
\draw[dashed] ({1/(1+1/\b)},0) circle ({1-1/(1+1/\b)});
\end{scope}

\begin{scope}
\clip ({-\m/(1+\m)},0) circle ({1-\m/(1+\m)});
\fill[fill=medgrey] ({1/(1+1/\b)},0) circle ({1-1/(1+1/\b)});
%\draw[line width=1.0pt] ({1/(1+1/\b)},0) circle ({1-1/(1+1/\b)});
\end{scope}
\begin{scope}
\clip ({1/(1+1/\b)},0) circle ({1-1/(1+1/\b)});
%\draw[line width=1.0pt] ({-\m/(1+\m)},0) circle ({1-\m/(1+\m)});
\end{scope}
\draw [<->] (-1.2,0) -- (1.2,0);
\draw [<->] (0,-1.2) -- (0,1.2);
\draw [dashed] (0,0) circle (1);
\draw (1,-0pt) node [below right] {$\scriptstyle1$};
\filldraw (1,0)circle ({0.6*1.5/1.8pt});
\filldraw ({(1-\m)/(1+\m)},0)circle ({0.6*1.5/1.8pt});
\draw ({(1-\m)/(1+\m)-0.1},0) node [above right] {$\scriptstyle\frac{1-\alpha\mu}{1+\alpha\mu}$};
\filldraw ({(1-1/\b)/(1+1/\b)},0)circle ({0.6*1.5/1.8pt});
\draw ({(1-1/\b)/(1+1/\b)+0.07},0) node [above left] {$\scriptstyle\frac{\beta-\alpha}{\beta+\alpha}$};

\def\x{.05}
%\draw (-.1,-.9) node [below,fill=white] {$\cG\left(2\left(I+\alpha(\cM_\mu\cap \cC_\beta)\right)^{-1}-I\right)$};
\draw (-.05,-.6) node [below,fill=white] {$\scriptstyle\cG\left(2J_{\alpha \cA}-I\right)$};
\draw [->] (-.05,-.6) -- (\x,{-sqrt(((2)/(2+2*\b))^2-(\x-2*(1+2*\b)/(2+2*\b)+1)^2)});

%
%\def\y{.3}
%\draw [->] (.4,.7) -- (\y,{sqrt((sqrt((\b + \m - 2*\b*\m)/   (\b + \m + 2*\b*\m)))^2-(\y)^2)});
%\draw (.4,.7) node [above,fill=white] {$\cG\left(\cL\left(\sqrt{1-\frac{4\alpha\mu}{1+2\alpha\mu+\alpha^2\mu/\beta}}\right)\right)$};

\draw (-1.2,.9) node [fill=white] {$R=\sqrt{1-\frac{4\alpha\mu}{1+2\alpha\mu+\alpha^2\mu/\beta}}$};

\def\y{-.30}
\draw [->] (-1.25,.7) -- (\y,{sqrt((sqrt((\b + \m - 2*\b*\m)/   (\b + \m + 2*\b*\m)))^2-(\y)^2)});
\end{tikzpicture}}
\end{tabular}
\end{center}
The radius $R$ is obtained with Stewart's theorem \cite{stewart}.
\qed%\end{proof}

As a special case, consider the optimization problem
\[
\begin{array}{ll}
\underset{x\in \cH}{\mbox{minimize}}& f(x)+g(x),
\end{array}
\]
where $f$ and $g$ are functions (not necessarily differentiable) and a minimizer exists.
Then \eqref{eq:dr} with $A=\partial f$ and $B=\partial g$ can be written as
\begin{align*}
x^{k+1/2}&=J_{\alpha \partial g}(z^k)\\
x^{k+1}&=J_{\alpha \partial f}(2x^{k+1/2}-z^k)\\
z^{k+1}&=z^k+x^{k+1}-x^{k+1/2},
\end{align*}
where $\alpha>0$ and $z^0\in \hilbert$ is a starting point. 
As an aside, the popular method ADMM is equivalent to this instance of Douglas--Rachford splitting \cite{gabay1983}.
% Equivalently, we can write the iteration as
% \begin{equation}
% z^{k+1}=\left(\tfrac{1}{2}I+\tfrac{1}{2}(2J_{\alpha A}-I)(2J_{\alpha B}-I)\right)z^k,\tag{DR}
% \label{eq:dr}
% \end{equation}

% The Facts~\ref{prop:refl-cvx} and \ref{prop:refl-sm-coco} can be used to prove linear convergence of the Douglas--Rachford splitting method \cite{lions1979}.
\begin{fact}
Assume $f$ is $\mu$-strongly convex and $L$-smooth with $0<\mu<L<\infty$.
Assume $g$ is convex, lower semi-continuous, and proper.
For $\alpha>0$, the iterates of \eqref{eq:dr} converge exponentially to the fixed point $z^\star$ with rate 
\[
\|z^k-z^\star\|\le
\left(
\frac{1}{2}+\frac{1}{2}
\max\left\{
\left|\frac{1-\alpha\mu}{1+\alpha \mu}\right|,\left|\frac{1-\alpha L}{1+\alpha L}\right|
\right\}\right)^k
\|z^0-z^\star\|
\]
\end{fact}
\begin{proof}
If $S_1$ is $R_1$-Lipschitz continuous and $S_2$ is $R_2$-Lipschitz continuous, then $S_1S_2$ is $(R_1R_2)$-Lipschitz continuous.
If $S$ is $R$-Lipschitz continuous, then $\frac{1}{2}I+\frac{1}{2}S$ is $\left(\frac{1}{2}+\frac{R}{2}\right)$-Lipschitz continuous.
The result follows from these observations and Fact~\ref{prop:refl-sm-coco}, which we state and prove below.
\qed\end{proof}

\begin{fact}[Theorem 1 \cite{giselsson2017linear}]
\label{prop:refl-cvx}
Let $0<\mu<L<\infty$ and $\alpha\in(0,\infty)$.
If $\cA= \partial \cF_{\mu,L}$, then 
$2J_{\alpha \cA}-I\subseteq \cL_R$ for
\[
R=
\max\left\{
\left|\frac{1-\alpha\mu}{1+\alpha \mu}\right|,\left|\frac{1-\alpha L}{1+\alpha L}\right|
\right\}.
\]
This result is tight in the sense that $2J_{\alpha \cA}-I\nsubseteq \cL_R$ for any smaller value of $R$.
\end{fact}
\begin{proof}
By Proposition~\ref{proposition:cvx-srg-first} and Theorems~\ref{thm:srg-scaling-translation} and \ref{thm:srg-inversion}, we have the geometry 
\begin{center}
\begin{tabular}{ccccc}
\!\!\!\!\!\!
\raisebox{-.5\height}{
\begin{tikzpicture}[scale=1.8]
\def\m{0.3};
\def\L{1.3};

\draw (1.4,.6) node [above] {$\scriptstyle\cG\left(I+\alpha \cA\right)$};
\def\t{120};
\draw[->] (1.4,.7)--({(\L-\m)/2*cos(\t)+(2+\m+\L)/2},{(\L-\m)/2*sin(\t)});

\fill[fill=medgrey] ({(2+\m+\L)/2},0) circle ({(\L-\m)/2});
\begin{scope}
\clip (-0.2,-1.2) rectangle (2.5,1.2);
\draw [dashed] (0,0) circle (1);
\end{scope}

\draw [<->] (-0.2,0) -- (2.5,0);
\draw [<->] (0,-1.2) -- (0,1.2);
\draw (1,0)  node [above left] {$\scriptstyle 1$};
\filldraw (1,0) circle ({0.6*1.5/1.8pt});
\filldraw ({1+\L},0) circle ({0.6*1.5/1.8pt});
\draw ({1+\L},0)  node [above left] {$\scriptstyle 1+\alpha L$};
\draw ({1+\m},0)  node [below right] {$\scriptstyle 1+\alpha\mu$};
\filldraw ({1+\m},0) circle ({0.6*1.5/1.8pt});
\end{tikzpicture}
}
&\!\!\!\!\!\!
$\stackrel{\bar{z}^{-1}}{\longrightarrow}$
&\!\!\!\!\!\!
\raisebox{-.5\height}{
\begin{tikzpicture}[scale=1.8]
\def\m{0.3};
\def\L{1.3};
\draw [<->] (0,-1.2) -- (0,1.2);
\draw [dashed] (0,0) circle (1);
\draw ({1/(1+\m)},0)  node [above right,fill=white] {$\scriptstyle\frac{1}{1+\alpha \mu}$};
\draw ({1/(1+\L)+.05},0)  node [above left, fill=white] {$\scriptstyle\frac{1}{1+\alpha L}$};

\draw (0.5,.6) node [above,fill=white] {$\scriptstyle\cG\left(J_{\alpha \cA}\right)$};
\def\t{90};
\draw[->] (.5,.6)--({1/(1+\m)-1/(1+\L))/2*cos(\t)+(1/(1+\m)+1/(1+\L))/2},{1/(1+\m)-1/(1+\L))/2*sin(\t)});

\fill [fill=medgrey] ({(1/(1+\m)+1/(1+\L))/2},0) circle ({(1/(1+\m)-1/(1+\L))/2});
\draw [<->] (-1.2,0) -- (1.2,0);

\filldraw ({1/(1+\L)},0) circle ({0.6*1.5/1.8pt});
\filldraw ({1/(1+\m)},0) circle ({0.6*1.5/1.8pt});
\draw (1,0)  node [below right] {$1$};
\filldraw (1,0) circle ({0.6*1.5/1.8pt});
\end{tikzpicture}}\\
&\!\!\!\!\!\!
$\stackrel{2z-1}{\longrightarrow}$
&\!\!\!\!\!\!\!\!\!\!\!\!\!\!\!\!\!\!\!\!\!\!\!\!\!\!\!\!\!\!\!\!\!
\!\!\!\!\!\!
\raisebox{-.5\height}{
\begin{tikzpicture}[scale=1.8]
\def\m{0.3};
\def\L{1.3};
%\draw [dashed] (0,0) circle (1);
 
\draw (0.7,.6) node [above,fill=white] {$\scriptstyle\cG\left(2J_{\alpha \cA}-I\right)$};
\draw (-1.3,-.65) node [below,fill=white] {$\scriptstyle R=\max\left\{\left|\frac{1-\alpha\mu}{1+\alpha \mu}\right|,\left|\frac{1-\alpha L}{1+\alpha L}\right|\right\}$};
\draw ({2/(1+\m)-1-0.05},0)  node [above right, fill=white] {$\scriptstyle\frac{1-\alpha \mu}{1+\alpha \mu}$};

\fill [fill=lightgrey] (0,0) circle ({(1-\m)/(1+\m)});

\draw[->] (-.45,.45)--({2/(1+\L)-1},0);

\draw (-.4,.4) node [above left] {$\scriptstyle\frac{1-\alpha L}{1+\alpha L}$};

\fill [fill=medgrey] ({(1/(1+\m)+1/(1+\L))-1},0) circle ({(1/(1+\m)-1/(1+\L))});
\draw [<->] (0,-1.2) -- (0,1.2);
\draw [<->] (-1.2,0) -- (1.2,0);
\filldraw ({2/(1+\L)-1},0) circle ({0.6*1.5/1.8pt});
\filldraw ({2/(1+\m)-1},0) circle ({0.6*1.5/1.8pt});
%\draw (1,0)  node [below right] {$1$};
%\filldraw (1,0) circle ({0.6*1.5/1.8pt});

\def\t{80}
\draw[->] (0.7,.6)--({(1/(1+\m)+1/(1+\L))-1+(1/(1+\m)-1/(1+\L))*cos(\t)},{(1/(1+\m)-1/(1+\L))*sin(\t)});
%\draw (0.4,.7) node [above, fill=white] {$\cG\left(2(I+\alpha \partial \cF_{\mu,L})^{-1}-I\right)$};

\def\s{200}
\draw[->] (-1.05,-.45)-- ({(1-\m)/(1+\m)*cos(\s)},{(1-\m)/(1+\m)*sin(\s)});

\draw (-1.05,-.45) node [left] {$\scriptstyle\cG\left(\cL_R\right)$};

\end{tikzpicture}
}
\end{tabular}
\end{center}
The containment holds for $R$ and fails for smaller $R$.
Since $\cL_R$ is SRG-full by Theorem~\ref{thm:srg-full}, the containment of the SRG in $\ecomplex$ equivalent to the containment of the class.
\qed\end{proof}

%%%%%%%%%%%%%%%%%%%%%%%%%%%%%%%%%%%%%%%%
\subsection{Sum of operators} \label{s:srg-sum}
%%%%%%%%%%%%%%%%%%%%%%%%%%%%%%%%%%%%%%%%
\begin{wrapfigure}[8]{r}{0.32\textwidth}
%\begin{figure}
\begin{center}
\vspace{-0.5in}
\begin{tabular}{c}
\raisebox{-.5\height}{
\begin{tikzpicture}[scale=1.2]
\fill[fill=medgrey] (0.85,0) circle (0.55);
\draw [<->] (-0.2,0) -- (1.7,0);
\draw [<->] (0,-.8) -- (0,.8);

\filldraw (1,.4) circle ({0.6*1.5/1.2pt});
\filldraw (1,-.4) circle ({0.6*1.5/1.2pt});
\draw [line width=1.0] (1,.4) -- (1,-.4);

\draw (1,.4) node [left] {$z$};
\draw (1,-.4) node [left] {$\bar{z}$};
\end{tikzpicture}
}
\end{tabular}
\end{center}
\vspace{-0.1in}
\caption{The chord property.}
\label{fig:chord_property}
%\end{figure}
\end{wrapfigure}
Given $z,w\in \complex$, define the \emph{line segment} between $z$ and $w$ as
\[
[z,w]=\{\theta z+(1-\theta)w\,|\,\theta\in[0,1]\}.
\]
We say an SRG-full class $\cA$
satisfies the \emph{chord property} if\\
$z\in \cG(\cA)\backslash\{\infty\}$ implies $[z,\bar{z}]\subseteq \cG(\cA)$.
See Figure~\ref{fig:chord_property}.

\begin{theorem}
\label{thm:srg-sum}
Let $\cA$ and $\cB$ be SRG-full classes such that $\infty\notin \cG(A)$ and $\infty\notin \cG(B)$.
Then
\[
\cG(\cA+\cB)\supseteq \cG(\cA)+\cG(\cB).
\]
If $\cA$ or $\cB$ furthermore satisfies the chord property, then
\[
\cG(\cA+\cB)=\cG(\cA)+\cG(\cB).
\]
\end{theorem}
Although we do not pursue this, one can generalize Theorem~\ref{thm:srg-sum} to allow $\infty$ by excluding the following exception:
if $\emptyset=\cG(\cA)$ and $\infty\in \cG(\cB)$, then $\{\infty\}=\cG(\cA+\cB)$.
\begin{proof}
We first show $\cG(\cA+\cB)\supseteq \cG(\cA)+\cG(\cB)$.
Assume $\cG(\cA)\ne \emptyset$ and $\cG(\cB)\ne \emptyset$ as otherwise there is nothing to show.
Let $z\in \cG(\cA)$ and $w\in \cG(\cB)$
and let $A_z$ and $A_w$ be their corresponding operators as defined in Lemma~\ref{lem:complex-srg}.
Then it is straightforward to see that
$A_z+A_w$ corresponds to complex multiplication
with respect to $(z+w)$,
and $z+w\in \cG(A_z+A_w)\subseteq \cG(\cA+\cB)$.

Next, we show $\cG(\cA+\cB)\subseteq \cG(\cA)+\cG(\cB)$.
Consider the case $\cG(\cA)\ne \emptyset$ and $\cG(\cB)\ne \emptyset$.
Without loss of generality, assume it is $\cA$ that satisfies the chord property.
Consider $A+B\in \cA+\cB$ such that $A\in \cA$ and $B\in \cB$ .
Consider $(x,u_A+u_B),(y,v_A+v_B)\in A+B$ such that $x\ne y$, $(x,u_A),(y,v_A)\in A$, and $(x,u_B),(y,v_B)\in B$.
Define
\begin{align*}
z_A&=\frac{\|u_A-v_A\|}{\|x-y\|}
\exp\left[
i\angle (u_A-v_A,x-y)
\right]\in\cG(A)\\
z_B&=\frac{\|u_B-v_B\|}{\|x-y\|}
\exp\left[
i\angle (u_B-v_B,x-y)
\right]\in\cG(B)\\
z&=\frac{\|u_A+u_B-v_A-v_B\|}{\|x-y\|}
\exp\left[
i\angle (u_A+u_B-v_A-v_B,x-y)
\right]\in\cG(A+B).
\end{align*}
(Note that $\Im z_A,\Im z_B,\Im z\ge 0$.)
Since
\begin{gather*}
\Re z_A=\frac{\langle u_A-v_A,x-y\rangle}{\|x-y\|^2},\qquad
\Re z_B=\frac{\langle u_B-v_B,x-y\rangle}{\|x-y\|^2},
\\
\Re z=\frac{\langle (u_A+u_B)-(v_A+v_B),x-y\rangle}{\|x-y\|^2},
\end{gather*}
we have $\Re z = \Re z_A + \Re z_B$.
Using \eqref{eq:srg-alternative2} and the triangle inequality, we have
\begin{align*}
\Im z&=\frac{\|P_{\{x-y\}^\perp}(u_A+u_B-v_A-v_B)\|}{\|x-y\|}\\
&\le \frac{\|P_{\{x-y\}^\perp}(u_A-v_A)\|+\|P_{\{x-y\}^\perp}(u_B-v_B)\|}{\|x-y\|}\\
&=\Im z_A+\Im z_B,
\end{align*}
and using the reverse triangle inequality, we have $\Im z\ge -\Im z_A+\Im z_B$.
%(Remember, we defined $z_A$, $z_B$ and $z$ such that $\Im z_A,\Im z_B,\Im z\ge 0$.)
%Repeating this argument, we get $-|\Im z|\ge -|\Im z_A|-|\Im z_B|$ and $-|\Im z|\le -|\Im z_A|+|\Im z_B|$.
Together, we conclude
\[
-\Im z_A+\Im z_B\le \Im z\le \Im z_A+\Im z_B
\]
and 
\[
z\in [z_A,\overline{z_A}]+z_B,\qquad
\overline{z}\in [z_A,\overline{z_A}]+\overline{z_B}.
\]
This shows
\begin{align*}
\cG (\cA+\cB)&\subseteq
\left\{
w_A+z_B
\,|\,
w_A\in \left[z_A,\overline{z_A}\right],\,
z_A\in \cG (\cA),\,z_B\in \cG (\cB)
\right\}\\
% &=
% \left\{
% w_A+z_B
% \,|\,
% w_A\in \left[z_A,\overline{z_A}\right],\,
% w_A\in \cG (\cA),\,
% z_A\in \cG (\cA),\,z_B\in \cG (\cB)
% \right\}\\
&=
\left\{
w_A+z_B
\,|\,
w_A\in \cG (\cA),\,z_B\in \cG (\cB)
\right\}
=\cG (\cA)+\cG (\cB),
\end{align*}
where the equality follows from the chord property.

%Consider the case $\cG(\cA)\ne \emptyset$ and $\cG(\cB)\ne \emptyset$.

%Considering all cases, we have
%\[
%z\in [z_A,\overline{z_A}]+z_B.
%\]
%and by symmetry, 
%\[
%\overline{z}\in [z_A,\overline{z_A}]+\overline{z_B}.
%\]
%
%By the same reasoning with the reverse triangle inequality, we have $|\Im z|\ge |\Im z_A|-|\Im z_B|$ and $|\Im z|\ge -|\Im z_A|+|\Im z_B|$.
%This tells us that
%\[
%|\Im z|\in [-|\Im z_A|,|\Im z_A|]+|\Im z_B|
%\]
%XXX
%XXX
%Therefore, if $\Im z_A\le \Im z_B$, then
%\[
%z\in [z_A,\overline{z_A}]+z_B.
%\]
%Likewise, if $\Im z_A\ge \Im z_B$, then
%\[
%z\in z_A+[z_B,\overline{z_B}]
%\subseteq 
% [z_A,\overline{z_A}]+z_B.
%\]
Now, consider the case $\cG(\cA)= \emptyset$ or $\cG(\cB)= \emptyset$ (or both).
(We also discuss this degenerate case in Section~\ref{ss:srg-full-appendix}).
Assume $\cG(\cA)= \emptyset$ without loss of generality and let $A\in \cA$ and $B\in \cB$.
Then $\dom{A}$ is empty or a singleton, and if $\{x\}=\dom{A}$ then $Ax$ is a singleton.
Therefore $\dom{A+B}\subseteq \dom{A}$ is empty or a singleton,
 and if $\{x\}=\dom{A}$ then $(A+B)x$ is empty or a singleton since $B$ is single-valued.
Therefore, $\cG(A+B)=\emptyset$ and we conclude $\cG(\cA+\cB)=\emptyset$.
\qed\end{proof}

\subsection{Composition of operators}
\label{s:srg-composition}
Given $z\in \complex$, define the \emph{right-hand arc} between $z$ and $\bar{z}$ as
\[
\rarc(z,\bar{z})
=\left\{re^{i(1-2\theta)\varphi}\,\Big|\,
z=re^{i\varphi},\,
\varphi\in(-\pi,\pi],\,\theta\in[0,1],\,r\ge 0
\right\}
\]
\begin{wrapfigure}[12]{r}{0.48\textwidth}
%\begin{figure}
\begin{center}
\vspace{-0.3in}
\begin{tabular}{cc}
\!\!\!\!\!\!\!\!\!\!\!\!\!\!\!
\raisebox{-.5\height}{
\begin{tikzpicture}[scale=1.5]
\def\x{0.5}
\fill[fill=medgrey] (-.4,0) circle (0.8);
\draw [<->] (-1.4,0) -- (0.7,0);
\draw [<->] (0,-1.) -- (0,1.);

\filldraw (.1,\x) circle[radius={0.6*1.5/1.5pt}];
\filldraw (.1,{-\x}) circle[radius={0.6*1.5/1.5pt}];
\begin{scope}
\clip (.1,-1.) rectangle (0.7,1.);
\draw [dashed] (0,0) circle ({sqrt(.1^2+\x^2)});
\end{scope}
\begin{scope}
\clip (-1.4,-1.) rectangle (0.1,1.);
\draw [line width=1.0] (0,0) circle (({sqrt(.1^2+\x^2)});
\end{scope}

\draw (.1,{\x}) node [above right] {$z$};
\draw (.1,{-\x}) node [below right] {$\bar{z}$};

\def\w{-.3}
\draw [->] (-0.75,-.8) -- (\w,{-sqrt(.1^2+\x^2-(\w)^2)});
\draw (-0.75,-.7)  node [below] {$\larc(z,\bar{z})$};
\end{tikzpicture}
}
&
\!\!\!\!\!\!\!\!\!\!\!\!\!\!\!\!\!\!
\raisebox{-.5\height}{
\begin{tikzpicture}[scale=1]
\fill[fill=medgrey] (0.85,0) circle (0.55);
\draw [<->] (-1.2,0) -- (1.6,0);
\draw [<->] (0,-1.4) -- (0,1.4);

\filldraw (1,.4) circle ({0.6*1.5/1pt});
\filldraw (1,-.4) circle ({0.6*1.5/1pt});
\begin{scope}
%\clip (-.3,-1.2) rectangle (1.6,1.2);
\draw [dashed] (0,0) circle ({sqrt(1+.4^2)});
\end{scope}
\begin{scope}
\clip (1,-1.2) rectangle (1.6,1.2);
\draw [line width=1.0] (0,0) circle ({sqrt(1+.4^2)});
\end{scope}

\draw (1,.4) node [left] {$z$};
\draw (1,-.4) node [left] {$\bar{z}$};

\def\w{1.05}
\draw [->] (-0.05,-.6) -- (\w,{sqrt(1+.4^2-\w^2)});
\draw (-0.05,-.6)  node [fill=white, left] {$\rarc(z,\bar{z})$};
\draw (-0,-.93)  node [below] {\phantom{$\larc(z,\bar{z})$}};
\end{tikzpicture}
}
\!\!\!\!\!\!\!\!\!
\end{tabular}
\vspace{-0.1in}
\end{center}
\caption{Left and right-arc properties.}
\label{fig:arc-properties}
%\end{figure}
\end{wrapfigure}
and the \emph{left-hand arc} as% between $z$ and $\bar{z}$ as
\[
\larc(z,\bar{z})
=-\rarc(-z,-\bar{z}).
%=\left\{re^{i\left(\theta|\varphi|+(1-\theta)(2\pi-|\varphi|)\right)}\,\Big|\,z=re^{i\varphi},\,|\varphi|\in[0,\pi],\,\theta\in[0,1],\,r\ge 0\right\}.
\]
We say an SRG-full class $\cA$ respectively satisfies the \emph{left-arc property} and \emph{right-arc property}
if $z\in \cG(\cA)\backslash\{\infty\}$
implies
$\larc{(z,\bar{z})}\subseteq \cG(\cA)$
and
$\rarc{(z,\bar{z})}\subseteq \cG(\cA)$, respectively.
We say $\cA$ satisfies \emph{an} arc property if the left or right-arc property is satisfied.
See Figure~\ref{fig:arc-properties}.

\begin{theorem}
\label{thm:composition-srg}
Let $\cA$ and $\cB$ be SRG-full classes such that $\infty\notin \cG(\cA)$, $\emptyset\ne\cG(\cA)$, $\infty\notin \cG(\cB)$, and $\emptyset\ne \cG(\cB)$.
Then
\[
\cG(\cA\cB)\supseteq \cG(\cA)\cG(\cB).
\]
If $\cA$ or $\cB$ furthermore satisfies a left or right arc property, then
\[
\cG(\cA\cB)=\cG(\cB\cA)=\cG(\cA)\cG(\cB).
\]
\end{theorem}
Although we do not pursue this, one can generalize Theorem~\ref{thm:composition-srg} to allow $\emptyset$ and $\infty$ by excluding the following exceptions:
if $\emptyset=\cG(\cA)$ and $\infty\in \cG(\cB)$, then $\{\infty\}=\cG(\cA\cB)$;
if $0\in \cG(\cA)$ and $\infty\in \cG(\cB)$, then $\infty\in\cG(\cA\cB)$;
if $\emptyset= \cG(\cA)$ and $0\in \cG(\cB)$, then $\{0\}=\cG(\cA\cB)$ and $\emptyset=\cG(\cB\cA)$.
\begin{proof}
We first show $\cG(\cA\cB)\supseteq \cG(\cA)\cG(\cB)$.
Assume $\cG(\cA)\ne \emptyset$ and $\cG(\cB)\ne \emptyset$ as otherwise there is nothing to show.
Let $z\in \cG(\cA)$ and $w\in \cG(\cB)$
and let $A_z$ and $A_w$ be their corresponding operators as defined in Lemma~\ref{lem:complex-srg}.
Then it is straightforward to see that
$A_zA_w$ corresponds to complex multiplication
with respect to $zw$,
and $zw\in \cG(A_zA_w)\subseteq \cG(\cA\cB)$.

Next, we show $\cG(\cA\cB)\subseteq \cG(\cA)\cG(\cB)$.
Let $A\in \cA$ and $B\in \cB$.
Consider $(u,s),(v,t)\in A$ and $(x,u),(y,v)\in B$, where $x\ne y$.
This implies $(x,s),(y,t)\in AB$.
Define
\[
z=\frac{\|s-t\|}{\|x-y\|}
\exp\left[
i\angle (s-t,x-y)
\right].
%=\frac{\|s-t\|}{\|u-v\|}\frac{\|u-v\|}{\|x-y\|}
%\exp\left[i\angle (s-t,x-y)\right]
\]
Consider the case $u=v$. Then $0\in \cG(\cB)$.
Moreover, $s=t$, since $A$ is single-valued (by the assumption $\infty\notin \cG(\cA)$), and $z=0$.
Therefore, $z=0\in \cG(\cA)\cG(\cB)$.

Next, consider the case $u\ne v$.
Define
\[
z_A=\frac{\|s-t\|}{\|u-v\|}
%\exp\left[i\angle (s-t,u-v)\right]
e^{i\varphi_A}
,\quad
z_B=\frac{\|u-v\|}{\|x-y\|}
%\exp\left[i\angle (u-v,x-y)\right]
e^{i\varphi_B}
,
\]
where $\varphi_A=\angle (s-t,u-v)$ and $\varphi_B= \angle (u-v,x-y)$.
Consider the case where $\cA$ satisfies the right-arc property.
Using the spherical triangle inequality (further discussed in the appendix)
we see that either $\varphi_A\ge \varphi_B$ and 
\begin{align*}
z&\in \frac{\|s-t\|}{\|u-v\|}
\frac{\|u-v\|}{\|x-y\|}
\exp\left[i[\varphi_A-\varphi_B,\varphi_A+\varphi_B]\right]\\
&\subseteq
\frac{\|s-t\|}{\|u-v\|}
\frac{\|u-v\|}{\|x-y\|}
\exp\left[i[\varphi_B-\varphi_A,\varphi_B+\varphi_A]\right]\\
&=
z_B
\rarc\left(
z_A,\overline{z_A}
\right)
\end{align*}
or $\varphi_A< \varphi_B$ and 
\begin{align*}
z&\in 
\frac{\|s-t\|}{\|u-v\|}
\frac{\|u-v\|}{\|x-y\|}
\exp\left[i[\varphi_B-\varphi_A,\varphi_B+\varphi_A]\right]\\
&=z_B
\rarc\left(
z_A,\overline{z_A}
\right).
\end{align*}
This gives us
\[
z\in 
\underbrace{
z_B}_{\in \cG(\cB)}
\underbrace{\rarc\left(
z_A,\overline{z_A}\right)
}_{\subseteq  \cG(\cA)}
\subseteq \cG(\cA)\cG(\cB).
\]
That $\bar{z}\in \cG(\cA)\cG(\cB)$ follows from the same argument.
That $z,\bar{z}\in \cG(\cA)\cG(\cB)$ when instead $\cB$ satisfies the right-arc property follows from the same argument.

Putting everything together, we conclude $\cG(\cA\cB)= \cG(\cA)\cG(\cB)$
when $\cA$ or $\cB$ satisfies the right-arc property.
When $\cA$ satisfies the left-arc property, $-\cA$ satisfies the right-arc property.
So 
\[
-\cG(\cA\cB)=
\cG(-\cA\cB)=\cG(-\cA)\cG(\cB)
-\cG(\cA)\cG(\cB)
\]
by Theorem~\ref{thm:srg-scaling-translation},
and we conclude $\cG(\cA\cB)=\cG(\cA)\cG(\cB)$.
When $\cB$ satisfies the left-arc property,  $\cB\circ(-I)$ satisfies the right-arc property.
So 
\[
-\cG(\cA\cB)=
\cG(\cA\cB\circ(-I))=\cG(\cA)\cG(\cB\circ(-I))
=
-\cG(\cA)\cG(\cB)
\]
by Theorem~\ref{thm:srg-scaling-translation},
and we conclude $\cG(\cA\cB)=\cG(\cA)\cG(\cB)$.
\qed\end{proof}

We cannot fully drop the arc property from the second part of Theorem~\ref{thm:composition-srg}.
Consider the SRG-full operator class $\cA$ represented by $h(a,b,c)=|a-b|+|c|$, which has $\cG(\cA)=\{\pm i\}$.
Linear operators on $\reals^3$ representing 90 degrees rotations are in $\cA$.
With this, one can show the strict containment $\cG(\cA\cA)=\{z\in \complex\,|\,|z|=1\}\supset \cG(\cA)\cG(\cA)$.

As a consequence of Theorem~\ref{thm:composition-srg}, the SRGs of operator classes commute under composition even though individual operators, in general, do not commute when an arc property is satisfied.
Several results in operator theory involving 2 operators exhibit previously unexplained symmetry.
The Ogura--Yamada--Combettes averagedness factor \cite{ogura2002,combettes2015},
the contraction factor of Giselsson \cite{giselsson20152}, the contraction factor of Moursi and Vandenberghe \cite{moursi2018douglas},
the contraction factor of Ryu, Taylor, Bergeling, and Giselsson \cite{OSPEP} are all symmetric in the assumptions of the two operators. 
Theorem~\ref{thm:composition-srg} shows that this symmetry is not a coincidence.

\subsubsection{Convergence analysis: alternating projections}
Consider the convex feasibility problem
\[
\mbox{find $x\in \cH$ \quad such that \quad }x\in C\cap D
\]
where $C\subseteq\cH$ and $D\subseteq\cH$ are nonempty closed convex sets and $C\cap D\ne \emptyset$. Consider the alternating projections method \cite[Theorem 13.7]{vonneumann1950}
\begin{equation}
x^{k+1}=P_CP_Dx^k,
\tag{AP}
\label{eq:ap}
\end{equation}
where $P_C$ and $P_D$ are projections onto $C$ and $D$ and $x_0\in \cH$ is a starting point.

\begin{fact}
The iterates of \eqref{eq:ap} converge in that $x^k\rightarrow x^\star$ weakly for some $x^\star\in C\cap D$.
\end{fact}
\begin{proof}
% The optimality condition [XXX] for $P_C(u)=\argmin_{z\in C}\|z-u\|_2$ is 
% \[
% \langle v-P_Cu,P_Cu-u\rangle \ge 0,
% \qquad \forall v\in C
% \]
% for any $u\in \cH$.
% Using this inequality twice,  we get
% \begin{align*}
% \langle P_Cy-P_Cx,P_Cx-x\rangle &\ge 0\\
% \langle P_Cx-P_Cy,P_Cy-y\rangle&\ge 0
% \end{align*}
% and adding these we get
% \[
% \langle P_Cx-P_Cy,x-y\rangle\ge \|P_Cx-P_Cy\|_2^2.
% \]
% Finally, we have
% \begin{align*}
% \|(2P_C-I)x-(2P_C-I)y\|_2^2&=\|2(P_Cx-P_Cy)-(x-y)\|_2^2\\
% &=4\|P_Cx-P_Cy\|_2^2-4\langle P_Cx-P_Cy,x-y\rangle+\|x-y\|_2^2\\
% &\le \|x-y\|_2^2,
% \end{align*}
% Since $2P_C-I$ is nonexpansive, $P_C$ is $1/2$-averaged.
% By the same reasoning, $P_D$ is also $1/2$-averaged.
By \cite[Proposition 4.16]{BCBook}, $P_C$ and $P_D$ are $1/2$-averaged.
By Fact~\ref{thm:fne-compose}, which we state and prove below,
$P_CP_D$ is $2/3$-averaged, and the iteration converges by the Krasnosel'ski\u{\i}--Mann theorem.
\qed\end{proof}

\begin{fact}[Propsition~4.42 \cite{BCBook}]
\label{thm:fne-compose}
Let $\cN_{1/2}$ be the class of firmly nonexpansive operators. Then
\[
\cN_{1/2}\cN_{1/2}\subset \cN_{2/3}.
\]
(containment is strict.)
Furthermore, 
\begin{center}
\begin{tabular}{c}
$\cG(\cN_{1/2}\cN_{1/2})=$
\raisebox{-.5\height}{
\begin{tikzpicture}[scale=1]
\def\w{.8};
\def\u{1.2};
%\draw [dashed] (0,0) circle (1);
\fill[fill=medgrey,samples=200,smooth] plot[domain=-pi:pi] (xy polar cs:angle=\x r,radius= {cos(\x r/2)^2});
\filldraw (1,0) circle ({0.6*1.5/1pt});
\draw (1,0) node [above right] {$1$};
\draw [<->] (-.3,0) -- (\u,0);
\draw [<->] (0,-\w) -- (0,\w);
\draw (0.7,-1) node {$\left\{re^{i\varphi}\,|\,0\le r\le \cos^2(\varphi/2)\right\}$};
\end{tikzpicture}}
\end{tabular}
\end{center}
\end{fact}
In Fact~\ref{thm:fne-compose}, the precise characterization $\cG(\cN_{1/2}\cN_{1/2})$ is new, but $\cN_{1/2}\cN_{1/2}\subset \cN_{2/3}$ is known.

%The composition of two firmly nonexpansive operators arise in alternating projections onto convex sets.
\emph{Proof outline}\,
We quickly outline the geometric insight while deferring the full proof with precise geometric arguments to the Section~\ref{s:geo-proofs} in the appendix.

\newpage
Define
\begin{center}
\begin{tabular}{cc}
%$Q=\left\{\frac{1}{2}+\frac{r}{2}e^{i\varphi}\,|\,
%0\le r\le 1
%r\in [0,1],\,\varphi\in[0,2\pi]
%\right\}=$
$Q=$
\raisebox{-.5\height}{
\begin{tikzpicture}[scale=1]
\fill [fill=medgrey] (1/2,0) circle (1/2);
\draw [<->] (-.2,0) -- (1.2,0);
\draw [<->] (0,-.6) -- (0,.6);
\filldraw (1,0) circle ({0.6*1.5/1pt});
\draw (1,0) node[above right] {$1$};
\end{tikzpicture}}
&\qquad
%$C=\left\{\frac{1}{2}+\frac{1}{2}e^{i\varphi}\right\}=$
%\,\bigg|\,\varphi\in[0,2\pi]\right\}=$
$C=$
\raisebox{-.5\height}{
\begin{tikzpicture}[scale=1]
\draw  (1/2,0) circle (1/2);
\draw [<->] (-.2,0) -- (1.2,0);
\draw [<->] (0,-.6) -- (0,.6);
\filldraw (1,0) circle ({0.6*1.5/1pt});
\draw (1,0) node[above right] {$1$};
\end{tikzpicture}}
\end{tabular}
\end{center}
and
\[
S=\bigcup_{0\le \varphi_1\le 2\pi}S_{\varphi_1},
\qquad
S_{\varphi_1}=
Q\left(\frac{1}{2}+\frac{1}{2}e^{i\varphi_1}\right).
\]
%\[
%S_{\varphi_1}=\left\{
%\left(\frac{1}{2}+\frac{r}{2}e^{i\varphi_2}\right)
%\left(\frac{1}{2}+\frac{1}{2}e^{i\varphi_1}\right)
%\,\bigg|\,r\in [0,1],\,\varphi_2\in [0,2\pi]
%\right\}
%\]
In geometric terms, this construction takes a point on the circle $C$, draws the disk whose diameter is the line segment between this point and the origin, and takes the union of such disks.
$S=\cG(\cN_{1/2})\cG(\cN_{1/2})$ follows from Theorem~\ref{thm:composition-srg}.
\begin{center}
\begin{tabular}{c}
\raisebox{-.5\height}{
\begin{tikzpicture}[scale=3]
\def\a{pi*.66/3};
\def\b{pi*.66*2/3};
\def\c{pi*.66};
\def\w{1.2};
\def\u{1.2};

\fill[fill=lightgrey] ({cos(\a/pi*180)/4+1/4},{sin(\a/pi*180)/4}) circle ({cos(\a/pi*180/2)/2});
\coordinate (A1) at ({1},{tan(\a/pi*180/2)}) ;
\coordinate (A2) at ({1/2*cos(\a/pi*180)+1/2},{sin(\a/pi*180)/2});
\filldraw (A2) circle[radius={0.6*1.5/3pt}];
%\filldraw (A1)  circle[radius={0.6*1.5/3pt}];
\draw [line width=1.0pt] (0,0) -- (A2);
%\draw [line width=1.0pt] ({1-tan(\a/pi*180/2)*(\w-tan(\a/pi*180/2))},{\w}) -- ({1+tan(\a/pi*180/2)*(\w-tan(\a/pi*180/2))},{2*tan(\a/pi*180/2)-\w});

\fill[fill=medgrey] ({cos(\b/pi*180)/4+1/4},{sin(\b/pi*180)/4}) circle ({cos(\b/pi*180/2)/2});
\coordinate (B1) at ({1},{tan(\b/pi*180/2)}) ;
\coordinate (B2) at ({1/2*cos(\b/pi*180)+1/2},{sin(\b/pi*180)/2});
\filldraw (B2) circle[radius={0.6*1.5/3pt}];
%\filldraw (B1)  circle[radius={0.6*1.5/3pt}];
\draw [line width=1.0pt] (0,0) -- (B2);
%\draw [line width=1.0pt] ({1-tan(\b/pi*180/2)*(\w-tan(\b/pi*180/2))},{\w}) -- ({1+tan(\b/pi*180/2)*(\w-tan(\b/pi*180/2))},{2*tan(\b/pi*180/2)-\w});

\fill[fill=darkgrey] ({cos(\c/pi*180)/4+1/4},{sin(\c/pi*180)/4}) circle ({cos(\c/pi*180/2)/2});
\coordinate (C1) at ({1},{tan(\c/pi*180/2)}) ;
\coordinate (C2) at ({1/2*cos(\c/pi*180)+1/2},{sin(\c/pi*180)/2});
\filldraw (C2) circle[radius={0.6*1.5/3pt}];
%\filldraw (C1)  circle[radius={0.6*1.5/3pt}];
\draw [line width=1.0pt] (0,0) -- (C2);
%\draw[ line width=1.0pt] ({1-tan(\c/pi*180/2)*(\w-tan(\c/pi*180/2))},{\w}) -- ({1+tan(\c/pi*180/2)*(\w-tan(\c/pi*180/2))},{2*tan(\c/pi*180/2)-\w});

\draw [dashed] (0,0) circle (1);
\draw [] (1/2,0) circle (1/2);

%\filldraw  ({cos(\a/pi*180)/4+1/4},{sin(\a/pi*180)/4}) circle[radius={0.6*1.5/3pt}];

\filldraw (0,0) circle[radius={0.6*1.5/3pt}];

\draw [<->] (-1.2,0) -- (\u,0);
\draw [<->] (0,-\w) -- (0,\w);
%\draw [] (1,-\w) -- (1,\w);

\coordinate (O) at (0,0);
\coordinate (E) at (1,0);
%\tkzMarkRightAngle[size=.053](O,A2,E);
%\tkzMarkRightAngle[size=.053](O,B2,E);
%\tkzMarkRightAngle[size=.053](O,C2,E);

\def\r{0.15};
\filldraw (1/2,0) circle[radius={0.6*1.5/3pt}];
\draw (1/2,0) -- (C2);
\draw ({1/2+\r},{0}) arc (0:{\c/pi*180}:\r) node[midway,right]{$\varphi_1$};

\def\f{-0.12};
\draw[->] (-0.3,0.3) --(\f,{sqrt((cos(\c/pi*180/2)/2)^2-(cos(\c/pi*180)/4+1/4-\f)^2)+sin(\c/pi*180)/4});
\draw  (-0.3,0.3) node[left] {$S_{\varphi_1}$};
\draw (0.5,-0.5) node[below] {$C$};
\end{tikzpicture}}

\end{tabular}
\end{center}
%The dashed circle is the unit circle.
%The solid circle is $C$.
%%Circle $C$, the solid circle, represents the points $(1/2)+(1/2)e^{i\varphi_1}$ for $\varphi_1\in[0,2\pi]$.
%The shaded circles represent instances of $S_{\varphi_1}$.
%We can characterize $\cG(\cN_{1/2}\cN_{1/2})$ by analyzing this construction
%since
%\begin{align*}
%S&=QC=(Q [0,1]) C=
%Q([0,1]C)\\
%&=QQ=\cG(\cN_{1/2})\cG(\cN_{1/2})=\cG(\cN_{1/2}\cN_{1/2})
%\end{align*}
%by Proposition~\ref{prop:monotone-srg} and Theorem~\ref{thm:composition-srg}.
%%, $\cG(\cN_{1/2})=Q$.
%%Since $\cN_{1/2}$ is SRG-full and satisfies the right-arc property, 
%% tells us that $\cG(\cN_{1/2}\cN_{1/2})=QQ=S$.

%This fact is known and can be analytically derived through the envelope theorem \cite[Exercise 5.22]{riley2006}.
%We provide a geometric proof, which was inspired by \cite[Exercise 4.15]{kline1977}.

To show $S=\left\{re^{i\varphi}\,|\,0\le r\le \cos^2(\varphi/2)\right\}$, we analyze $S$ in the inverted space.
Write $\cI\colon\ecomplex\rightarrow\ecomplex$ for the mapping $\cI(z)=\bar{z}^{-1}$.
\newpage

%Next, we analyze $S$ in the inverted space, i.e.,
%we analyze 
%\[
%\cI(S)=\bigcup_{0\le \varphi_1\le 2\pi}\cI(S_{\varphi_1}).
%\]
%\newpage

%\begin{figure}
\begin{center}
\begin{tabular}{c}
\raisebox{-.5\height}{
\begin{tikzpicture}[scale=3]
\def\a{pi*.66/3};
\def\b{pi*.66*2/3};
\def\c{pi*.66};
\def\w{2};
\def\u{2.5};

\begin{scope}
\clip (-1.2,-1.2) rectangle (\u,\w);

\coordinate (A3) at  ({1-tan(\a/pi*180/2)*(\w-tan(\a/pi*180/2))},{\w});
\fill[fill=lightgrey] (A3)
-- (\u,{\w-(\u-(1-tan(\a/pi*180/2)*(\w-tan(\a/pi*180/2))))/tan(\a/pi*180/2)})
-- (\u,\w)
-- cycle;

\draw (2,-1) node {$\cup \{\binfty\}$};

%-- ({1+tan(\a/pi*180/2)*(\w-tan(\a/pi*180/2))},{2*tan(\a/pi*180/2)-\w})
%-- ({1+tan(\a/pi*180/2)*(\w-tan(\a/pi*180/2))},{\w})

\coordinate (B3) at  ({1-tan(\b/pi*180/2)*(\w-tan(\b/pi*180/2))},{\w});
\fill[fill=medgrey] (B3)
-- (\u,{\w-(\u-(1-tan(\b/pi*180/2)*(\w-tan(\b/pi*180/2))))/tan(\b/pi*180/2)}) 
-- (\u,\w)
-- cycle;
\draw (2,.6) node {$\cup \{\binfty\}$};

\coordinate (C3) at ({1-tan(\c/pi*180/2)*(\w-tan(\c/pi*180/2))},{\w});
\fill[fill=darkgrey] (C3)
--(\u,{\w-(\u-(1-tan(\c/pi*180/2)*(\w-tan(\c/pi*180/2))))/tan(\c/pi*180/2)}) 
-- (\u,\w)
-- cycle;

\draw (2,1.7) node {$\cup \{\binfty\}$};

\draw (1.9,1.35) node {$\cI(S_{\varphi_1})$};

\end{scope}

\fill[fill=lightgrey] ({cos(\a/pi*180)/4+1/4},{sin(\a/pi*180)/4}) circle ({cos(\a/pi*180/2)/2});
\coordinate (A1) at ({1},{tan(\a/pi*180/2)}) ;
\coordinate (A2) at ({1/2*cos(\a/pi*180)+1/2},{sin(\a/pi*180)/2});
\filldraw (A2) circle[radius={0.6*1.5/3pt}];
\filldraw (A1)  circle[radius={0.6*1.5/3pt}];
\draw [line width=1.0pt] (0,0) -- ({1},{tan(\a/pi*180/2)});
%\draw [line width=1.0pt] ({1-tan(\a/pi*180/2)*(\w-tan(\a/pi*180/2))},{\w}) -- ({1+tan(\a/pi*180/2)*(\w-tan(\a/pi*180/2))},{2*tan(\a/pi*180/2)-\w});

\fill[fill=medgrey] ({cos(\b/pi*180)/4+1/4},{sin(\b/pi*180)/4}) circle ({cos(\b/pi*180/2)/2});
\coordinate (B1) at ({1},{tan(\b/pi*180/2)}) ;
\coordinate (B2) at ({1/2*cos(\b/pi*180)+1/2},{sin(\b/pi*180)/2});
\filldraw (B2) circle[radius={0.6*1.5/3pt}];
\filldraw (B1)  circle[radius={0.6*1.5/3pt}];
\draw [line width=1.0pt] (0,0) -- ({1},{tan(\b/pi*180/2)});
%\draw [line width=1.0pt] ({1-tan(\b/pi*180/2)*(\w-tan(\b/pi*180/2))},{\w}) -- ({1+tan(\b/pi*180/2)*(\w-tan(\b/pi*180/2))},{2*tan(\b/pi*180/2)-\w});

\fill[fill=darkgrey] ({cos(\c/pi*180)/4+1/4},{sin(\c/pi*180)/4}) circle ({cos(\c/pi*180/2)/2});
\coordinate (C1) at ({1},{tan(\c/pi*180/2)}) ;
\coordinate (C2) at ({1/2*cos(\c/pi*180)+1/2},{sin(\c/pi*180)/2});
\filldraw (C2) circle[radius={0.6*1.5/3pt}];
\filldraw (C1)  circle[radius={0.6*1.5/3pt}];
\draw [line width=1.0pt] (0,0) -- ({1},{tan(\c/pi*180/2)});
%\draw[ line width=1.0pt] ({1-tan(\c/pi*180/2)*(\w-tan(\c/pi*180/2))},{\w}) -- ({1+tan(\c/pi*180/2)*(\w-tan(\c/pi*180/2))},{2*tan(\c/pi*180/2)-\w});

\draw [dashed] (0,0) circle (1);
\draw [] (1/2,0) circle (1/2);

%\filldraw  ({cos(\a/pi*180)/4+1/4},{sin(\a/pi*180)/4}) circle[radius={0.6*1.5/3pt}];

\filldraw (0,0) circle[radius={0.6*1.5/3pt}];

%\draw [dashed] (-1,0) circle (2);
% \filldraw (-3,0) circle[radius={0.6*1.5/3pt}];
% \draw (-3,0) node [above right] {$-3$};

\draw [<->] (-1.2,0) -- (\u,0);
\draw [<->] (0,-1.2) -- (0,\w);
\draw [] (1,-1.2) -- (1,\w);

\draw (0.5,-0.5) node[below] {$C$};

\draw (1,-1) node[right] {$\cI(C)$};

%\coordinate (C) at ($(C)!(A)!(B)$);
%%\draw (A)node[below left]{$A$}--(B)node[below right]{$B$}--(C)node[above left]{$C$}--cycle;
%%\draw[dashed] (A)--(D)node[above right]{$D$};

\tkzMarkRightAngle[size=.053](A2,A1,A3);
\tkzMarkRightAngle[size=.053](B2,B1,B3);
\tkzMarkRightAngle[size=.053](C2,C1,C3);

\coordinate (O) at (0,0);
\coordinate (E) at (1,0);
\tkzMarkRightAngle[size=.053](A1,A2,E);
\tkzMarkRightAngle[size=.053](B1,B2,E);
\tkzMarkRightAngle[size=.053](C1,C2,E);
%\tkzMarkRightAngle[size=.075]((0,0),B2,(1,0));
%\tkzMarkRightAngle[size=.075]((0,0),C2,(1,0));

\draw (O) node[below left] {$O$};
\draw (C2) node[above left] {$A$};
\draw (C1) node[above right] {$\cI(A)$};

\def\f{-0.12};
\draw[->] (-0.3,0.3) --(\f,{sqrt((cos(\c/pi*180/2)/2)^2-(cos(\c/pi*180)/4+1/4-\f)^2)+sin(\c/pi*180)/4});
\draw  (-0.3,0.3) node[left] {$S_{\varphi_1}$};
\end{tikzpicture}}

\end{tabular}
\end{center}
%\caption{XXX}
%\end{figure}
%Again, $\cI(z)=\bar{z}^{-1}$.
%The dashed circle, the unit circle, is mapped onto itself.
%Circle $C$, the solid circle, is mapped to  $\cI(C)$, the verticle line going through $1$.
%Each shaded circle $S_{\varphi_1}$ is mapped to a half-space $\cI(S_{\varphi_1})$.
%Let point $A$ be the nonzero intersection between $C$ and the boundary of $S_{\varphi_1}$.
%Then point $\cI(A)$ is the non-infinite intersection between $\cI(C)$ and the boundary of $\cI(S_{\varphi_1})$.
%%$S_{\varphi_1}$ and $C$ intersect at $A$ and $O$
%By construction,  $\overline{OA}$ is the diameter of $S_{\varphi_1}$.
%The (infinite) line containing $O$, $A$, and $\cI(A)$ is mapped onto itself, excluding the origin.
%Since $\cI$ is conformal, the right angle at $A$ between the boundary of $S_{\varphi_1}$ and the diameter $\overline{OA}$
%is mapped to a right angle between boundary of $\cI(S_{\varphi_1})$ and $\overline{O\cI(A)}$.
%\newpage

The union of the half-spaces $\cI(S)=\bigcup_{0\le \varphi_1\le 2\pi}\cI(S_{\varphi_1})$ forms a parabola.
%\begin{figure}
\begin{center}
\begin{tabular}{c}
\raisebox{-.5\height}{
\begin{tikzpicture}[scale=3]
\def\a{pi*.17};
\def\b{pi*.45};
\def\c{pi*.66};
\def\w{2};
\def\u{2.5};

\begin{scope}
\clip (-1.2,-1) rectangle (\u,\w);

\coordinate (A3) at  ({1-tan(\a/pi*180/2)*(\w-tan(\a/pi*180/2))},{\w});
\fill[fill=lightgrey] (A3)
-- (\u,{\w-(\u-(1-tan(\a/pi*180/2)*(\w-tan(\a/pi*180/2))))/tan(\a/pi*180/2)})
-- (\u,\w)
-- cycle;

%-- ({1+tan(\a/pi*180/2)*(\w-tan(\a/pi*180/2))},{2*tan(\a/pi*180/2)-\w})
%-- ({1+tan(\a/pi*180/2)*(\w-tan(\a/pi*180/2))},{\w})

\coordinate (B3) at  ({1-tan(\b/pi*180/2)*(\w-tan(\b/pi*180/2))},{\w});
\fill[fill=medgrey] (B3)
-- (\u,{\w-(\u-(1-tan(\b/pi*180/2)*(\w-tan(\b/pi*180/2))))/tan(\b/pi*180/2)}) 
-- (\u,\w)
-- cycle;

%\coordinate (C3) at ({1-tan(\c/pi*180/2)*(\w-tan(\c/pi*180/2))},{\w});
%\fill[fill=darkgrey] (C3)
%--(\u,{\w-(\u-(1-tan(\c/pi*180/2)*(\w-tan(\c/pi*180/2))))/tan(\c/pi*180/2)}) 
%-- (\u,\w)
%-- cycle;

\end{scope}

\coordinate (A1) at ({1},{tan(\a/pi*180/2)}) ;
\coordinate (A2) at ({1/2*cos(\a/pi*180)+1/2},{sin(\a/pi*180)/2});
%\filldraw (A2) circle[radius={0.6*1.5/3pt}];
%\filldraw (A1)  circle[radius={0.6*1.5/3pt}];
\draw [line width=1.0pt] (0,0) -- ({2},{2*tan(\a/pi*180/2)});
%\draw [line width=1.0pt] ({1-tan(\a/pi*180/2)*(\w-tan(\a/pi*180/2))},{\w}) -- ({1+tan(\a/pi*180/2)*(\w-tan(\a/pi*180/2))},{2*tan(\a/pi*180/2)-\w});

\coordinate (B1) at ({1},{tan(\b/pi*180/2)}) ;
\coordinate (B2) at ({1/2*cos(\b/pi*180)+1/2},{sin(\b/pi*180)/2});
%\filldraw (B2) circle[radius={0.6*1.5/3pt}];
%\filldraw (B1)  circle[radius={0.6*1.5/3pt}];
\draw [line width=1.0pt] (0,0) -- ({2},{2*tan(\b/pi*180/2)});
%\draw [line width=1.0pt] ({1-tan(\b/pi*180/2)*(\w-tan(\b/pi*180/2))},{\w}) -- ({1+tan(\b/pi*180/2)*(\w-tan(\b/pi*180/2))},{2*tan(\b/pi*180/2)-\w});

\begin{scope}
\clip (-1.2,-1) rectangle (\u,\w);

\coordinate (C1) at ({1},{tan(\c/pi*180/2)}) ;
\coordinate (C2) at ({1/2*cos(\c/pi*180)+1/2},{sin(\c/pi*180)/2});
%\filldraw (C2) circle[radius={0.6*1.5/3pt}];
%\filldraw (C1)  circle[radius={0.6*1.5/3pt}];
%\draw [line width=1.0pt] (0,0) -- ({2},{2*tan(\c/pi*180/2)});
%\draw[ line width=1.0pt] ({1-tan(\c/pi*180/2)*(\w-tan(\c/pi*180/2))},{\w}) -- ({1+tan(\c/pi*180/2)*(\w-tan(\c/pi*180/2))},{2*tan(\c/pi*180/2)-\w});
\end{scope}
%\draw [] (1/2,0) circle (1/2);

%\filldraw  ({cos(\a/pi*180)/4+1/4},{sin(\a/pi*180)/4}) circle[radius={0.6*1.5/3pt}];

\filldraw (0,0) circle[radius={0.6*1.5/3pt}];
\filldraw (1,0) circle[radius={0.6*1.5/3pt}];

%\draw [dashed] (-1,0) circle (2);
% \filldraw (-3,0) circle[radius={0.6*1.5/3pt}];
% \draw (-3,0) node [above right] {$-3$};

\draw [<->] (-1.2,0) -- (\u,0);
\draw [<->] (0,-1) -- (0,\w);
\draw [] (1,-1) -- (1,\w);

%\coordinate (C) at ($(C)!(A)!(B)$);
%%\draw (A)node[below left]{$A$}--(B)node[below right]{$B$}--(C)node[above left]{$C$}--cycle;
%%\draw[dashed] (A)--(D)node[above right]{$D$};

\tkzMarkRightAngle[size=.053](A2,A1,A3);
\tkzMarkRightAngle[size=.053](B2,B1,B3);
%\tkzMarkRightAngle[size=.053](C2,C1,C3);

%\tkzMarkRightAngle[size=.075]((0,0),B2,(1,0));
%\tkzMarkRightAngle[size=.075]((0,0),C2,(1,0));

\draw[line width=1.0pt, domain=-1:\w,variable=\y] plot ({1-(\y)^2/4},{\y});

\coordinate (O) at (0,0);
\coordinate (E) at (2,-\w);

\coordinate (A7) at  ({1-tan(\a/pi*180/2)^2},{2*tan(\a/pi*180/2)}) ;
\coordinate (A8) at (2,{2*tan(\a/pi*180/2)}) ;

\coordinate (B7) at  ({1-tan(\b/pi*180/2)^2},{2*tan(\b/pi*180/2)}) ;
\coordinate (B8) at (2,{2*tan(\b/pi*180/2)}) ;

\filldraw (A7) circle[radius={0.6*1.5/3pt}];
\filldraw (B7) circle[radius={0.6*1.5/3pt}];

\tkzMarkRightAngle[size=.053](A7,A8,E);
\tkzMarkRightAngle[size=.053](B7,B8,E);

\draw (O) -- (A7) -- (A8) ;
\draw (O) -- (B7) -- (B8) ;

\tkzMarkSegment[pos=.5,mark=||](O,A7) ;
\tkzMarkSegment[pos=.5,mark=||](A7,A8) ;

\tkzMarkSegment[pos=.5,mark=|](O,B7) ;
\tkzMarkSegment[pos=.5,mark=|](B7,B8) ;

\draw (2.25,1.2) node {$\cup \{\binfty\}$};
\draw (2.25,-.85) node {$\cup \{\binfty\}$};

\draw (0,0) node [above left] {focus};
\draw (0,0) node [below left] {$O$};
\draw (1,0) node [above right] {vertex};
\draw (1,0) node [below right] {$1$};

\filldraw (2,0) circle[radius={0.6*1.5/3pt}];
%\draw (2,0) node [above right] {directrix};
\draw (2,0) node [below right] {$2$};
\draw [<->] (2,-1) -- (2,\w);

\filldraw (B1) circle[radius={0.6*1.5/3pt}];
\draw (B1) node[right] {$A'$};
\filldraw (B8) circle[radius={0.6*1.5/3pt}];
\draw (B8) node[right] {$E$};
\draw (B7) node[above ] {$B$};

%\draw [dashed] (-1,0) circle (2);
% \filldraw (-3,0) circle[radius={0.6*1.5/3pt}];
% \draw (-3,0) node [above right] {$-3$};

\draw[->] (-0.2,-.3) node[left]{$x=1-\frac{y^2}{4}$} -- ({1-(-.5)^2/4},-.5) ;
%\draw  (1.5,-1.5) node [] {$x=1-\frac{y^2}{4}$};
\draw (1,-.75) node[right] {$\cI(C)$};
\draw (2.,.3) node[right] {$D$};
\draw (2,.17) node[right] {(directrix)};

\coordinate (Bp) at  ({1-tan(\b/pi*180/2)^2+0.6*cos(90-\b/pi*180/2)},{2*tan(\b/pi*180/2)-0.6*sin(90-\b/pi*180/2)}) ;
\filldraw (Bp) circle[radius={0.6*1.5/3pt}];
\draw (Bp) node[right] {$B'$};
\end{tikzpicture}}

\end{tabular}
\end{center}

Find the largest circle tangent to the parabola at point $1$ and invert back.

\begin{center}
\begin{tabular}{ccc}
\raisebox{-.5\height}{
\begin{tikzpicture}[scale=.9]
\def\a{pi*.17};
\def\b{pi*.45};
\def\c{pi*.66};
\def\w{2.5};
\def\u{3.5};

\fill[fill=medgrey] (-\u,-\w) rectangle (1.5,\w);
\begin{scope}
\clip (-\u,-\w) rectangle (1.5,\w);
\fill[fill=white] (-1,0) circle (2);
\end{scope}

\fill[fill=darkgrey, domain=-\w:\w,variable=\y] (1.5,{-\w}) -- plot ({1-(\y)^2/4},{\y}) -- (1.5,{\w});
\filldraw (-1,0) circle[radius={0.6*1.5/0.9pt}];
\draw (-1,0) node [below ] {$-1$};
\filldraw (1,0) circle[radius={0.6*1.5/0.9pt}];
\draw (1,0) node [below left] {$1$};
\filldraw (-3,0) circle[radius={0.6*1.5/0.9pt}];
\draw (-3,0) node [below right] {$-3$};
%\fill[fill=darkgrey] ({1-(\w)^2/4},-\w) rectangle (\u,\w);
%\begin{scope}
%\clip[domain=-\w:\w,variable=\y] plot ({1-(\y)^2/4},{\y});
%\fill[fill=medgrey] (-1.5,-\w) rectangle (\u,\w);
%\end{scope}
%
%
%\begin{scope}
%\clip[domain=-\w:\w,variable=\y] plot ({1-(\y)^2/4},{\y});
%\fill[fill=medgrey] (-1.5,-\w) rectangle (\u,\w);
%\end{scope}
%\fill[fill=white] (-1,0) circle (2);
%

\draw [<->] (-\u,0) -- (1.5,0);
\draw [<->] (0,-\w) -- (0,\w);
%
%\draw[line width=1.0pt, domain=-\w:\w,variable=\y] plot ({1-(\y)^2/4},{\y});

\begin{scope}
\clip (-1,0) circle (2);
\draw[->] (-.65,.15) node[above, text width=3cm,fill=white]{\footnotesize $x=1-\frac{y^2}{4}$ and \phantom{xxxx} $x=\sqrt{4-y^2}-1$ \phantom{xx}  have matching curvature} -- (1,0) ;
\end{scope}

\draw (.95,2.2) node {$\cup \{\binfty\}$};

\draw (-2.7,2.2) node {$\cup \{\binfty\}$};

\end{tikzpicture}}
&$\stackrel{\cI}{\longrightarrow}$
&
\raisebox{-.5\height}{
\begin{tikzpicture}[scale=1.5]
\def\a{pi*.17};
\def\b{pi*.45};
\def\c{pi*.66};
\def\w{1.2};
\def\u{1.2};

\draw [dashed] (0,0) circle (1);
\fill[fill=medgrey] (1/3,0) circle (2/3);
\fill[fill=darkgrey,samples=200,smooth] plot[domain=-pi:pi] (xy polar cs:angle=\x r,radius= {cos(\x r/2)^2});
%\draw [decorate,decoration={brace,amplitude=4.5pt}] (1/3,0) -- (1,0) ;
%\draw (1/3,0.1) node [above] {$1/3$};
\filldraw (-1/3,0) circle[radius={0.6*1.5/1.5pt}];
\draw (-1/3,0) node [above left] {$-\frac{1}{3}$};
%
%
%

%\draw[domain=-10:10, variable=\t] plot ({4*(\t-(\t)^2)/((4+(\t)^2)^2)},{16*\t/((4+(\t)^2)^2)});

%\fill[fill=darkgrey, domain=-\w:\w,variable=\y] (\u,{-\w}) -- plot ({1-(\y)^2/4},{\y}) -- (\u,{\w});

\filldraw (1,0) circle[radius={0.6*1.5/1.5pt}];
\draw (1,0) node [above right] {$1$};
%\fill[fill=darkgrey] ({1-(\w)^2/4},-\w) rectangle (\u,\w);
%\begin{scope}
%\clip[domain=-\w:\w,variable=\y] plot ({1-(\y)^2/4},{\y});
%\fill[fill=medgrey] (-1.5,-\w) rectangle (\u,\w);
%\end{scope}
%
%
%\begin{scope}
%\clip[domain=-\w:\w,variable=\y] plot ({1-(\y)^2/4},{\y});
%\fill[fill=medgrey] (-1.5,-\w) rectangle (\u,\w);
%\end{scope}
%\fill[fill=white] (-1,0) circle (2);
%

\draw [<->] (-1.2,0) -- (\u,0);
\draw [<->] (0,-\w) -- (0,\w);
%
%\draw[line width=1.0pt, domain=-\w:\w,variable=\y] plot ({1-(\y)^2/4},{\y});
\end{tikzpicture}}
\end{tabular}
\end{center}
%%The largest circle symmetric about the horizontal axis touching the point $(1,0)$ has radius $2$. This fact is easily verified with calculus.
%The region exterior to the circle centered at $-1$ with radius $2$ contains the region towards the right of the parabola.
%This is easily verified with calculus.
%%The largest circle symmetric about the horizontal axis touching the point $(1,0)$ has radius $2$. This fact is easily verified with calculus.
%The circle with the lighter shade corresponds to $\cG(\cN_{2/3})$ by  Proposition~\ref{prop:monotone-srg}.
%Since $\cN_{2/3}$ is SRG-full, strict containment of the SRG in $\ecomplex$ implies strict containment of the class by Theorem~\ref{thm:srg-full}.
%The inverse curve of the parabola with the focus as the center of inversion is known as the \emph{cardioid}
%and it has the polar coordinate representation $r(\varphi)\le \cos^2(\varphi/2)$. The expression of the Theorem is the region bounded by this curve.
The largest circle to the left of the parabola is inverted to the smallest circle (i.e., tight averagedness circle) containing the SRG.
The known formula $r(\varphi)\le \cos^2(\varphi/2)$ describes the parabola under the inversion mapping.
\qed
%\end{proof}

%\[
%\left\{
%\frac{1}{x^2+y^2}(x,y)
%\,\Big|\,x\ge 1-\frac{y^2}{4}
%\right\}
%\]
%which translates to a polar coordinate representation

%https://math.stackexchange.com/questions/824027/non-brute-force-proof-of-parabola-tangent-property

\subsection{Tightness and constructing lower bounds}
An advantage of geometric proofs is that tightness is often immediate.
In the proof of Fact~\ref{prop:refl-sm-coco}, for example,
%, which states that the result does not hold with any smaller value of $R$.
it is clear that finding a smaller circle containing the SRG is not possible.
Consequently, the rate of Fact~\ref{fact:tight_example} cannot be improved.

% \[
% R=\sqrt{1-\frac{4\alpha\mu}{1+2\alpha\mu+\alpha^2\mu/\beta}}.
% \]
% This result is tight in the sense that $2J_{\alpha \cA}-I\nsubseteq \cL_R$ for any smaller value of $R$.

% As an example, consider Fact~\ref{prop:refl-sm-coco}

% Point $A$ within the geometric diagrams is an extreme point that limits the size of $R$.
%\subsection{Construction }
However, although tightness is proved with the geometric arguments, sometimes one may wish to construct an explicit counterexample achieving the tight rate.
This can be done by picking the extreme point on the complex plane, finding a corresponding $2\times 2$ matrix with Lemma~\ref{lem:complex-srg}, and reverse engineering the proof.

% Finally, convert this complex number to a $2\times 2$ matrix.

%For example, consider Fact~\ref{fact:tight_example}.
In the setup of Fact~\ref{prop:refl-sm-coco},
\begin{center}
\raisebox{-.5\height}{
\begin{tikzpicture}[scale=2.5]
\def\m{0.5};
\def\b{.8};
\begin{scope}
\clip ({-\m/(1+\m)+sqrt(1/(1 + \m)/(1 + \m) + (4*\b*\m*(-1 + \b*\m))/(\b + \m + 2*\b*\m)/(\b + \m + 2*\b*\m))},-0.7)  rectangle (-1.2,0.9);
\draw[dashed] ({-\m/(1+\m)},0) circle ({1-\m/(1+\m)});
\end{scope}

\begin{scope}
\clip ({-\m/(1+\m)+sqrt(1/(1 + \m)/(1 + \m) + (4*\b*\m*(-1 + \b*\m))/(\b + \m + 2*\b*\m)/(\b + \m + 2*\b*\m))},-0.7)  rectangle (1.2,0.9);
\draw[dashed] ({1/(1+1/\b)},0) circle ({1-1/(1+1/\b)});
\end{scope}

\begin{scope}
\clip ({-\m/(1+\m)},0) circle ({1-\m/(1+\m)});
\fill[fill=medgrey] ({1/(1+1/\b)},0) circle ({1-1/(1+1/\b)});
%\draw[line width=1.0pt] ({1/(1+1/\b)},0) circle ({1-1/(1+1/\b)});
\end{scope}
\begin{scope}
\clip ({1/(1+1/\b)},0) circle ({1-1/(1+1/\b)});
%\draw[line width=1.0pt] ({-\m/(1+\m)},0) circle ({1-\m/(1+\m)});
\end{scope}
\draw [<->] (-1.2,0) -- (1.2,0);
\draw [<->] (0,-.7) -- (0,.7);
%\draw [dashed] (0,0) circle (1);
\draw (1,-0pt) node [below right] {$1$};
\draw (-1,-0pt) node [below left] {$-1$};
\filldraw (1,0)circle ({0.6*1.5/2.5pt});
\filldraw (-1,0)circle ({0.6*1.5/2.5pt});

\coordinate (O) at (0,0);
\filldraw (O) circle ({0.6*1.5/2.5pt});
%\draw (O) node[below right] {$O$};

%\coordinate (C) at ({(1-\m)/(1+\m)},0);
%\filldraw (C) circle ({0.6*1.5/2.5pt});
%\draw (C) node[above right] {$C$};
%\draw (.4,-.3) node [below] {$\frac{1-\alpha\mu}{1+\alpha\mu}$};
%\draw[->] (.4,-.3)--({(1-\m)/(1+\m)},0);

%\filldraw ({(1-1/\b)/(1+1/\b)},0)circle ({0.6*1.5/2.5pt});
%\draw[->] (-.3,.4)--({(1-1/\b)/(1+1/\b)},0);
%\draw (-.3,.4) node [above] {$\frac{\beta-\alpha}{\beta+\alpha}$};

\coordinate (C) at ({1/(1+1/\b)},0);
%\draw (C) node[above right] {$C$};
\filldraw (C)circle ({0.6*1.5/2.5pt});
\draw (.5,-.2) node [below] {$\frac{\beta}{\beta+\alpha}$};
\draw[->] (.5,-.2)--(C);

\coordinate (B) at ({-\m/(1+\m)},0);
%\draw (B) node[above] {$B$};
\filldraw (B) circle ({0.6*1.5/2.5pt});
\draw (-.5,-.2) node [below] {$\frac{-\alpha\mu}{1+\alpha\mu}$};
\draw[->] (-.5,-.2)-- (B);

%\draw [line width=1.0pt] ({-\m/(1+\m)},0.2) -- ({-\m/(1+\m)+\b/(1 + \b) + \m/(1 + \m)},0.2);

%\draw [line width=1.0pt] ({-\m/(1+\m)+sqrt(1/(1 + \m)/(1 + \m) + (4*\b*\m*(-1 + \b*\m))/(\b + \m + 2*\b*\m)/(\b + \m + 2*\b*\m))},0) -- ({-\m/(1+\m)+sqrt(1/(1 + \m)/(1 + \m) + (4*\b*\m*(-1 + \b*\m))/(\b + \m + 2*\b*\m)/(\b + \m + 2*\b*\m))},{(2*sqrt(-(\b*\m*(-1 + \b*\m))))/(\b + \m + 2*\b*\m)});

\coordinate (A) at  ({-\m/(1+\m)+sqrt(1/(1 + \m)/(1 + \m) + (4*\b*\m*(-1 + \b*\m))/(\b + \m + 2*\b*\m)/(\b + \m + 2*\b*\m))},{(2*sqrt(-(\b*\m*(-1 + \b*\m))))/(\b + \m + 2*\b*\m)}) ;

%\draw [dashed] (0,0) circle ({sqrt((\b + \m - 2*\b*\m)/   (\b + \m + 2*\b*\m))});

\filldraw (A) circle ({0.6*1.5/2.5pt});
\draw (A) node[above] {$z$};

\draw (A) -- (B) -- (C) -- (A);
\draw (A) -- (O);

\coordinate(M1) at (-1,0);
\tkzMarkSegment[pos=.5,mark=||](A,B) ;
\tkzMarkSegment[pos=.5,mark=||](B,M1) ;

\coordinate(M2)  at (1,0);
\tkzMarkSegment[pos=.5,mark=|](A,C) ;
\tkzMarkSegment[pos=.5,mark=|](M2,C) ;

%\draw[->] (.5,.7)--({-\m/(1+\m)+sqrt(1/(1 + \m)/(1 + \m) + (4*\b*\m*(-1 + \b*\m))/(\b + \m + 2*\b*\m)/(\b + \m + 2*\b*\m))},{(2*sqrt(-(\b*\m*(-1 + \b*\m))))/(\b + \m + 2*\b*\m)});
%\draw (1,.7) node [above] {$
%\left(
%\sqrt{\frac{1}{(1+\alpha\mu)^2}-\frac{4\alpha^2\mu\beta(1-\mu\beta)}{(\beta+\alpha\mu(\alpha+2\beta))^2}}
%,\frac{2\alpha\sqrt{\mu\beta(1-\mu\beta)}}{\beta+\alpha\mu(\alpha+2\beta)}\right)
%$};

%\draw [decorate,decoration={brace,amplitude=4.5pt}] ({-\m/(1+\m)+sqrt(1/(1 + \m)/(1 + \m) + (4*\b*\m*(-1 + \b*\m))/(\b + \m + 2*\b*\m)/(\b + \m + 2*\b*\m))},{(2*sqrt(-(\b*\m*(-1 + \b*\m))))/(\b + \m + 2*\b*\m)}) -- (0,0);
%\draw [line width=1.0pt] (0,0) -- ({-\m/(1+\m)+sqrt(1/(1 + \m)/(1 + \m) + (4*\b*\m*(-1 + \b*\m))/(\b + \m + 2*\b*\m)/(\b + \m + 2*\b*\m))},{(2*sqrt(-(\b*\m*(-1 + \b*\m))))/(\b + \m + 2*\b*\m)});

%\draw[->](0.7,-0.7) -- ({(-\m/(1+\m)+sqrt(1/(1 + \m)/(1 + \m) + (4*\b*\m*(-1 + \b*\m))/(\b + \m + 2*\b*\m)/(\b + \m + 2*\b*\m)))/2+.07},{((2*sqrt(-(\b*\m*(-1 + \b*\m))))/(\b + \m + 2*\b*\m))/2-.03});
%\draw (0.8,-0.7) node [below] {
%$\sqrt{\frac{\beta+\alpha^2\mu-2\alpha\beta\mu}{\beta+\alpha\mu(\alpha+2\beta)}}$
%$\sqrt{1-\frac{4\alpha\mu}{1+2\alpha\mu+\alpha^2\mu/\beta}}$};
\end{tikzpicture}}
\end{center}
the extreme point $z$ %that limits the minimal value of $R$
corresponds to the complex number
\[
z=
\frac{1-\alpha^2\mu/\beta}{
1+2\alpha\mu+\alpha^2\mu/\beta
}
+
\frac{
2\alpha\sqrt{(1-\mu\beta)\mu/\beta}
}{
1+2\alpha\mu+\alpha^2\mu/\beta
}
i.
\]
Lemma~\ref{lem:complex-srg} provides a corresponding operator $A_z\colon\reals^2\rightarrow\reals^2$
\[
A_z
\begin{bmatrix}
\zeta_1\\
\zeta_2
\end{bmatrix}
=
\underbrace{
\frac{1}{1+2\alpha\mu+\alpha^2\mu/\beta}
\begin{bmatrix}
1-\alpha^2\mu/\beta&
-2\alpha\sqrt{(1-\mu\beta)\mu/\beta}\\
2\alpha\sqrt{(1-\mu\beta)\mu/\beta}&
1-\alpha^2\mu/\beta
\end{bmatrix}
}_{=M}
\begin{bmatrix}
\zeta_1\\
\zeta_2
\end{bmatrix}
\]
In the proof, the depicted geometry was obtained through the transformations $A\mapsto I+\alpha A$, $A\mapsto A^{-1}$, and $A\mapsto 2A-I$.
We revert the transformations by applying
$A\mapsto \frac{1}{2}I+\frac{1}{2}A$, $A\mapsto A^{-1}$, and $A\mapsto \frac{1}{\alpha}(A-I)$
and define $A\colon\reals^2\rightarrow\reals^2$ a as
\[
A
\begin{bmatrix}
\zeta_1\\
\zeta_2
\end{bmatrix}=\frac{1}{\alpha}\left(\left(\frac{1}{2}I+\frac{1}{2}M\right)^{-1}-I\right)
\begin{bmatrix}
\zeta_1\\
\zeta_2
\end{bmatrix}.
\]
(We do not show the individual entries as they are very complicated.)
Finally, if $B=0$, then the fixed-point iteration
\[
z^{k+1}=\left(\tfrac{1}{2}I+\tfrac{1}{2}(2J_{\alpha A}-I)(2J_{\alpha B}-I)\right)z^k
\]
converges at the exact rate given by Fact~\ref{fact:tight_example}.

If $\mathcal{A}$ is SRG-full and $z\in\mathcal{G}( \mathcal{A})$, then there is an operator $A$ on $\mathbb{R}^2$ constructed such that $\{z,\overline{z}\}=\mathcal{G}(A)$ so explicit counterexamples providing the lower bounds can be constructed in $\mathbb{R}^2$.
When an operator class is not SRG-full, counter examples still exist, but they may not be in $\mathbb{R}^2$.

\section{Insufficiency of metric subregularity for linear convergence}
\label{s:metric-subregularity}
Recently, there has been much interest in analyzing optimization methods under assumptions weaker than strong convexity or strong monotonicity.
One approach is to assume \emph{metric subregularity} in place of strong monotonicity and establish linear convergence.

In this section, we show that it is not always possible to replace strong monotonicity with metric subregularity.
%and retain linear convergence.
In particular, we show impossibility results proving  the insufficiency of metric subregularity in establishing linear convergence for certain setups where strong monotonicity is sufficient.

\subsection{Inverse Lipschitz continuity and metric subregularity}
Let
\begin{align*}
\cL^{-1}_{\gamma}&=\left\{A^{-1}\,|\,A\in \cL_{\gamma}\right\},
\\
&=\big\{A\colon\dom{A}\rightarrow\hilbert\,|\,\gamma^2\| Ax-Ay\|^2\ge \|x-y\|^2,\,\forall\, x,y\in\hilbert,\,\dom{A}\subseteq\hilbert\big\}
\end{align*}
be the class of inverse Lipschitz continuous operators
with parameter $\gamma\in (0,\infty)$, which has the SRG
%\begin{wrapfigure}{l}{0.3\textwidth}
\begin{center}
\begin{tabular}{c}
$\cG(\cL_\gamma^{-1})=$
\raisebox{-.5\height}{
\begin{tikzpicture}[scale=1.2]
\fill[fill=medgrey] (-1,-1) rectangle (1,1);
\fill[fill=white] (0,0) circle (0.55);
%\draw[line width=1.0pt] (0,0) circle (0.75);
\draw [<->] (-1,0) -- (1,0);
\draw [<->] (0,-1) -- (0,1);
\draw (0.45,0) node [above right] {$1/\gamma$};
\filldraw (.55,0) circle ({0.6*1.5/1.2pt});
\draw (0.6,.8) node {$\cup \{\binfty\}$};
\draw (0,-1.25) node {$\left\{z\in \complex\,\big|\,|z|^2\ge 1/\gamma^2\right\}$};
\end{tikzpicture}}
\end{tabular}
\end{center}
%\end{wrapfigure}
It is clear that inverse Lipschitz continuity is weaker than strong monotonicity in the sense that
 $A\in \cM_{1/\gamma}$ implies $ A\in \cL^{-1}_{\gamma}$.

An operator $A\colon\hilbert\rightrightarrows\hilbert$ is $\gamma$-metrically subregular at $x_0$ for $y_0$ if 
$y_0\in Ax_0$ and there exists a neighborhood $V$ of $x_0$ such that
\[
d(x,A^{-1}y_0)\le \gamma d(y_0,A(x)),\quad
\forall\, x\in V.
\]
Although not necessarily obvious from first sight, metric subregularity is weaker than inverse Lipschitz continuity, i.e.,
 $A\in \cL^{-1}_{\gamma}$ implies $A$ is metrically subregular at $x$ for $y$ with parameter $\gamma$, for any $(x,y)\in A$.

Metric subregularity of $A$ is equivalent to ``calmness'' of $A^{-1}$ \cite{Dontchev2004}, and calmness is also known as ``Upper Lipschitz continuity'' \cite{Robinson1981}.
For subdifferential operators of convex functions, metric subregularity is equivalent to the ``error bound condition'' \cite{drusvyatskiy_error_bound_2018}.
See \cite{implicit_rockafellar} for an in-depth treatment of this subject.

Metric subregularity has been used in place of strong monotonicity to establish linear convergence for a wide range of setups.
Leventhal \cite[Theorem 3.1]{leventhal2009} used metric subregularity for the proximal point method;
Bauschke, Noll, and Phan \cite[Lemma 3.8]{bauschke_metric_sub_2015} and Liang, Fadili, and Peyr\'e \cite[Theorem 3]{Liang2016}  for the Krasnoselskii--Mann iteration;
%with respect to a nonexpansive operator $T$ when $I-T$ is metrically subregular.
%at $x^\star$ for $0$, then the 
Latafat and Patrinos \cite[Theorem 3.3]{latafat2017} for their splitting method AFBA;
Ye et al.\ \cite{ye_FBS_ms_2018} for the proximal gradient method, the proximal alternating linearized minimization algorithm,
and the randomized block coordinate proximal gradient method;
and Yuan, Zeng, and Zhang for ADMM, DRS, and PDHG \cite{yuan_admm_ms_2018}.
See \cite{schmidt2016,Bolte2017,drusvyatskiy_error_bound_2018,Necoara2018,Zhang2019} for a systematic study of this subject.
%in-depth treatment on the use of metric subregularity and related assumptions in convergence analysis.
Although most recent work concerns sufficiency of metric subregularity or related assumptions in establishing linear convergence, Zhang \cite{Zhang2019} studied the necessary and sufficient conditions.

% Lai and Yin proved linear convergence  ``Restricted strong convexity'' and established metric subregularity under their problem setup \cite{yin_nuc_norm_2013}.
%``Restricted strong convexity'' is a relaxation of strong monotonicity, rather than inverse Lipschitz continuity, and therefore is different from metric subregularity.

\subsection{Impossibility proofs}
Douglas--Rachford splitting (DRS) is known to be a strict contraction under the combined assumption of Lipschitz continuity and strong monotonicity: \cite[Proposition 4]{lions1979}, \cite[Theorem 4.1]{HanYuan2012_convergence}, \cite[Table 1 under $A=B=I$]{DengYin2015_global}, \cite[Theorems 5--7]{DavisYin2017_faster}, \cite[Theorem 6.3]{giselsson20152}, and \cite[Theorem 4]{OSPEP}.
Is it possible to establish linear convergence with Lipschitz continuity and metric subregularity or a variation of metric subregularity? Then answer is no in the sense of Corollaries~\ref{cor:ms_impossibility1} and \ref{cor:ms_impossibility2}.

Define the DRS operator with respect to operators $A$ and $B$ with parameters $\alpha$ and $\theta$ as
\[
%T(\cA,\cB,\theta)=(1-\theta)I+\theta(2J_{\alpha \cA}-I)(2J_{\alpha \cB}-I).
D_{\alpha,\theta}(A,B)=(1-\theta)I+\theta(2J_{\alpha A}-I)(2J_{\alpha B}-I)
\]
and the class of DRS operators as
\[
D_{\alpha,\theta}(\cA,\cB)=\{D_{\alpha,\theta}(A,B)\,|\,A\in\cA,\,B\in\cB,\,
A\colon\hilbert\rightrightarrows\hilbert,\,
B\colon\hilbert\rightrightarrows\hilbert\}.
\]
Define $T(B,A,\alpha,\theta)$ and $T(\cB,\cA,\alpha,\theta)$ analogously.

\begin{theorem}
\label{thm:ms1}
Let $0<1/\gamma\le L<\infty$ and $\alpha\in(0,\infty)$.
Let $\cA=\cM\cap \cL_L\cap  (\cL_\gamma)^{-1}$
and $\cB=\cM$.
Then for any $\theta\ne 0$
\[
%2J_{\alpha \cA}-I\nsubseteq \cL_{1-\varepsilon}
D_{\alpha,\theta}(\cA,\cB)\nsubseteq \cL_{1-\varepsilon}
\]
for any $\varepsilon>0$.
The same conclusion holds for $D_{\alpha,\theta}(\cB,\cA)$.
\end{theorem}
\begin{proof}
%Consider the case $\alpha/\gamma< 1< \alpha L$. We have
We have the geometry
\begin{center}
\begin{tabular}{ccc}
\raisebox{-.5\height}{
\begin{tikzpicture}[scale=1]
\def\L{1.2}
\def\g{2}
\begin{scope}
\clip (0,-1.2) rectangle (1.2,1.2);
\fill[fill=medgrey] (0,0) circle ({\L});
\fill[fill=white] (0,0) circle ({1/\g});
\end{scope}
\draw [<->] (-0.2,0) -- (1.5,0);
\draw [<->] (0,-1.4) -- (0,1.4);
\draw [line width=1.0pt] (0,{\L}) -- (0,{1/\g});
\filldraw (0,{\L})   circle({0.6*1.5/1pt});
\draw (0,{\L}) node[left] {$\alpha Li=A$};
\filldraw (0,{1/\g})   circle({0.6*1.5/1pt});
\draw (0,{1/\g}) node[ left] {$\frac{\alpha }{\gamma }i=B$};
\filldraw (1,0) circle({0.6*1.5/1pt});
\draw (1,0) node [above] {$1$};
\draw [->] (-.3,-.7) -- (0,-.75);
\draw (-.3,-.7)  node [left] {$\cG\left(\alpha \cA\right)$};
\end{tikzpicture}}
&
$\xrightarrow{\frac{2}{1+\alpha z}-1}$
&\!\!\!\!\!\!\!\!\!\!\!\!\!\!\!\!\!\!
\raisebox{-.5\height}{
\begin{tikzpicture}[scale=1]
\def\L{1.2}
\def\g{2}
\begin{scope}
\clip (0,0) circle (1);
\fill [fill=medgrey] (0,0) circle (1);
\fill [fill=white]({(1+(1/\L)^2)/(-1+(1/\L)^2)},0) circle ({2*(1/\L)/(1-(1/\L)^2)});
\fill [fill=white]({(1+\g^2)/(-1+\g^2)},0) circle ({2*\g/(1-\g^2)});
\end{scope}

\begin{scope}
\clip ({(-1+(1/\L)^2)/(1+(1/\L)^2)},1.2) rectangle ({(-1+\g^2)/(1+\g^2)},0.5);
\draw [line width=1.0pt] (0,0) circle (1);
\end{scope}

\draw [<->] (-1.2,0) -- (1.2,0);
\draw [<->] (0,-1.4) -- (0,1.4);
\draw [dashed] (0,0) circle (1);

\coordinate (A) at  ({(-1+\g^2)/(1+\g^2)},{(2*\g)/(1+\g^2)}) ;
\coordinate (B) at  ({(-1+(1/\L)^2)/(1+(1/\L)^2)},{(2*(1/\L))/(1+(1/\L)^2)}) ;

\filldraw ({2/(1+1/\g)-1},0) circle({0.6*1.5/1pt});
\filldraw ({2/(1+\L)-1},0) circle({0.6*1.5/1pt});

\filldraw (A)  circle({0.6*1.5/1pt});
\draw  ({(-1+\g^2)/(1+\g^2)-0.2},{(2*\g)/(1+\g^2)-0.05}) node  [above right] {$\frac{i-\alpha/\gamma}{i+\alpha/\gamma}=B'$};
\filldraw (B) circle({0.6*1.5/1pt});
\draw  ({(-1+(1/\L)^2)/(1+(1/\L)^2)+0.2},{(2*(1/\L))/(1+(1/\L)^2)-0.15}) node  [above left] {$\frac{i-\alpha L}{i+\alpha L}=A'$};

\def\w{-0.5};
\draw [->] (-.4,-.6) -- ({sqrt((2*(1/\L)/(1-(1/\L)^2))^2-(\w)^2)+(1+(1/\L)^2)/(-1+(1/\L)^2)},\w);
\draw (-.4,-.6)  node [left, fill=white] {$\cG\left(2J_{\alpha\cA}-I\right)$};
\end{tikzpicture}
}
\end{tabular}
\end{center}
Note the line segment $\overline{AB}$ is mapped to the (minor) arc $\arc{A'B'}$.
Using Theorem~\ref{thm:composition-srg}, we have
\begin{center}
\begin{tabular}{ccccccc}
\raisebox{-.5\height}{
\begin{tikzpicture}[scale=0.9]
\def\L{1.2}
\def\g{2}
\begin{scope}
\clip (0,0) circle (1);
\fill [fill=medgrey] (0,0) circle (1);
\fill [fill=white]({(1+(1/\L)^2)/(-1+(1/\L)^2)},0) circle ({2*(1/\L)/(1-(1/\L)^2)});
\fill [fill=white]({(1+\g^2)/(-1+\g^2)},0) circle ({2*\g/(1-\g^2)});
\end{scope}
\draw [<->] (-1.2,0) -- (1.2,0);
\draw [<->] (0,-1.2) -- (0,1.2);
\draw [dashed] (0,0) circle (1);

\filldraw (1,0) circle({0.6*1.5/0.9pt});
\draw (1,0) node [above left] {$1$};

\draw (0,-1.4) node {$\cG(2J_{\alpha \cA}-I)$};
\end{tikzpicture}
}
&\!\!\!\!\!\!
$\times$
&\!\!\!\!\!\!
\raisebox{-.5\height}{
\begin{tikzpicture}[scale=0.9]
\fill [fill=medgrey] (0,0) circle (1);
\draw [<->] (-1.2,0) -- (1.2,0);
\draw [<->] (0,-1.2) -- (0,1.2);
\filldraw (1,0) circle({0.6*1.5/0.9pt});
\draw (1,0) node [above left] {$1$};
\draw (0,-1.4) node {$\cG(2J_{\alpha \cB}-I)$};
\end{tikzpicture}
}
&\!\!\!\!\!\!
$=$
&\!\!\!\!\!\!\!\!\!\!\!\!\!\!\!\!\!\!
\raisebox{-.5\height}{
\begin{tikzpicture}[scale=0.9]
\fill [fill=medgrey] (0,0) circle (1);
\draw [<->] (-1.2,0) -- (1.2,0);
\draw [<->] (0,-1.2) -- (0,1.2);
\filldraw (1,0) circle({0.6*1.5/0.9pt});
\draw (1,0) node [above left] {$1$};
\draw (0,-1.4) node {$\cG\left((2J_{\alpha \cA}-I)(2J_{\alpha \cB}-I)\right)$};
\end{tikzpicture}}
&\!\!\!\!\!\!\!\!\!\!\!\!\!\!\!\!\!\!
$\xrightarrow{(1-\theta)+\theta z}$
&\!\!\!\!\!\!\!\!\!
\raisebox{-.5\height}{
\begin{tikzpicture}[scale=0.9]
\def\t{0.8}
\fill [fill=medgrey] ({1-\t},0) circle (\t);
\draw [<->] (-1.2,0) -- (1.2,0);
\draw [<->] (0,-1.2) -- (0,1.2);
\filldraw (1,0) circle({0.6*1.5/0.9pt});
\draw (1,0) node [above left] {$1$};
\filldraw ({1-\t},0) circle({0.6*1.5/0.9pt});
\draw ({1-(\t/2)},-0.1) node [below] {$\theta$};
\draw [decorate,decoration={brace,amplitude=4.5pt}]  (1,0) --({1-\t},0) ;
%\draw ({1-\t},0) node [below right] {$1-\theta$};
\draw (0,-1.4) node {$\cG\left(D_{\alpha,\theta}(\cA,\cB)\right)$};
\end{tikzpicture}}
\end{tabular}
\end{center}
because $\arc{A'B'}$ is on the unit circle, and since $\cG(2J_{\alpha \cB}-I)=\{z\,|\,|z|\le 1\}$.
So we have $1\in \cG(D_{\alpha,\theta}(\cA,\cB))$, but $1\notin  \cG(\cL_{1-\varepsilon})$ for any $\varepsilon>0$.
Therefore,  $\cG (D_{\alpha,\theta}(\cA,\cB))\nsubseteq \cG(\cL_{1-\varepsilon})$ and, with Theorem~\ref{thm:srg-full}, we conclude $D_{\alpha,\theta}(\cA,\cB)\nsubseteq \cL_{1-\varepsilon}$.
%Since $\cG\left(\cA\right)\nsubseteq \cG(\cL_R)$ and since $\cL_R$ is SRG-representable, we have $\cA\nsubseteq \cL_R$ by Theorem~\ref{thm:srg-representation}.
The result for
%the cases $\alpha/\gamma\ge 1$ or $1\ge\alpha L$ and for
the operator $D_{\alpha,\theta}(\cB,\cA)$ follows from similar reasoning.
\qed\end{proof}

\begin{corollary}
\label{cor:ms_impossibility1}
Let $0<1/\gamma\le L<\infty$ and $\alpha\in(0,\infty)$.
Let $B\in\cM$ and let $A\in \cM\cap \cL_L$ satisfy a condition weaker than or equal to $\gamma$-inverse Lipschitz continuity, such as $\gamma$-metric subregularity.
It is not possible to establish a strict contraction of the DRS operators $D_{\alpha,\theta}(A,B)$ or $D_{\alpha,\theta}(B,A)$ for any $\alpha>0$ and $\theta\ne 0$ without further assumptions.
\end{corollary}

\begin{theorem}
Let $\gamma,L,\alpha \in(0,\infty)$.
Let $\cA=\cM\cap (\cL_\gamma)^{-1}$ and $\cB=\cM\cap  \cL_L$.
If $1/\gamma\le L$ and $\theta\ne 0$, then
\[D_{\alpha,\theta}(\cA,\cB)\nsubseteq \cL_{1-\varepsilon}\]
for any $\varepsilon>0$.
% \[
% (1-\theta)I+\theta(2J_{\alpha \cA}-I)(2J_{\alpha \cB}-I)\nsubseteq  \cL_{1-\varepsilon}
% \]
% and 
% \[
% (1-\theta)I+\theta(2J_{\alpha \cB}-I)(2J_{\alpha \cA}-I)\nsubseteq  \cL_{1-\varepsilon}
% \]
% for any $\varepsilon>0$.
If $1/\gamma> L$ and $\theta\in(0,1)$, then 
$D_{\alpha,\theta}(\cA,\cB)\subseteq \cL_R$
%(1-\theta)I+\theta(2J_{\alpha \cA}-I)(2J_{\alpha \cB}-I)\subseteq \cL_R
%(1-\theta)I+\theta(2J_{\alpha \cB}-I)(2J_{\alpha \cA}-I)\subseteq \cL_R
for 
\[
R= \sqrt{1-\frac{4\theta(1-\theta)(1-\gamma L)^2}{(1+\gamma^2/\alpha^2)(1+\alpha^2 L^2)}}.
\]
This result is tight in the sense that 
$D_{\alpha,\theta}(\cA,\cB)\nsubseteq \cL_R$ for any smaller value of $R$.
The same conclusion holds for $D_{\alpha,\theta}(\cB,\cA)$.
\end{theorem}
\begin{proof}
Consider the case $\alpha/\gamma<1$ and $\alpha L<1$. We have
\begin{center}
\begin{tabular}{cc}
\raisebox{-.5\height}{
\begin{tikzpicture}[scale=1]
\def\g{2.2}

%\fill [fill=lightgrey] (0,0) circle ({sqrt(1/(1+(1/\g)^2))});
\fill[fill=medgrey] (0,0) circle (1);
\begin{scope}
\fill[fill=white]({2/(1-(1/\g)^2)-1},0) circle ({(2/\g)/(1-(1/\g)^2)});
\end{scope}
\begin{scope}
\clip  ({2/(1+(1/\g)^2)-1},-1.3) rectangle (2.7,1.3);
\draw[dashed]({2/(1-(1/\g)^2)-1},0) circle ({(2/\g)/(1-(1/\g)^2)});
\draw[dashed] (0,0) circle (1);
\end{scope}

\draw({2/(1-(1/\g))-1},0)  node [below left] {$\frac{1+\alpha/\gamma}{1-\alpha/\gamma}$};
\filldraw ({2/(1-(1/\g)^2)-1},0)   circle ({0.6*1.5/1pt});
\draw ({2/(1-(1/\g)^2)-1+0.1},0)   node [above ] {$\frac{1+\alpha^2/\gamma^2}{1-\alpha^2/\gamma^2}$};

\draw [<->] (-1.2,0) -- (3,0);
\draw [<->] (0,-1.2) -- (0,1.2);

\draw[->] (-.02,0.25)--({2/(1+(1/\g))-1},.05);
\draw  (.1,0.25) node [left] {$\frac{1-\alpha/\gamma}{1+\alpha/\gamma}$};
\filldraw ({2/(1+(1/\g))-1},0) circle ({0.6*1.5/1pt});
\filldraw ({2/(1-(1/\g))-1},0)  circle ({0.6*1.5/1pt});

%\draw [dashed] (0,0) circle (1);
\draw (1,-0pt) node [below right] {$1$};
\filldraw (1,0) circle ({0.6*1.5/1pt});
%
%\def\x{0.15}
%\draw [->] (-1.2,.8) -- (\x,{sqrt((1/2)^2-(\x-(1/2))^2)});
%\draw (-1.2,.8) node [above, fill=white] {$\cG\left(\left(I+\alpha (\cM\cap \cL_L^{-1})\right)^{-1}\right)$};
%
%
%\def\y{0.6}
%\draw [->] (1.2,-1.) -- (\y,{-sqrt((sqrt(1/(1+(1/\g)^2)))^2-(\y)^2)});
%\draw (1.2,-1.) node [right, fill=white] {$\cG\left(\cL\left(\frac{\gamma}{\sqrt{\alpha^2+\gamma^2}}\right)\right)$};
\draw (0.,-1.2)  node [below] {$\cG(2J_{\alpha \cA}-I)$};
\end{tikzpicture}}
&
\raisebox{-.5\height}{
\begin{tikzpicture}[scale=1]
\def\g{2.2}
\coordinate (B) at ({2/(1+(1/\g))-1},0) ;
\coordinate (A) at ({(1-(1/\g)^2)/(1+(1/\g)^2)},{2*(1/\g)/(1+(1/\g)^2)});
\coordinate (Ap) at ({(1-(1/\g)^2)/(1+(1/\g)^2)},{-2*(1/\g)/(1+(1/\g)^2)}) ;
\coordinate (D) at ({2/(1-(1/\g)^2)-1},0);
\coordinate (O) at (0,0);
%\fill [fill=lightgrey] (0,0) circle ({sqrt(1/(1+(1/\g)^2))});
\fill[fill=medgrey] (0,0) circle (1);
\begin{scope}
\clip (0,0) circle (1);
\fill[fill=white]({2/(1-(1/\g)^2)-1},0) circle ({(2/\g)/(1-(1/\g)^2)});
\end{scope}
\begin{scope}
\clip  ({2/(1+(1/\g)^2)-1},-1.3) rectangle (2.7,1.3);
\draw[dashed]({2/(1-(1/\g)^2)-1},0) circle ({(2/\g)/(1-(1/\g)^2)});
\draw[dashed] (0,0) circle (1);
\end{scope}

\begin{scope}
\clip ({2/(1-(1/\g)^2)-1},0) circle ({(2/\g)/(1-(1/\g)^2)});
%\fill [fill=darkgrey] (A)--(B)--(Ap) -- (0,0)-- cycle;
\fill [fill=darkgrey] (A)--(Ap) -- (0,0)-- cycle;
\end{scope}

\begin{scope}
\clip  ({2/(1+(1/\g)^2)-1},-1.3) rectangle (-1.2,1.3);
\draw (0,0) circle (1);
\end{scope}

%\filldraw ({2/(1+(1/\g))-1},0) circle ({0.6*1.5/1pt});
%\draw ({2/(1+(1/\g))-1},0)  node [above left] {$\frac{1-\alpha/\gamma}{1+\alpha/\gamma}$};
%\filldraw ({2/(1-(1/\g))-1},0)  circle ({0.6*1.5/1pt});
%\draw({2/(1-(1/\g))-1},0)  node [above right] {$\frac{1+\alpha/\gamma}{1-\alpha/\gamma}$};
%\filldraw ({2/(1-(1/\g)^2)-1},0)   circle ({0.6*1.5/1pt});
%\draw ({2/(1-(1/\g)^2)-1},0)   node [above ] {$\frac{1+\alpha^2/\gamma^2}{1-\alpha^2/\gamma^2}$};

\draw [<->] (-1.2,0) -- (3,0);
\draw [<->] (0,-1.2) -- (0,1.2);
%\draw (1,-0pt) node [below right] {$1$};
%\filldraw (1,0) circle ({0.6*1.5/1pt});

%\fill (B) circle ({0.6*1.5/1pt});
\fill (A)   circle ({0.6*1.5/1pt});
\fill (Ap) circle ({0.6*1.5/1pt});
\fill (D) circle ({0.6*1.5/1pt});

\draw (A) node[above] {$A$};
\draw ({(1-(1/\g)^2)/(1+(1/\g)^2)-0.05},{-2*(1/\g)/(1+(1/\g)^2)-0.05})  node[below] {$A'$};
\draw (D) node[below] {$D$};
%\draw (B) node[below left] {$B$};
%\draw (O) node[below left] {$O$};

\fill (O) circle ({0.6*1.5/1pt});

\draw  (O)--(A)--(D);

\fill (-1,0) circle ({0.6*1.5/1pt});
%\draw (-1,0) node[above right] {$C$};

\def\r{0.2};
\draw (\r,{0}) arc (0:{acos((-1+\g^2)/(1+\g^2))}:\r) node[pos=0.9, right]{$\varphi_A$};

\draw (0.,-1.2)  node [below] {\phantom{$\cG(2J_{\alpha \cA}-I)$}};

\end{tikzpicture}}
\end{tabular}
\end{center}
Let $\underline{S_A}=\arc{ACA'}$  and let $\overline{S_A}$ as the region bounded by  $\arc{ACA'}\cup \overline{AA'}$.
These sets provide an inner and outer bound of $\cG(2J_{\alpha \cA}-1)$ in the sense that
\[
\underline{S_A}
\subseteq 
\cG(2J_{\alpha \cA}-1)
\subseteq 
\overline{S_A}.
\]
Note that $J_{\alpha \cA}$ satisfies the left-arc property.
By the law of cosines, we have
% \begin{align*}
% \overline{AO}^2+
% \overline{OD}^2
% -2\cdot \overline{AO}\cdot
% \overline{OD}\cdot 
% \cos(\varphi_A)=
% \overline{AD}^2
% \end{align*}
% reorganizing we get
\begin{align*}
\cos(\varphi_A)&=
\frac{1}{2\cdot \overline{AO}\cdot
\overline{OD}}
\left(
\overline{AO}^2+
\overline{OD}^2-\overline{AD}^2\right)\\
&=
\tfrac{1}{2\cdot 1\cdot
\left(\tfrac{1+\alpha^2/\gamma^2}{1-\alpha^2/\gamma^2}\right)
}
\left(
1^2+
\left(\tfrac{1+\alpha^2/\gamma^2}{1-\alpha^2/\gamma^2}\right)^2
-
\left(
\tfrac{1+\alpha^2/\gamma^2}{1-\alpha^2/\gamma^2}
-
\tfrac{1-\alpha/\gamma}{1+\alpha/\gamma}
\right)^2
\right)=\tfrac{1-\alpha^2/\gamma^2}{1+\alpha^2/\gamma^2}.
\end{align*}
Likewise, we have
\begin{center}
\begin{tabular}{cc}
\raisebox{-.5\height}{
\begin{tikzpicture}[scale=1.2]
\def\L{.35};

\begin{scope}
\clip ({(1-\L^2)/(1+\L^2)},-1.2)  rectangle (2.3,1.2);
\draw[dashed] ({2*(1/(1-\L^2))-1},0) circle ({2*\L/(1-\L^2)});
\end{scope}

\begin{scope}
\clip (0,0) circle (1);
\fill[fill=medgrey] ({2*(1/(1-\L^2))-1},0) circle ({2*\L/(1-\L^2)});
%\draw[line width=1.0pt] ({(1+2*\b)/(2+2*\b)},0) circle ({(1)/(2+2*\b)});
\end{scope}

\draw [<->] (-1.2,0) -- (2.3,0);
\draw [<->] (0,-1.2) -- (0,1.2);

\begin{scope}
\clip ({(1-\L^2)/(1+\L^2)},-1.2)  rectangle (-1.2,1.2);
\draw [dashed] (0,0) circle (1);
\end{scope}

\filldraw (1,0) circle ({0.6*1.5/1.2pt});
\draw (1,-0) node [below right] {$1$};

\filldraw ({(1+\L^2)/(1-\L^2)},0)circle ({0.6*1.5/1.2pt});
\draw ({(1+\L^2)/(1-\L^2)-0.2-0.15},0) node [above right] {$\frac{1+\alpha^2 L^2}{1-\alpha^2 L^2}$};

\filldraw ({(1-\L)/(1+\L)},0)circle ({0.6*1.5/1.2pt});
\draw [->] (-0.15,.25) -- ({(1-\L)/(1+\L)},0);
\draw (-0.1,.25) node [left] {$\frac{1-\alpha L}{1+\alpha L}$};

\filldraw ({(1+\L)/(1-\L)},0)circle ({0.6*1.5/1.2pt});
%\draw [->] (1.75,-.25) -- ({(1+\L)/(1-\L)},0);
\draw ({(1+\L)/(1-\L)},0) node [above right] {$\frac{1+\alpha L}{1-\alpha L}$};

\draw (1.,-1.)  node [below] {$\cG(2J_{\alpha \cB}-I)$};

%
%\def\x{.43}
%\draw (.2,-.9) node [below,fill=white] {$\cG\left(2\left(I+\alpha(\cM_\mu\cap \cC_\beta)\right)^{-1}-I\right)$};
%\draw [->] (.1,-.9) -- ({\x},{-sqrt((2*\L/(1-\L^2))^2-(\x-(1+\L^2)/(1-\L^2))^2)});

%
%\def\y{.3}
%\draw [->] (.4,.7) -- (\y,{sqrt((sqrt((\b + \m - 2*\b*\m)/   (\b + \m + 2*\b*\m)))^2-(\y)^2)});
%\draw (.4,.7) node [above,fill=white] {$\cG\left(\cL\left(\sqrt{1-\frac{4\alpha\mu}{1+2\alpha\mu+\alpha^2\mu/\beta}}\right)\right)$};
\end{tikzpicture}}
&
\raisebox{-.5\height}{
\begin{tikzpicture}[scale=1.2]
\def\L{.35};

\coordinate (O) at (0,0) ;
\coordinate (A) at ({(1-\L^2)/(1+\L^2)},{2*\L/(1+\L^2)});
\coordinate (Ap) at ({(1-\L^2)/(1+\L^2)},{-2*\L/(1+\L^2)}) ;
\fill [fill=darkgrey] (O)--(A)--(Ap);

\begin{scope}
\clip ({(1-\L^2)/(1+\L^2)},-1.2)  rectangle (2.3,1.2);
\draw[dashed] ({2*(1/(1-\L^2))-1},0) circle ({2*\L/(1-\L^2)});
\end{scope}

\begin{scope}
\clip (0,0) circle (1);
\fill[fill=medgrey] ({2*(1/(1-\L^2))-1},0) circle ({2*\L/(1-\L^2)});
%\draw[line width=1.0pt] ({(1+2*\b)/(2+2*\b)},0) circle ({(1)/(2+2*\b)});
\end{scope}

\begin{scope}
\clip ({(1-\L^2)/(1+\L^2)},-1.2)  rectangle (1.2,1.2);
\draw (0,0) circle (1);
\end{scope}

\draw [<->] (-1.2,0) -- (2.3,0);
\draw [<->] (0,-1.2) -- (0,1.2);

\begin{scope}
\clip ({(1-\L^2)/(1+\L^2)},-1.2)  rectangle (-1.2,1.2);
\draw [dashed] (0,0) circle (1);
\end{scope}

\fill (O) circle ({0.6*1.5/1.2pt});
\fill (A)   circle ({0.6*1.5/1.2pt});
\fill (Ap) circle ({0.6*1.5/1.2pt});
\fill (1,0) circle ({0.6*1.5/1.2pt});

\draw (A) node[above] {$A$};
\draw (Ap) node[below] {$A'$};
\draw (O) node[below left] {$O$};
\draw (1,0) node[above right] {$C$};

\def\r{0.2};
\draw (\r,{0}) arc (0:{acos((-1+(1/\L^2))/(1+(1/\L^2)))}:\r) node[pos=0.9, right]{$\varphi_B$};

\draw (1.,-1.)  node [below] {\phantom{$\cG(2J_{\alpha \cB}-I)$}};

%
%\def\x{.43}
%\draw (.2,-.9) node [below,fill=white] {$\cG\left(2\left(I+\alpha(\cM_\mu\cap \cC_\beta)\right)^{-1}-I\right)$};
%\draw [->] (.1,-.9) -- ({\x},{-sqrt((2*\L/(1-\L^2))^2-(\x-(1+\L^2)/(1-\L^2))^2)});

%
%\def\y{.3}
%\draw [->] (.4,.7) -- (\y,{sqrt((sqrt((\b + \m - 2*\b*\m)/   (\b + \m + 2*\b*\m)))^2-(\y)^2)});
%\draw (.4,.7) node [above,fill=white] {$\cG\left(\cL\left(\sqrt{1-\frac{4\alpha\mu}{1+2\alpha\mu+\alpha^2\mu/\beta}}\right)\right)$};
\end{tikzpicture}}
\end{tabular}
\end{center}
Let $\underline{S_B}=\arc{A'CA}$ and let $\overline{S_B}$ as the circular sector bounded by  $\arc{A'CA}\cup \overline{AO}\cup\overline{OA'}$.
Again, we have
\[
\underline{S_B}
\subseteq 
\cG(2J_{\alpha \cB}-1)
\subseteq 
\overline{S_B},
\]
and 
\[
\cos\varphi_B=\tfrac{1-\alpha^2 L^2}{1+\alpha^2 L^2}.
\]
Using the arccosine sum identity \cite[p. 80, 4.4.33]{abramowitz_stegun}, we get
\[
\cos(\varphi_A-\varphi_B)
=1-\tfrac{2(\alpha L-\alpha/\gamma)^2}{(1+\alpha ^2L^2)(1+\alpha^2/\gamma^2)}.
\]

When $1/\gamma\le L$, we have $\varphi_A\le \varphi_B$. 
In this case, 
\begin{center}
\begin{tabular}{ccccc}
\raisebox{-.5\height}{
\begin{tikzpicture}[scale=1]
\def\g{3}
\coordinate (B) at ({2/(1+(1/\g))-1},0) ;
\coordinate (A) at ({(1-(1/\g)^2)/(1+(1/\g)^2)},{2*(1/\g)/(1+(1/\g)^2)});
\coordinate (Ap) at ({(1-(1/\g)^2)/(1+(1/\g)^2)},{-2*(1/\g)/(1+(1/\g)^2)}) ;
\coordinate (O) at (0,0);
%\fill [fill=lightgrey] (0,0) circle ({sqrt(1/(1+(1/\g)^2))});
\begin{scope}
\clip  ({2/(1+(1/\g)^2)-1},-1.2) rectangle (-1.2,1.2);
%\fill[fill=medgrey] (0,0) circle (1);
\draw(0,0) circle (1);
\end{scope}

\draw [<->] (-1.2,0) -- (1.2,0);
\draw [<->] (0,-1.2) -- (0,1.2);
\filldraw (A) circle ({0.6*1.5/1pt});
\filldraw (Ap) circle ({0.6*1.5/1pt});

\draw[dashed]  (O)--(A);

\def\r{0.3};
\draw (\r,{0}) arc (0:{acos((-1+\g^2)/(1+\g^2))}:\r) node[pos=0.9, right]{$\varphi_A$};

\draw (0.,-1.2)  node [below] {$\underline{S_A}$};

\end{tikzpicture}}
&
$\times$
&
\raisebox{-.5\height}{
\begin{tikzpicture}[scale=1]
\def\L{.7};

\coordinate (O) at (0,0) ;
\coordinate (A) at ({(1-\L^2)/(1+\L^2)},{2*\L/(1+\L^2)});
\coordinate (Ap) at ({(1-\L^2)/(1+\L^2)},{-2*\L/(1+\L^2)}) ;
%\fill [fill=darkgrey] (O)--(A)--(Ap);

\draw[dashed] (O)--(A);

\begin{scope}
\clip ({(1-\L^2)/(1+\L^2)},-1.2)  rectangle (1.2,1.2);
\draw (0,0) circle (1);
\end{scope}

\draw [<->] (-1.2,0) -- (1.2,0);
\draw [<->] (0,-1.2) -- (0,1.2);

\def\r{0.3};
\draw (\r,{0}) arc (0:{acos((-1+(1/\L^2))/(1+(1/\L^2)))}:\r) node[pos=0.9, right]{$\varphi_B$};

\filldraw (A) circle ({0.6*1.5/1pt});
\filldraw (Ap) circle ({0.6*1.5/1pt});

\draw (0.,-1.2)  node [below] {$\underline{S_B}$};

\end{tikzpicture}}
&
$=$
&
\raisebox{-.5\height}{
\begin{tikzpicture}[scale=1]

\draw(0,0) circle (1);

\draw [<->] (-1.2,0) -- (1.2,0);
\draw [<->] (0,-1.2) -- (0,1.2);

\draw (0.,-1.2)  node [below] {$\{z\in \complex\,|\,|z|=1\}$};
\end{tikzpicture}}
\end{tabular}
\end{center}
Therefore
\[
1\in 
(1-\theta)1+\theta
\underline{S_A}\,
\underline{S_B}
\subseteq 
\cG (D_{\alpha,\theta}(\cA,\cB)),
\]
but $1\notin  \cG(\cL_{1-\varepsilon})$ for any $\varepsilon>0$.
Therefore, we conclude $\cG (D_{\alpha,\theta}(\cA,\cB))\nsubseteq \cG(\cL_{1-\varepsilon})$.

%-------------------------------------------------------------------------------------------------------------------------------------------------------------------------------

When $1/\gamma> L$, we have $\varphi_A>\varphi_B$. In this case,
\begin{center}
\begin{tabular}{ccccc}
\raisebox{-.5\height}{
\begin{tikzpicture}[scale=1]
\def\g{1.7}
\coordinate (B) at ({2/(1+(1/\g))-1},0) ;
\coordinate (A) at ({(1-(1/\g)^2)/(1+(1/\g)^2)},{2*(1/\g)/(1+(1/\g)^2)});
\coordinate (Ap) at ({(1-(1/\g)^2)/(1+(1/\g)^2)},{-2*(1/\g)/(1+(1/\g)^2)}) ;
\coordinate (O) at (0,0);
%\fill [fill=lightgrey] (0,0) circle ({sqrt(1/(1+(1/\g)^2))});
\begin{scope}
\clip  ({2/(1+(1/\g)^2)-1},-1.2) rectangle (-1.2,1.2);
%\fill[fill=medgrey] (0,0) circle (1);
\draw(0,0) circle (1);
\end{scope}

\draw [<->] (-1.2,0) -- (1.2,0);
\draw [<->] (0,-1.2) -- (0,1.2);
%\draw (1,-0pt) node [below right] {$1$};
%\filldraw (1,0) circle ({0.6*1.5/1pt});

\filldraw (A) circle ({0.6*1.5/1pt});
\filldraw (Ap) circle ({0.6*1.5/1pt});
\draw[dashed]  (O)--(A);

\def\r{0.3};
\draw (\r,{0}) arc (0:{acos((-1+\g^2)/(1+\g^2))}:\r) node[pos=0.9, right]{$\varphi_A$};

\draw (0.,-1.2)  node [below] {$\underline{S_A}$};

\end{tikzpicture}}
&
$\times$
&
\raisebox{-.5\height}{
\begin{tikzpicture}[scale=1]
\def\L{.2};
\coordinate (O) at (0,0) ;
\coordinate (A) at ({(1-\L^2)/(1+\L^2)},{2*\L/(1+\L^2)});
\coordinate (Ap) at ({(1-\L^2)/(1+\L^2)},{-2*\L/(1+\L^2)}) ;
%\fill [fill=darkgrey] (O)--(A)--(Ap);

\draw[dashed] (O)--(A);

\begin{scope}
\clip ({(1-\L^2)/(1+\L^2)},-1.2)  rectangle (1.2,1.2);
\draw (0,0) circle (1);
\end{scope}

\draw [<->] (-1.2,0) -- (1.2,0);
\draw [<->] (0,-1.2) -- (0,1.2);

\def\r{0.3};
\draw (\r,{0}) arc (0:{acos((-1+(1/\L^2))/(1+(1/\L^2)))}:\r);
\draw (0.75,0.45) node [above] {$\varphi_B$};
\draw [->] (0.7,0.5) -- ({\r*cos(acos((-1+(1/\L^2))/(1+(1/\L^2)))/2)},{\r*sin(acos((-1+(1/\L^2))/(1+(1/\L^2)))/2)});
\filldraw (A) circle ({0.6*1.5/1pt});
\filldraw (Ap) circle ({0.6*1.5/1pt});

\draw (0.,-1.2)  node [below] {$\underline{S_B}$};

\end{tikzpicture}}
&
$=$
&
\raisebox{-.5\height}{
\begin{tikzpicture}[scale=1]
\def\g{1.7}
\def\L{.2};

\coordinate (A) at ({cos(acos((-1+\g^2)/(1+\g^2))-acos((-1+(1/\L^2))/(1+(1/\L^2))))},{sin(acos((-1+\g^2)/(1+\g^2))-acos((-1+(1/\L^2))/(1+(1/\L^2)))});
\coordinate (Ap) at ({cos(acos((-1+\g^2)/(1+\g^2))-acos((-1+(1/\L^2))/(1+(1/\L^2))))},{-sin(acos((-1+\g^2)/(1+\g^2))-acos((-1+(1/\L^2))/(1+(1/\L^2)))});

\coordinate (O) at (0,0);
%\fill [fill=lightgrey] (0,0) circle ({sqrt(1/(1+(1/\g)^2))});
\begin{scope}
\clip  (({cos(acos((-1+\g^2)/(1+\g^2))-acos((-1+(1/\L^2))/(1+(1/\L^2))))},-1.2) rectangle (-1.2,1.2);
\draw(0,0) circle (1);
\end{scope}

%\filldraw ({2/(1+(1/\g))-1},0) circle ({0.6*1.5/1pt});
%\draw ({2/(1+(1/\g))-1},0)  node [above left] {$\frac{1-\alpha/\gamma}{1+\alpha/\gamma}$};
%\filldraw ({2/(1-(1/\g))-1},0)  circle ({0.6*1.5/1pt});
%\draw({2/(1-(1/\g))-1},0)  node [above right] {$\frac{1+\alpha/\gamma}{1-\alpha/\gamma}$};
%\filldraw ({2/(1-(1/\g)^2)-1},0)   circle ({0.6*1.5/1pt});
%\draw ({2/(1-(1/\g)^2)-1},0)   node [above ] {$\frac{1+\alpha^2/\gamma^2}{1-\alpha^2/\gamma^2}$};

\draw [<->] (-1.2,0) -- (1.2,0);
\draw [<->] (0,-1.2) -- (0,1.2);
%\draw (1,-0pt) node [below right] {$1$};
%\filldraw (1,0) circle ({0.6*1.5/1pt});

\draw[dashed]  (O)--(A);

%\filldraw (A) circle ({0.6*1.5/1pt});

\def\r{0.3};
\draw (\r,{0}) arc (0:{acos((-1+\g^2)/(1+\g^2))-acos((-1+(1/\L^2))/(1+(1/\L^2)))}:\r);
\draw (1,0.2) node {$\varphi_A-\varphi_B$};
%\draw (\r,{0}) arc (0:{acos((-1+\g^2)/(1+\g^2))}:\r) node[pos=0.9, right]{$\varphi_A-\varphi_B$};
%\draw (\r,{0}) arc (0:{acos((-1+(1/\L^2))/(1+(1/\L^2)))}:\r) node[pos=0.9, right]{$\varphi_A-\varphi_B$};

\filldraw (A) circle ({0.6*1.5/1pt});
\filldraw (Ap) circle ({0.6*1.5/1pt});

\draw (0.,-1.2)  node [below] {$\underline{S_A}\,\underline{S_B}$};

\end{tikzpicture}}
\end{tabular}
\end{center}

\begin{center}
\begin{tabular}{ccccccc}
\raisebox{-.5\height}{
\begin{tikzpicture}[scale=1]
\def\g{1.7}
\coordinate (B) at ({2/(1+(1/\g))-1},0) ;
\coordinate (A) at ({(1-(1/\g)^2)/(1+(1/\g)^2)},{2*(1/\g)/(1+(1/\g)^2)});
\coordinate (Ap) at ({(1-(1/\g)^2)/(1+(1/\g)^2)},{-2*(1/\g)/(1+(1/\g)^2)}) ;
\coordinate (O) at (0,0);
%\fill [fill=lightgrey] (0,0) circle ({sqrt(1/(1+(1/\g)^2))});
\begin{scope}
\clip  ({2/(1+(1/\g)^2)-1},-1.2) rectangle (-1.2,1.2);
%\fill[fill=medgrey] (0,0) circle (1);
\fill[fill=medgrey](0,0) circle (1);
\end{scope}

\draw [<->] (-1.2,0) -- (1.2,0);
\draw [<->] (0,-1.2) -- (0,1.2);
%\draw (1,-0pt) node [below right] {$1$};
%\filldraw (1,0) circle ({0.6*1.5/1pt});

\draw[dashed]  (O)--(A);

\def\r{0.3};
\draw (\r,{0}) arc (0:{acos((-1+\g^2)/(1+\g^2))}:\r) node[pos=0.9, right]{$\varphi_A$};

\draw (0.,-1.2)  node [below] {$\overline{S_A}$};

\end{tikzpicture}}
&
\!\!\!\!\!\!\!\!
$\times$
\!\!\!\!\!\!\!\!
&
\raisebox{-.5\height}{
\begin{tikzpicture}[scale=1]
\def\L{.2};
\coordinate (O) at (0,0) ;
\coordinate (A) at ({(1-\L^2)/(1+\L^2)},{2*\L/(1+\L^2)});
\coordinate (Ap) at ({(1-\L^2)/(1+\L^2)},{-2*\L/(1+\L^2)}) ;
%\fill [fill=darkgrey] (O)--(A)--(Ap);

\fill [fill=medgrey] (O)--(A)--(Ap);

\begin{scope}
\clip (0,0) circle (1);
\fill[fill=medgrey] ({2*(1/(1-\L^2))-1},0) circle ({2*\L/(1-\L^2)});
%\draw[line width=1.0pt] ({(1+2*\b)/(2+2*\b)},0) circle ({(1)/(2+2*\b)});
\end{scope}

\draw [<->] (-1.2,0) -- (1.2,0);
\draw [<->] (0,-1.2) -- (0,1.2);

\def\r{0.3};
\draw (\r,{0}) arc (0:{acos((-1+(1/\L^2))/(1+(1/\L^2)))}:\r);
\draw (0.85,0.45) node [above] {$\varphi_B$};
\draw [->] (0.8,0.5) -- ({\r*cos(acos((-1+(1/\L^2))/(1+(1/\L^2)))/2)},{\r*sin(acos((-1+(1/\L^2))/(1+(1/\L^2)))/2)});

\draw (0.,-1.2)  node [below] {$\overline{S_B}$};

\end{tikzpicture}}
&
\!\!\!\!\!\!\!\!
$=$
\!\!\!\!\!\!\!\!\!\!\!\!\!\!\!\!\!
&
\raisebox{-.5\height}{
\begin{tikzpicture}[scale=1]
\def\g{1.7}
\def\L{.2};

\coordinate (A) at ({cos(acos((-1+\g^2)/(1+\g^2))-acos((-1+(1/\L^2))/(1+(1/\L^2))))},{sin(acos((-1+\g^2)/(1+\g^2))-acos((-1+(1/\L^2))/(1+(1/\L^2)))});
\coordinate (Ap) at ({cos(acos((-1+\g^2)/(1+\g^2))-acos((-1+(1/\L^2))/(1+(1/\L^2))))},{-sin(acos((-1+\g^2)/(1+\g^2))-acos((-1+(1/\L^2))/(1+(1/\L^2)))});

\coordinate (B) at ({cos(-acos((-1+\g^2)/(1+\g^2))-acos((-1+(1/\L^2))/(1+(1/\L^2))))},{sin(-acos((-1+\g^2)/(1+\g^2))-acos((-1+(1/\L^2))/(1+(1/\L^2)))});\coordinate (Bp) at ({cos(-acos((-1+\g^2)/(1+\g^2))-acos((-1+(1/\L^2))/(1+(1/\L^2))))},{-sin(-acos((-1+\g^2)/(1+\g^2))-acos((-1+(1/\L^2))/(1+(1/\L^2)))});

\coordinate (O) at (0,0);
%\fill [fill=lightgrey] (0,0) circle ({sqrt(1/(1+(1/\g)^2))});

%\filldraw ({2/(1+(1/\g))-1},0) circle ({0.6*1.5/1pt});
%\draw ({2/(1+(1/\g))-1},0)  node [above left] {$\frac{1-\alpha/\gamma}{1+\alpha/\gamma}$};
%\filldraw ({2/(1-(1/\g))-1},0)  circle ({0.6*1.5/1pt});
%\draw({2/(1-(1/\g))-1},0)  node [above right] {$\frac{1+\alpha/\gamma}{1-\alpha/\gamma}$};
%\filldraw ({2/(1-(1/\g)^2)-1},0)   circle ({0.6*1.5/1pt});
%\draw ({2/(1-(1/\g)^2)-1},0)   node [above ] {$\frac{1+\alpha^2/\gamma^2}{1-\alpha^2/\gamma^2}$};

%\draw (1,-0pt) node [below right] {$1$};
%\filldraw (1,0) circle ({0.6*1.5/1pt});

\begin{scope}
\clip  (A)--(B)--(0,-1.2)--(-1.2,-1.2)--(-1.2,1.2)--(1.2,1.2)--cycle;
\fill[fill=medgrey](0,0) circle (1);
\end{scope}

\begin{scope}
\clip  (Ap)--(Bp)--(0,1.2)--(-1.2,1.2)--(-1.2,-1.2)--(1.2,-1.2)--cycle;
\fill[fill=medgrey](0,0) circle (1);
\end{scope}
\draw[dashed]  (O)--(A);
\draw[dashed]  (O)--(B);
\draw[dashed]  (A)--(B);
\def\r{0.3};
\draw (\r,{0}) arc (0:{acos((-1+\g^2)/(1+\g^2))-acos((-1+(1/\L^2))/(1+(1/\L^2)))}:\r);
\draw (-1.2,.7) node [above] {$\varphi_A-\varphi_B$};

\draw [->] (-1.2,.8) --({\r*cos((acos((-1+\g^2)/(1+\g^2))-acos((-1+(1/\L^2))/(1+(1/\L^2))))/3)},{\r*sin((acos((-1+\g^2)/(1+\g^2))-acos((-1+(1/\L^2))/(1+(1/\L^2))))/3)});

\def\s{0.22};
\draw (\s,{0}) arc (0:{-acos((-1+\g^2)/(1+\g^2))-acos((-1+(1/\L^2))/(1+(1/\L^2)))}:\s);

\draw [->] (-1,-.9)--({\s*cos((-acos((-1+\g^2)/(1+\g^2))-acos((-1+(1/\L^2))/(1+(1/\L^2))))/3)},{\s*sin((-acos((-1+\g^2)/(1+\g^2))-acos((-1+(1/\L^2))/(1+(1/\L^2))))/3)});

\draw(-1,-.8) node[below] {$-\varphi_A-\varphi_B$};

\draw [<->] (-1.2,0) -- (1.2,0);
\draw [<->] (0,-1.2) -- (0,1.2);

\draw (0.,-1.2)  node [below] {$\overline{S_A}\,\overline{S_B}$};

\end{tikzpicture}}
&
\!\!\!\!\!\!\!\!
$\subseteq$
\!\!\!\!\!\!\!\!
&
\raisebox{-.5\height}{
\begin{tikzpicture}[scale=1]
\def\g{1.7}
\def\L{.2};

\coordinate (A) at ({cos(acos((-1+\g^2)/(1+\g^2))-acos((-1+(1/\L^2))/(1+(1/\L^2))))},{sin(acos((-1+\g^2)/(1+\g^2))-acos((-1+(1/\L^2))/(1+(1/\L^2)))});
\coordinate (Ap) at ({cos(acos((-1+\g^2)/(1+\g^2))-acos((-1+(1/\L^2))/(1+(1/\L^2))))},{-sin(acos((-1+\g^2)/(1+\g^2))-acos((-1+(1/\L^2))/(1+(1/\L^2)))});

\coordinate (B) at ({cos(-acos((-1+\g^2)/(1+\g^2))-acos((-1+(1/\L^2))/(1+(1/\L^2))))},{sin(-acos((-1+\g^2)/(1+\g^2))-acos((-1+(1/\L^2))/(1+(1/\L^2)))});\coordinate (Bp) at ({cos(-acos((-1+\g^2)/(1+\g^2))-acos((-1+(1/\L^2))/(1+(1/\L^2))))},{-sin(-acos((-1+\g^2)/(1+\g^2))-acos((-1+(1/\L^2))/(1+(1/\L^2)))});

\coordinate (O) at (0,0);
%\fill [fill=lightgrey] (0,0) circle ({sqrt(1/(1+(1/\g)^2))});

%\filldraw ({2/(1+(1/\g))-1},0) circle ({0.6*1.5/1pt});
%\draw ({2/(1+(1/\g))-1},0)  node [above left] {$\frac{1-\alpha/\gamma}{1+\alpha/\gamma}$};
%\filldraw ({2/(1-(1/\g))-1},0)  circle ({0.6*1.5/1pt});
%\draw({2/(1-(1/\g))-1},0)  node [above right] {$\frac{1+\alpha/\gamma}{1-\alpha/\gamma}$};
%\filldraw ({2/(1-(1/\g)^2)-1},0)   circle ({0.6*1.5/1pt});
%\draw ({2/(1-(1/\g)^2)-1},0)   node [above ] {$\frac{1+\alpha^2/\gamma^2}{1-\alpha^2/\gamma^2}$};

%\draw (1,-0pt) node [below right] {$1$};
%\filldraw (1,0) circle ({0.6*1.5/1pt});

\begin{scope}
\clip  ({cos(acos((-1+\g^2)/(1+\g^2))-acos((-1+(1/\L^2))/(1+(1/\L^2))))},-1.2) rectangle (-1.2,1.2);
\fill[fill=medgrey](0,0) circle (1);
\end{scope}

\draw [<->] (-1.2,0) -- (1.2,0);
\draw [<->] (0,-1.2) -- (0,1.2);

\draw[dashed]  (O)--(A);

%\filldraw (A) circle ({0.6*1.5/1pt});

\def\r{0.3};
\draw (\r,{0}) arc (0:{acos((-1+\g^2)/(1+\g^2))-acos((-1+(1/\L^2))/(1+(1/\L^2)))}:\r);
\draw (1,0.2) node {$\varphi_A-\varphi_B$};
%\draw (\r,{0}) arc (0:{acos((-1+\g^2)/(1+\g^2))}:\r) node[pos=0.9, right]{$\varphi_A-\varphi_B$};
%\draw (\r,{0}) arc (0:{acos((-1+(1/\L^2))/(1+(1/\L^2)))}:\r) node[pos=0.9, right]{$\varphi_A-\varphi_B$};

\draw (0.,-1.2)  node [below] {\phantom{$\overline{S_B}$}};

\end{tikzpicture}}
\end{tabular}
\end{center}
Using the outer bounds 
$\overline{S_A}$ and $\overline{S_B}$ we establish correctness. Using  the inner bounds
$\underline{S_A}$ and $\underline{S_B}$ we establish tightness.
\begin{center}
\raisebox{-.5\height}{
\begin{tikzpicture}[scale=1.5]
\def\g{1.7}
\def\L{.2};
\def\t{1.2};

\fill[fill=lightgrey] (0,0) circle ({sqrt(1-\t*(2-\t)*(1-\g*\L)^2/(1+(\g)^2)/(1+(\L)^2))});

\coordinate (A) at ({(1-(\t/2))+(\t/2)*cos(acos((-1+\g^2)/(1+\g^2))-acos((-1+(1/\L^2))/(1+(1/\L^2))))},{(\t/2)*sin(acos((-1+\g^2)/(1+\g^2))-acos((-1+(1/\L^2))/(1+(1/\L^2)))});

\coordinate (B) at ({(1-(\t/2))+(\t/2)*cos(-acos((-1+\g^2)/(1+\g^2))-acos((-1+(1/\L^2))/(1+(1/\L^2))))},{(\t/2)*sin(-acos((-1+\g^2)/(1+\g^2))-acos((-1+(1/\L^2))/(1+(1/\L^2)))});

\coordinate (O) at ({1-(\t/2)},0);

\begin{scope}
\clip  ({(1-(\t/2))+(\t/2)*cos(acos((-1+\g^2)/(1+\g^2))-acos((-1+(1/\L^2))/(1+(1/\L^2))))},-1.2) rectangle (-1.2,1.2);
\fill[fill=medgrey]({1-(\t/2)},0) circle (\t/2);
\draw ({1-(\t/2)},0) circle (\t/2);
\end{scope}

\draw [<->] (-1.2,0) -- (1.2,0);
\draw [<->] (0,-1.2) -- (0,1.2);

\draw[dashed]  (O)--(A);

 \draw [decorate,decoration={brace,amplitude=4.5pt}] (O) -- (A);
 \draw (0.5,0.55) node[below] {$\theta$};

\filldraw (O) circle ({0.6*1.5/1.5pt});
\draw (O) node[below] {$1-\theta$};

\def\r{0.3};
\draw ({(1-(\t/2))+\r},{0}) arc (0:{acos((-1+\g^2)/(1+\g^2))-acos((-1+(1/\L^2))/(1+(1/\L^2)))}:\r);
\draw (1.2,0.15) node {$\varphi_A-\varphi_B$};
%\draw (\r,{0}) arc (0:{acos((-1+\g^2)/(1+\g^2))}:\r) node[pos=0.9, right]{$\varphi_A-\varphi_B$};
%\draw (\r,{0}) arc (0:{acos((-1+(1/\L^2))/(1+(1/\L^2)))}:\r) node[pos=0.9, right]{$\varphi_A-\varphi_B$};

\def\u{0.6};
\draw[->] (.8,-.9)--(\u,{-sqrt((\t/2)^2-(\u-1+(\t/2))^2)});
\draw (1.1,-.85) node [below] {$(1-\theta)+\theta \underline{S_A}\,\underline{S_B}$};

\def\u{0.6};
\draw[->] (1.2,-0.2)--({(1-(\t/2))+(\t/2)*cos(acos((-1+\g^2)/(1+\g^2))-acos((-1+(1/\L^2))/(1+(1/\L^2))))},-0.25);
\draw (1.2,-0.2) node [right] {$\subseteq (1-\theta)+\theta \overline{S_A}\,\overline{S_B}$};

\def\w{-.8};
\draw[->] (-1.2,-.6)--(\w,{-sqrt((1-\t*(2-\t)*(1-\g*\L)^2/(1+(\g)^2)/(1+(\L)^2))-(\w)^2)});
\draw (-1.15,-.6) node [left] {$\cG\left(\cL_R\right)$};
\draw (-1.45,-.9) node [below] {$R= \sqrt{1-\frac{4\theta(1-\theta)(1-\gamma L)^2}{(1+\gamma^2/\alpha^2)(1+\alpha^2 L^2)}}$};
\end{tikzpicture}}
\end{center}
With the Pythagorean theorem, we can verify that the containment holds for $R$ and fails for smaller $R$.
Since $\cL_R$ is SRG-full by Theorem~\ref{thm:srg-full}, the containment of the SRG in $\ecomplex$ equivalent to the containment of the class.

The result for the cases $\alpha/\gamma\ge 1$ or $\alpha L\ge 1$ and for the operator $D_{\alpha,\theta}(\cB,\cA)$ follows from similar reasoning.
\qed\end{proof}

\begin{corollary}
\label{cor:ms_impossibility2}
Let $0<1/\gamma\le L<\infty$ and $\alpha\in(0,\infty)$.
Let $B\in \cM\cap  \cL_L$ and let $A\in \cM$ satisfy a condition weaker than or equal to $\gamma$-inverse Lipschitz continuity, such as $\gamma$-metric subregularity.
It is not possible to establish a strict contraction of the DRS operators $T(A,B,\alpha,\theta)$ or $T(B,A,\alpha,\theta)$ for any $\alpha>0$ and $\theta\in \reals$ without further assumptions.
\end{corollary}

\section{Conclusion}
\label{s:conclusion}
In this work, we presented the scaled relative graph, a tool that maps the action of an operator to the extended complex plane.
This machinery enables us to analyze nonexpansive and monotone operators with geometric arguments.
% which are more visual and intuitive than classical analytic proofs based on inequalities.
The geometric ideas should complement the classical analytical approaches and bring clarity.

%It is interesting to compare the strengths and weaknesses of the analytic and geometric proof techniques. Analytic proofs are usually more concise since geometric proofs have illustrative diagrams. Also, analytical proofs tend to be straightforward to verify through direct, but sometimes long, calculations. In contrast, geometric proofs require a complete understanding to verify confidently. On the other hand, geometric proofs have the advantage that a single or a few geometric diagrams concisely capture and communicate the core insight. The ease of easily discerning tightness is another strength of geometric proofs. To summarize, geometric proofs tend to be easier to understand at a rough level but may require more work to carefully verify.

Extending this geometric framework to more general setups and spaces is an interesting future direction.
%The analysis of this work relies on the inner product and the induced norm of the Hilbert space. 
Some fixed-point iterations, such as the power iteration of non-symmetric matrices \cite{mises1929} or the Bellman iteration \cite{Bellman1952}, are analyzed most effectively through notions other than the norm induced by the inner product (the Euclidean norm for finite-dimensional spaces).
Whether it is possible to gain insight through geometric arguments in such setups would be worthwhile to investigate.

\section*{Acknowledgments}
We thank Pontus Giselsson for the illuminating discussion on metric subregularity.
We thank Minyong Lee and Yeoil Yoon for helpful discussions on inversive geometry and its use in high school mathematics competitions.
We thank Xinmeng Huang for the aid in drawing Figure~\ref{fig:matrix_SRG}.
We thank the Erwin Schr\"odinger Institute and the organizers of its workshop in February 2019, which provided us with fruitful discussions and helpful feedback that materially improved this paper.

\bibliographystyle{spmpsci} 
\bibliography{srg,Master_Bibliography}

\appendix

% Take as an example, consider one of the case.
% We take a point outside of the circle.
% Then we reverse the process in the proof.
% We get the complex number by the lemma.
% Finally, convert this complex number to a $2\times 2$ matrix.

% \[
% \lambda_\pm=
% \frac{1}{1+2\alpha\mu+\alpha^2\mu/\beta}
% \left(
% 1-\alpha^2\mu/\beta\pm 2
% \sqrt{
% \alpha^2\mu/\beta(1-\mu\beta)i
% }
% \right)
% \]
% with 
% \[
% |\lambda_\pm|^2=1-
% \frac{4\alpha\mu}{1+2\alpha\mu+\alpha^2\mu/\beta}
% \]
% We now revert the mappings
% $\frac{1}{2}z+\frac{1}{2}$ and $\bar{z}^{-1}$ to get 
% \[
% z=
% \frac{4(1+\alpha\mu)}{3-2\alpha\mu-\alpha^2\mu/\beta}
% +
% \frac{
% 2(1+2\alpha\mu+\alpha^2\mu/\beta-2\alpha\sqrt{(1-\mu\beta)\mu/\beta})
% }{3-2\alpha\mu-\alpha^2\mu/\beta}
% i
% \]
% Now we apply Lemma~\ref{lem:complex-srg}
% \[
% \frac{1}{3-2\alpha\mu-\alpha^2\mu/\beta}
% \begin{bmatrix}
% 4(1+\alpha\mu)&-2(1+2\alpha\mu+\alpha^2\mu/\beta-2\alpha\sqrt{(1-\mu\beta)\mu/\beta})\\
% 2(1+2\alpha\mu+\alpha^2\mu/\beta-2\alpha\sqrt{(1-\mu\beta)\mu/\beta})&4(1+\alpha\mu)
% \end{bmatrix}
% \]

\section{Further discussion}
%\subsection{Further introductory remarks}
% XXX

% XXX

% \textbf{Geometry.}
% We use without reference results from elementary geometry such as the fact that $2$ non-identical circles intersect at at most $2$ points.
% We do, however, explicitly state the following two results as they are not as well known.
%However, we explicitly describe Stewart's theorem and the spherical triangle inequality as they are not as well known.

\subsection{The role of maximality}
A fixed-point iteration $x^{k+1}=Tx^k$ becomes undefined if its iterates ever escape the domain of $T$.
This is why we assume the monotone operators are maximal as maximality ensures $\dom{ J_{\alpha A}}=\hilbert$.
The results of this work are otherwise entirely independent of the notion of maximality.

In Section~\ref{s:prelim}, we define $\cM$ to contain all monotone operators (maximal or not).
This choice is necessary to make $\cM$ SRG-full.
For the other classes $\cL_{L}$, $\cC_{\beta }$, $\cM_{\mu}$, and $\cN_\theta$, we make no restriction on the domain or maximality so that they can be SRG-full.
% one could also define the notion of maximality.
% The same is true with other operators as well.
% $\cL_{L}$, $\cC_{\beta }$, $\cM_{\mu}$, $\cN_\theta$ are SRG-full classes only because we allow non-maximal operators within the set.

% iN 
% This choice is necessary for the notion of SRG-fullness discussed in Section~\ref{ss:srg-full}.
% %As an aside, such operators can be extended to be maximal or to have full domain.
% In particular, $\cM$ would be an SRG-full class only when we allow non-maximal monotone operators within the set.

\subsection{Minkowski-type set notation}
Given $\alpha\in \reals$ and sets $U,V\subseteq\hilbert$, write
\[
\alpha U=\{\alpha u\,|\,u\in U\},\quad
U+V=\{u+v\,|\,u\in U,\,v\in V\},\quad U-V=U+(-V).
\]
Notice that if either $U$ or $V$ is $\emptyset$, then $U+V=\emptyset$.

Given $Z,W\subseteq \complex$, write
\begin{gather*}
  Z+W=\{z+w\,|\,z\in Z,\,w\in W\},\quad ZW=\{zw\,|\,z\in Z,\,w\in W\}.
\end{gather*}
Given $\alpha\in \complex$, $\alpha\ne 0$, and $Z\subseteq\ecomplex$, write 
\[
\alpha Z = \{ \alpha z\,|\,z\in Z\}.
\]

Given a class of operators $\cA$ and $\alpha\ne 0$, write
\begin{gather*}
 \cA^{-1}=\{ A^{-1}\,|\,A\in \cA\},\qquad
\alpha \cA=\{\alpha A\,|\,A\in \cA\},\qquad
 \cA\alpha =\{A\alpha \,|\,A\in \cA\}.
\end{gather*}
Given classes of operators $\cA$ and $\cB$ and $\alpha>0$, write
\begin{align*}
\cA+\cB&=\{A+B\,|\,A\in \cA,\,B\in \cB,\,A\colon\hilbert\rightrightarrows\hilbert,\,B\colon\hilbert\rightrightarrows\hilbert\}\\
\cA\cB&=\{AB\,|\,A\in \cA,\,B\in \cB,\,A\colon\hilbert\rightrightarrows\hilbert,\,B\colon\hilbert\rightrightarrows\hilbert\}\\
% I+\alpha \cA&=\{I+\alpha A\,|\,A\in \cA,\,A\colon\hilbert\rightrightarrows\hilbert,\,I\colon\hilbert\rightarrow\hilbert\}\\
J_{\alpha \cA}&=\{J_{\alpha A}\,|\,A\in \cA,\,A\colon\hilbert\rightrightarrows\hilbert\}.
\end{align*}

%The inversion of a generalized circle can be determined by the following steps:

% To describe how this happens, we define some notation for defining generalized circular regions. For $v\in\CC$, $x\leq y\in\RR$, let $C(v,x,y)$ denote the generalized circular region between $x$ and $y$ in the direction of $v$. When $x> y$, we define $C(v,x,y)$ to be the complement of $C(v,y,x)$. So for instance, $C(e_1,a,b)$ below denotes the region circle between $a$ and $b$ in the direction $e_1$. $C(e_1,b^{-1},a^{-1})$ denotes the \emph{outside} of the circle between $a^{-1}$ and $b^{-1}$ (the outside since $b^{-1}>a^{-1}$). Similarly, $C(e_1,f,\infty)$ is the half-space to the right of $f$, which can be thought of as a circle between $f$ and $+\infty$.

% Using this notation, it is easy to describe the action of $\cI$. We have
% \begin{align*}
% \cI\p{C(v,x,y)}=C(v,y^{-1},x^{-1}).
% \end{align*}
% The top row of the above figure corresponds to some example generalized circular regions. Directly below each generalized circular region in the top row, is its image under $\cI$ in the bottom row. Regions in the bottom row can be obtained from the above formula (recalling that we define $0/(\pm\infty)=0$). We will use this transformation frequently, so it is important that the reader understands how to take the inversion of a generalized circle as shown above.

\subsection{SRG-full classes}
\label{ss:srg-full-appendix}

There is one degenerate case to keep in mind for the sake of rigor.
%The class of all operators $\cA_\text{full}$ is represented by $h(a,b,c)=0$ and has $\cG(\cA_\text{full})=\ecomplex$.
The SRG-full class of operators $\cA_\text{null}$  represented by 
$h(a,b,c)=a+b+|c|$
has $\cG(\cA_\text{null})=\emptyset$.
However, the class $\cA_\text{null}$ is not itself empty; it contains operators whose graph contains zero or one pair, i.e.,
$A\in \cA_\text{null}$ if and only if we have either a) $\dom{A}=\emptyset$ or b) $\dom{A}=x$ and $Ax=\{y\}$ for some $x,y\in \hilbert$.

Theorem~\ref{thm:srg-intersection} does not apply when the operator classes are not SRG-full.
For example, although
\[
\partial \cF_{\mu,L}=\partial \cF_{\mu,\infty}\cap \partial \cF_{0,L}
\]
we have the strict containment
\begin{center}
\begin{tabular}{lll}
\raisebox{-.5\height}{
\begin{tikzpicture}[scale=1.5]
\fill[fill=medgrey] (0.65,0) circle (0.35);
\draw [<->] (-0.5,0)-- (1.2,0);
\draw [<->] (0,-.7) -- (0,.7);
\filldraw (1,0) circle ({0.6*1.5/1.5pt});
\draw (1,0pt) node [above right] {$L$};
\filldraw (.3,0) circle ({0.6*1.5/1.5pt});
\draw (0.3,0pt) node [above left] {$\mu$};
\draw (0.6,-.65) node {$\cG(\partial \mathcal{F}_{\mu,L})$};
\end{tikzpicture}
}
&$\subset$ &
\raisebox{-.5\height}{
\begin{tikzpicture}[scale=1.5]
\begin{scope}
\clip (0.3,-.7) rectangle (1.2,.7);
\fill[fill=medgrey] (0.5,0) circle (0.5);
\end{scope}
\begin{scope}
\clip (0.3,-.7) rectangle (-0.5,.7);
\draw[dashed] (0.5,0) circle (0.5);
\end{scope}
\draw[dashed] (0.3,0.458258) -- (0.3,0.7);
\draw[dashed] (0.3,-0.458258) -- (0.3,-0.7);
\draw [<->] (-0.5,0) -- (1.2,0);
\draw [<->] (0,-.7) -- (0,.7);

\draw (0.3,0pt) node [above left] {$\mu$};
\filldraw (.3,0) circle ({0.6*1.5/1.5pt});
\draw (1.4,-.65) node{$\cG(\partial \mathcal{F}_{\mu,\infty})\cap \cG(\partial \mathcal{F}_{0,L})$};
\draw (1,0) node [above right] {$L$};
\filldraw (1,0) circle ({0.6*1.5/1.5pt});
\end{tikzpicture}}
\end{tabular}
\end{center}

\section{Invariant circle number}
\label{s:invariant-circle}
Let $\cA$ be an SRG-full class such that $\cG(\cA)\ne\emptyset$ and $\cG(\cA)\ne\ecomplex$.
Define the \emph{circle number} of $\cA$ as
\begin{align*}
\inf\Big\{k&\in \mathbb{N}\,\Big|\,\cG(\cA)=B_1\cap\dots\cap B_k,\,\text{$B_i$ is a disk or a half-space for $i=1,\dots,k$}
\Big\},
\end{align*}
which is a positive integer or $\infty$.
For example, $\cM\cap \cL_1$ has circle number $2$ since
\begin{center}
$\cG(\cM\cap \cL_1)=\left\{z\,|\,|z|\le 1\right\}\cap \left\{z\,|\,\Re z\ge 0\right\}=$
\raisebox{-.5\height}{
\begin{tikzpicture}[scale=.7]
\begin{scope}
\clip (-1.2,-1.2) rectangle (0,1.2);
\draw[dashed] (0,0) circle (1);
\end{scope}
\begin{scope}
\clip (0,-1.2) rectangle (1.2,1.2);
\fill[fill=medgrey] (0,0) circle (1);
\end{scope}
\draw[dashed] (0,1)--(0,1.4);
\draw[dashed] (0,-1)--(0,-1.4);
%\draw [<->] (-1.2,0) -- (1.2,0);
%\draw [<->] (0,-1.2) -- (0,1.2);
\filldraw (1,0) circle ({0.6*1.5/.7pt});
\draw (1,0) node [ right] {$1$};

\filldraw (0,0) circle ({0.6*1.5/.7pt});
\draw (0,0) node [left] {$0$};
\end{tikzpicture}}
\end{center}
In this section, we show that the circle number of an operator class is invariant under certain operations.
This is analogous to how the genus or the winding number are topological invariants under homeomorphisms.
That it is impossible to continuously deform a donut into a sphere since they have different numbers of holes, an invariant, is a standard argument of topology.
The circle number serves as an analogous invariant for operator classes.

\begin{theorem}
%\emph{Circle number.}
\label{thm:circle-invariant}
The circle number of an SRG-full operator class is invariant under non-zero pre and post-scalar multiplication, addition by identity, and inversion.
\end{theorem}
\begin{proof}
Let $T$ be a one-to-one mapping from an operator to an operator
and let $T'$ the the corresponding one-to-one mapping from $\ecomplex$ to $\ecomplex$.
In particular, consider the following four cases:
first $T(A)=\alpha A$ and $T'(z)=\alpha z$,
second $T(A)= A\alpha $ and $T'(z)=\alpha z$,
third $T(A)= I+A $ and $T'(z)=1+z$,
and fourth $T(A)= A^{-1} $ and $T'(z)=\bar{z}^{-1}$.

If 
\[
\cG(\cA)=B_1\cap \dots\cap B_k,
\]
then
\[
\cG(T(\cA))=T'(B_1)\cap \dots\cap T'(B_k),
\]
where $T'(B_1),\dots,T'(B_k)$ are each a disk or a half-space.
Therefore, the circle number of $T(\cA)$ satisfies
\begin{align*}
\inf\Big\{k\,\Big|\,\cG(T(\cA))=T'(B_1)\cap\dots&\cap T'(B_k)
\Big\}\le
\inf\Big\{k\,\Big|\,\cG(\cA)=B_1\cap\dots\cap B_k
\Big\}.
\end{align*}
Since $T$ and $T'$ are invertible mappings, the argument goes in the other direction as well, and we conclude that the infimums are equal.
\qed\end{proof}
\begin{corollary}
\label{cor:impossible-mapping-mon-lip}
There is no one-to-one mapping from $\mathcal{M}$ to $\mathcal{M}\cap \mathcal{L}_L$
constructed via pre and post-scalar multiplication, addition with the identity operator, and operator inversion.
\end{corollary}

Such one-to-one mappings between operator classes are used for translating a nice result on a simple operator class to another operator class.
%We did this in Section~\ref{ss:srg-classes-proofs}.
In \cite{kvextension2010,kvextension2_2010,OSPEP} the maximal monotone extension theorem was translated to extension theorems of other operator classes.
Corollary~\ref{cor:impossible-mapping-mon-lip} shows that this approach will not work for $\mathcal{M}\cap \mathcal{L}_L$, a class of operators considered by the extragradient method \cite{korpelevich1976}, forward-backward-forward splitting \cite{tseng2000}, and other related methods \cite{davis2018,malitsky2018}.
In fact, \cite{OSPEP} shows that certain simple interpolation condition for $\cM$ fails for $\mathcal{M}\cap \mathcal{L}_L$.

\section{Deferred proofs}

\begin{figure}
\begin{center}
\includegraphics[width=0.4\textwidth]{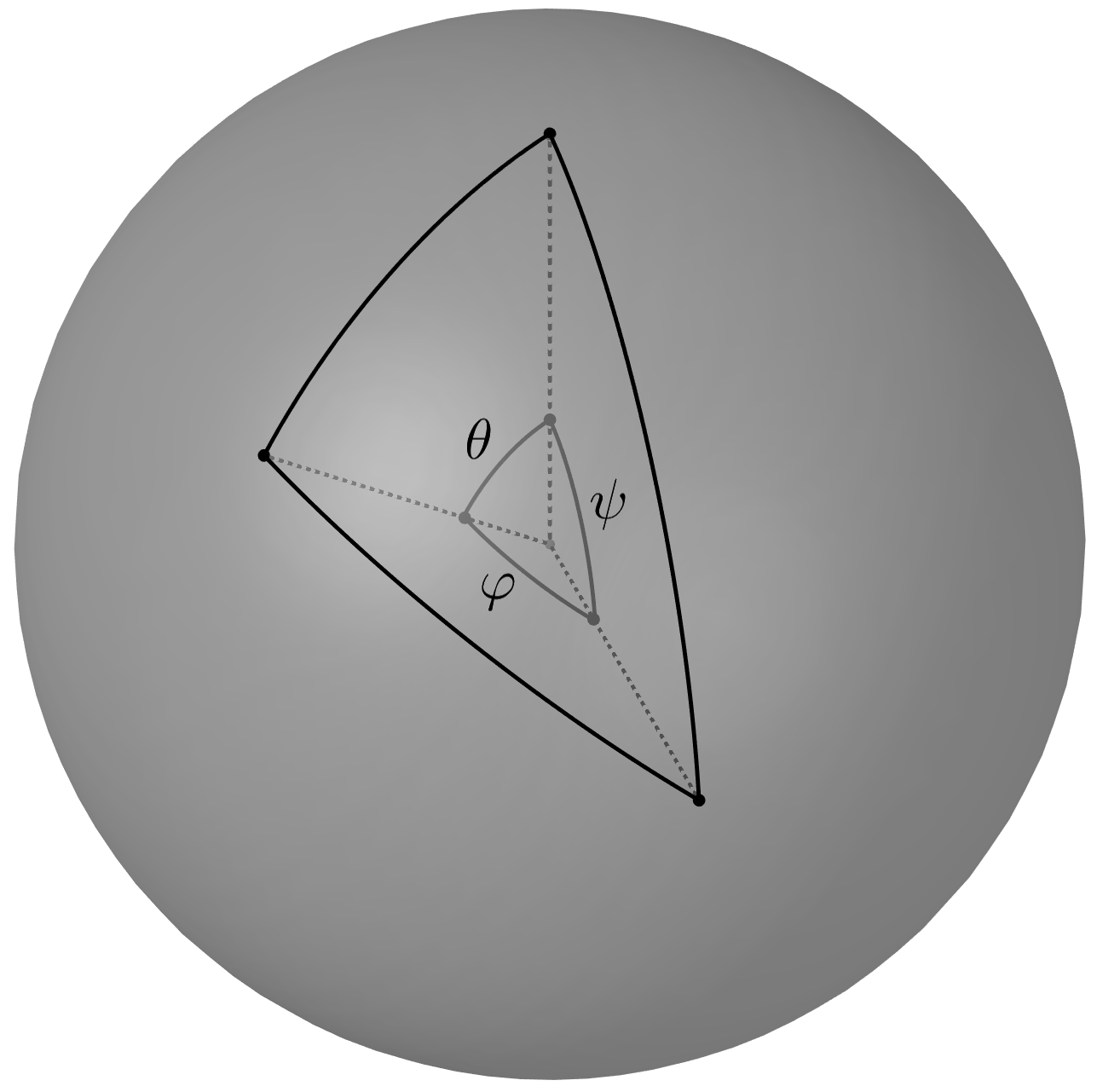}
\end{center}
\caption{Spherical triangle inequality:
$|\theta-\varphi|\le \psi\le \theta+\varphi$.}
\label{fig:spherical}
\end{figure}

\begin{fact}[Spherical triangle inequality]
Any nonzero $a,b,c\in \hilbert$ satisfies
\[\left|\angle (a,b)-\angle (b,c)\right|
\le \angle (a,c)\le 
%\max\left\{ \angle (a,b)+\angle (b,c),\pi/2\right\}
 \angle (a,b)+\angle (b,c).
\]
\end{fact}
Figure~\ref{fig:spherical} illustrates the inequality.
We use the spherical triangle inequality in Theorem~\ref{thm:composition-srg} to argue that there is no need to consider a third dimension and that we can continue the analysis in 2D.
%All other geometric arguments on this paper rely on 2D geometry.

\begin{proof}[Proof of spherical triangle inequality]
Although this result is known, we provide a proof for completeness.
Without loss of generality, assume $a$, $b$, and $c$ are unit vectors.
%since scaling the inputs by positive scalars does not change the output of $\angle$,
Let $\theta=\angle (a,b)$ and $\varphi=\angle (b,c)$, and
without loss of generality, assume $\theta\ge \varphi$.
Then we have
\[
a=b\cos\theta+u\sin\theta,\quad
c=b\cos \varphi+v\sin\varphi,
\]
where $u$ and $v$ are unit vectors orthogonal to $b$,
and
\[
\langle a,c\rangle
=\cos\theta\cos\varphi+\langle u,v\rangle \sin\theta\sin\varphi.
\]
Since $\theta,\varphi\in [0, \pi]$, we have $\sin\theta\sin\varphi\ge 0$.
Since $\|u\|=\|v\|=1$,  we have $|\langle u,v\rangle| \le 1$.
Therefore
\[
\cos\theta\cos\varphi-\sin\theta\sin\varphi
\le
\langle a,c\rangle
\le 
\cos\theta\cos\varphi+\sin\theta\sin\varphi.
\]
Since $\cos(\alpha\pm\beta)=\cos\alpha\cos\beta\mp\sin\alpha\sin\beta$, 
\cite[p. 72, 4.3.17]{abramowitz_stegun}
we have
\[
\cos(\theta+\varphi)\le\langle a,c\rangle\le \cos(\theta-\varphi),
\]
and we conclude 
\[
\angle (a,c)=\arccos(\langle a,c\rangle)\in [\theta-\varphi,\theta+\varphi].
\]
\qed\end{proof}
\label{s:geo-proofs}
\begin{proof}[Proof of Fact \ref{prop:sm-coco-forward}]
First consider the case  $\mu<1/(2\beta)$.
By Proposition~\ref{prop:monotone-srg} and Theorem~\ref{thm:srg-scaling-translation}, we have the geometry 
\begin{center}
\begin{tabular}{cc}
\raisebox{-.5\height}{
\begin{tikzpicture}[scale=1.5]
\def\m{.5};
\def\b{0.6};
\fill [fill=lightgrey] (0,0) circle ({sqrt(1-2*\m+\m/\b)});
\begin{scope}
\clip ({1-\m},-1.1) rectangle (-1.1,1.1);
\fill[fill=medgrey] ({1-1/2/\b},0) circle ({1/2/\b});
\end{scope}
%\draw [dashed] (0,0) circle ({sqrt(1-2*\m+\L^2)});
\draw [<->] (-1.2,0) -- (1.2,0);
\draw [<->] (0,-1.2) -- (0,1.2);
%\draw [dashed] ({1-\m},{sqrt(\L^2-\m^2)}) -- (0,{sqrt(\L^2-\m^2)});

%\filldraw ({1-\m},{sqrt(1/4/\b^2-(\m-1/2/\b)^2)}) circle ({0.6*1.5/1.5pt});

%Draw braced note
%\draw [decorate,decoration={brace,amplitude=4.5pt}] ({sqrt(1-2*\m+\m/\b)*cos(240)},{sqrt(1-2*\m+\m/\b)*sin(240)}) -- (0,0);
%\draw [line width=1.0pt]({sqrt(1-2*\m+\m/\b)*cos(240)},{sqrt(1-2*\m+\m/\b)*sin(240)}) -- (0,0);
%\draw[->] (-1.1,-.77)--({sqrt(1-2*\m+\m/\b)*cos(240)/2-0.07},{sqrt(1-2*\m+\m/\b)*sin(240)/2+.03});
%\draw (-1.1,-.77) node[below] {$\sqrt{1-2\alpha\mu+\alpha^2\mu/\beta}$};

\coordinate (I1) at  ({1-1/\b/2},0) ;
\coordinate (I2) at ({1-\m},0);
\coordinate (I3) at ({1-\m},{sqrt(1/4/\b^2-(\m-1/2/\b)^2)});
%\draw [] (I1)-- (I3);

\filldraw (I2) circle ({0.6*1.5/1.5pt});

%\draw  ({1-1/\b/2},0.05) node [above left, fill=medgrey] {$1-1/(2\beta)$};
\filldraw ({1-\m},0) node [below ] {$1-\alpha \mu$};

\filldraw ({1-1/\b},0) circle ({0.6*1.5/1.5pt}) ;
\begin{scope}
\clip ({1-\m},-1.1) rectangle (-1.1,1.1);
\clip ({1-1/2/\b},0) circle ({1/2/\b});
\draw ({1-1/\b},0.03) node [above right, fill=medgrey] {$1-\alpha/\beta$};
\end{scope}
\filldraw (1,0) circle ({0.6*1.5/1.5pt})  node [above right] {$1$};
%\filldraw (0,{sqrt(\L^2-\m^2)}) circle ({0.6*1.5/1.5pt}) node [below right] {$\alpha\sqrt{L^2-\mu^2}$} ;
%\filldraw (0,{sqrt(1-2*\m+\L^2)}) circle ({0.6*1.5/1.5pt}) node [above right] {$\sqrt{1-2\alpha\mu+\alpha^2 L^2}$};

\coordinate (O1) at ({1-1/\b},0);
\coordinate (O2) at ({1-1/\b/2},0);
\coordinate (O3) at ({1-\m},{sqrt(1/4/\b^2-(\m-1/2/\b)^2)});

\draw[->] (1.1,.8)--({1-\m,0.4});
\draw(1.1,.8) node[above] {$\cG\left(I-\alpha\cA \right)$};

\def\x{-.8};
\draw[->] (-.9,-.8)--(\x,{-sqrt(1-2*\m+\m/\b-(\x)^2)});
\draw(-.9,-.8) node[below ] {$\cG\left(\cL_R\right)$};
\draw(1.2,-1.05) node[] {$R=\sqrt{1-2\alpha\mu+\alpha^2\mu/\beta}$};
\end{tikzpicture}
}
&
\!\!\!\!\!\!\!\!\!\!\!\!\!\!\!\!\!\!\!\!\!\!\!\!\!\!\!
\raisebox{-.5\height}{
\begin{tikzpicture}[scale=1.5]
\def\m{.5};
\def\b{0.6};
\fill [fill=lightgrey] (0,0) circle ({sqrt(1-2*\m+\m/\b)});

\begin{scope}
\clip ({1-\m},-1.1) rectangle (-1.1,1.1);
\fill[fill=medgrey] ({1-1/2/\b},0) circle ({1/2/\b});
\end{scope}

\begin{scope}
\clip ({1-\m},-1.1) rectangle (1.1,1.1);
\draw[dashed] ({1-1/2/\b},0) circle ({1/2/\b});
\end{scope}

%\draw [dashed] (0,0) circle ({sqrt(1-2*\m+\L^2)});
\draw [<->] (-1.2,0) -- (1.2,0);
\draw [<->] (0,-1.2) -- (0,1.2);
%\draw [dashed] ({1-\m},{sqrt(\L^2-\m^2)}) -- (0,{sqrt(\L^2-\m^2)});

%\filldraw ({1-\m},{sqrt(1/4/\b^2-(\m-1/2/\b)^2)}) circle ({0.6*1.5/1.5pt});

%\draw [dashed] (0,0) circle ({sqrt(1-2*\m+\m/\b)});

%\draw [decorate,decoration={brace,amplitude=4.5pt}] ({sqrt(1-2*\m+\m/\b)*cos(240)},{sqrt(1-2*\m+\m/\b)*sin(240)}) -- (0,0);
%\draw [line width=1.0pt]({sqrt(1-2*\m+\m/\b)*cos(240)},{sqrt(1-2*\m+\m/\b)*sin(240)}) -- (0,0);
%\draw[->] (-1.1,-.77)--({sqrt(1-2*\m+\m/\b)*cos(240)/2-0.07},{sqrt(1-2*\m+\m/\b)*sin(240)/2+.03});
%\draw (-1.1,-.77) node[below] {$\sqrt{1-2\alpha\mu+\alpha^2\mu/\beta}$};

\coordinate (I1) at  ({1-1/\b/2},0) ;
\coordinate (I2) at ({1-\m},0);
\coordinate (I3) at ({1-\m},{sqrt(1/4/\b^2-(\m-1/2/\b)^2)});
\coordinate (I4) at ({1-\m},{-sqrt(1/4/\b^2-(\m-1/2/\b)^2)});
%\draw [] (I1)-- (I3);

\filldraw (I1) circle ({0.6*1.5/1.5pt});
\filldraw (I2) circle ({0.6*1.5/1.5pt});
\filldraw (I3) circle ({0.6*1.5/1.5pt});
\filldraw (I4) circle ({0.6*1.5/1.5pt});
\draw (I3) -- (I1);
\tkzMarkRightAngle[size=.06666](I1,I2,I3);

\draw (I1) node[below] {$C$};
\draw (I3) node[above] {$B$};
\draw (I4) node[below] {$B'$};
\draw (0,0) -- (I3);

%\draw  ({1-1/\b/2},0.05) node [above left, fill=medgrey] {$1-1/(2\beta)$};
\filldraw ({1-\m},0) node [below right ] {$D$};

\filldraw ({1-1/\b},0) circle ({0.6*1.5/1.5pt}) ;
\draw ({1-1/\b},0.03) node [above left,] {$A$};
\filldraw (1,0) circle ({0.6*1.5/1.5pt})  node [above right] {$1$};
%\filldraw (0,{sqrt(\L^2-\m^2)}) circle ({0.6*1.5/1.5pt}) node [below right] {$\alpha\sqrt{L^2-\mu^2}$} ;
%\filldraw (0,{sqrt(1-2*\m+\L^2)}) circle ({0.6*1.5/1.5pt}) node [above right] {$\sqrt{1-2\alpha\mu+\alpha^2 L^2}$};

\filldraw (0,0) circle ({0.6*1.5/1.5pt}) ;

\draw (0,0) node[above left] {$O$};
%\filldraw (I3) -- (0,0);

\coordinate (O1) at ({1-1/\b},0);
\coordinate (O2) at ({1-1/\b/2},0);
\coordinate (O3) at ({1-\m},{sqrt(1/4/\b^2-(\m-1/2/\b)^2)});

% \tkzMarkSegment[pos=.5,mark=||](O1,O2) ;
% \tkzMarkSegment[pos=.5,mark=||](O2,O3) ;
\end{tikzpicture}
}
\end{tabular}
\end{center}
To clarify, $O$ is the center of the circle with radius $\overline{OB}$ (lighter shade)
and $C$ is the center of the circle with radius $\overline{AC}=\overline{CB}$ defining the inner region (darker shade).
%The inner region (darker shade) is contained within the disk with radius  $\sqrt{1-2\alpha\mu+\alpha^2\mu/\beta}$ centered at the origin (lighter shade) by the following geometric reasoning.
With two applications of the Pythagorean theorem, we get
\begin{align*}
\overline{OB}^2
&=\overline{OD}^2+\overline{DB}^2
=\overline{OD}^2+\overline{BC}^2-\overline{CD}^2\\
&=(1-\alpha \mu)^2+(\alpha/(2\beta))^2-(\alpha/(2\beta)-\alpha \mu)^2
=1-2\alpha\mu+\alpha^2\mu/\beta.
\end{align*}
%Since $\overline{BD}^2=\overline{BC}^2-\overline{CD}^2$
Since $\overline{B'B}$ is a chord of circle $O$, it is within the circle.
Since $2$ non-identical circles intersect at at most $2$ points, and since $A$ is within circle $O$, arc $\arc{BAB'}$ is within circle $O$.
Finally, the region bounded by $\overline{B'B}\cup\arc{BAB'}$ (darker shade) is within circle $O$ (lighter shade).

In the cases $\mu=1/(2\beta)$ and $\mu>1/(2\beta)$, we have a slightly different geometry, but the same arguments and calculations hold.
\begin{center}
\begin{tabular}{cc}
\raisebox{-.5\height}{
\begin{tikzpicture}[scale=1.5]
\def\m{0.8333333};
\def\b{0.6};
\fill [fill=lightgrey] (0,0) circle ({sqrt(1-2*\m+\m/\b)});

\begin{scope}
\clip ({1-\m},-1.1) rectangle (-1.1,1.1);
\fill[fill=medgrey] ({1-1/2/\b},0) circle ({1/2/\b});
\end{scope}

\begin{scope}
\clip ({1-\m},-1.1) rectangle (1.1,1.1);
\draw[dashed] ({1-1/2/\b},0) circle ({1/2/\b});
\end{scope}

%\draw [dashed] (0,0) circle ({sqrt(1-2*\m+\L^2)});
\draw [<->] (-1.2,0) -- (1.2,0);
\draw [<->] (0,-1.2) -- (0,1.2);
%\draw [dashed] ({1-\m},{sqrt(\L^2-\m^2)}) -- (0,{sqrt(\L^2-\m^2)});

%\filldraw ({1-\m},{sqrt(1/4/\b^2-(\m-1/2/\b)^2)}) circle ({0.6*1.5/1.5pt});

%\draw [dashed] (0,0) circle ({sqrt(1-2*\m+\m/\b)});

%\draw [decorate,decoration={brace,amplitude=4.5pt}] ({sqrt(1-2*\m+\m/\b)*cos(240)},{sqrt(1-2*\m+\m/\b)*sin(240)}) -- (0,0);
%\draw [line width=1.0pt]({sqrt(1-2*\m+\m/\b)*cos(240)},{sqrt(1-2*\m+\m/\b)*sin(240)}) -- (0,0);
%\draw[->] (-1.1,-.77)--({sqrt(1-2*\m+\m/\b)*cos(240)/2-0.07},{sqrt(1-2*\m+\m/\b)*sin(240)/2+.03});
%\draw (-1.1,-.77) node[below] {$\sqrt{1-2\alpha\mu+\alpha^2\mu/\beta}$};

\coordinate (I1) at  ({1-1/\b/2},0) ;
\coordinate (I2) at ({1-\m},0);
\coordinate (I3) at ({1-\m},{sqrt(1/4/\b^2-(\m-1/2/\b)^2)});
\coordinate (I4) at ({1-\m},{-sqrt(1/4/\b^2-(\m-1/2/\b)^2)});
%\draw [] (I1)-- (I3);

\filldraw (I1) circle ({0.6*1.5/1.5pt});
\filldraw (I3) circle ({0.6*1.5/1.5pt});
\filldraw (I4) circle ({0.6*1.5/1.5pt});
\draw (I3) -- (I1);
\tkzMarkRightAngle[size=.066666666666](I1,I2,I3);

\draw (I1) node[below right] {$C=D$};
\draw (I3) node[above] {$B$};
\draw (I4) node[below] {$B'$};
\draw (0,0) -- (I3);

%\draw  ({1-1/\b/2},0.05) node [above left, fill=medgrey] {$1-1/(2\beta)$};

\filldraw ({1-1/\b},0) circle ({0.6*1.5/1.5pt}) ;
\draw ({1-1/\b},0.03) node [above left,] {$A$};
\filldraw (1,0) circle ({0.6*1.5/1.5pt})  node [above right] {$1$};
%\filldraw (0,{sqrt(\L^2-\m^2)}) circle ({0.6*1.5/1.5pt}) node [below right] {$\alpha\sqrt{L^2-\mu^2}$} ;
%\filldraw (0,{sqrt(1-2*\m+\L^2)}) circle ({0.6*1.5/1.5pt}) node [above right] {$\sqrt{1-2\alpha\mu+\alpha^2 L^2}$};

\filldraw (0,0) circle ({0.6*1.5/1.5pt}) ;

\draw (0,0) node[below left] {$O$};
%\filldraw (I3) -- (0,0);

\coordinate (O1) at ({1-1/\b},0);
\coordinate (O2) at ({1-1/\b/2},0);
\coordinate (O3) at ({1-\m},{sqrt(1/4/\b^2-(\m-1/2/\b)^2)});

\draw (-1.2,-1) node {Case $\mu=1/(2\beta)$};

% \tkzMarkSegment[pos=.5,mark=||](O1,O2) ;
% \tkzMarkSegment[pos=.5,mark=||](O2,O3) ;
\end{tikzpicture}
}
&
\raisebox{-.5\height}{
\begin{tikzpicture}[scale=1.5]
\def\m{1.2};
\def\b{0.6};
\fill [fill=lightgrey] (0,0) circle ({sqrt(1-2*\m+\m/\b)});

\begin{scope}
\clip ({1-\m},-1.1) rectangle (-1.1,1.1);
\fill[fill=medgrey] ({1-1/2/\b},0) circle ({1/2/\b});
\end{scope}

\begin{scope}
\clip ({1-\m},-1.1) rectangle (1.1,1.1);
\draw[dashed] ({1-1/2/\b},0) circle ({1/2/\b});
\end{scope}

%\draw [dashed] (0,0) circle ({sqrt(1-2*\m+\L^2)});
\draw [<->] (-1.2,0) -- (1.2,0);
\draw [<->] (0,-1.2) -- (0,1.2);
%\draw [dashed] ({1-\m},{sqrt(\L^2-\m^2)}) -- (0,{sqrt(\L^2-\m^2)});

%\filldraw ({1-\m},{sqrt(1/4/\b^2-(\m-1/2/\b)^2)}) circle ({0.6*1.5/1.5pt});

%\draw [dashed] (0,0) circle ({sqrt(1-2*\m+\m/\b)});

%\draw [decorate,decoration={brace,amplitude=4.5pt}] ({sqrt(1-2*\m+\m/\b)*cos(240)},{sqrt(1-2*\m+\m/\b)*sin(240)}) -- (0,0);
%\draw [line width=1.0pt]({sqrt(1-2*\m+\m/\b)*cos(240)},{sqrt(1-2*\m+\m/\b)*sin(240)}) -- (0,0);
%\draw[->] (-1.1,-.77)--({sqrt(1-2*\m+\m/\b)*cos(240)/2-0.07},{sqrt(1-2*\m+\m/\b)*sin(240)/2+.03});
%\draw (-1.1,-.77) node[below] {$\sqrt{1-2\alpha\mu+\alpha^2\mu/\beta}$};

\coordinate (I1) at  ({1-1/\b/2},0) ;
\coordinate (I2) at ({1-\m},0);
\coordinate (I3) at ({1-\m},{sqrt(1/4/\b^2-(\m-1/2/\b)^2)});
\coordinate (I4) at ({1-\m},{-sqrt(1/4/\b^2-(\m-1/2/\b)^2)});
%\draw [] (I1)-- (I3);

\filldraw (I1) circle ({0.6*1.5/1.5pt});
\filldraw (I2) circle ({0.6*1.5/1.5pt});
\filldraw (I3) circle ({0.6*1.5/1.5pt});
\filldraw (I4) circle ({0.6*1.5/1.5pt});
\draw (I3) -- (I1);
\tkzMarkRightAngle[size=.06666666666](I1,I2,I3);

\draw (I1) node[below] {$C$};
\draw (I3) node[above] {$B$};
\draw (I4) node[below] {$B'$};
\draw (0,0) -- (I3);

%\draw  ({1-1/\b/2},0.05) node [above left, fill=medgrey] {$1-1/(2\beta)$};
\filldraw ({1-\m},0) node [below left ] {$D$};

\filldraw ({1-1/\b},0) circle ({0.6*1.5/1.5pt}) ;
\draw ({1-1/\b},0.03) node [above left,] {$A$};
\filldraw (1,0) circle ({0.6*1.5/1.5pt})  node [above right] {$1$};
%\filldraw (0,{sqrt(\L^2-\m^2)}) circle ({0.6*1.5/1.5pt}) node [below right] {$\alpha\sqrt{L^2-\mu^2}$} ;
%\filldraw (0,{sqrt(1-2*\m+\L^2)}) circle ({0.6*1.5/1.5pt}) node [above right] {$\sqrt{1-2\alpha\mu+\alpha^2 L^2}$};

\filldraw (0,0) circle ({0.6*1.5/1.5pt}) ;

\draw (0.05,0) node[below left] {$O$};
%\filldraw (I3) -- (0,0);

\coordinate (O1) at ({1-1/\b},0);
\coordinate (O2) at ({1-1/\b/2},0);
\coordinate (O3) at ({1-\m},{sqrt(1/4/\b^2-(\m-1/2/\b)^2)});

\draw (1.2,-1) node {Case $\mu>1/(2\beta)$};

% \tkzMarkSegment[pos=.5,mark=||](O1,O2) ;
% \tkzMarkSegment[pos=.5,mark=||](O2,O3) ;
\end{tikzpicture}
}
\end{tabular}
\end{center}

The containment holds for $R$ and fails for smaller $R$.
Since $\cL_R$ is SRG-full by Theorem~\ref{thm:srg-full}, the containment of the SRG in $\ecomplex$ equivalent to the containment of the class.
% Since $\cL\left(\sqrt{1-2\alpha\mu+\alpha^2\mu/\beta}\right)$ is SRG-representable, the inclusion of the SRG in $\ecomplex$ implies the inclusion of the class by Theorem~\ref{thm:srg-representation}.
\qed\end{proof}

We quickly state Stewart's theorem \cite{stewart}, which we use for the Proof of Fact~\ref{prop:refl-sm-coco}.
For a triangle $\triangle ABC$ and Cevian $\overline{CD}$ to the side $\overline{AB}$,
\vspace{-0.2in}
\begin{center}
\begin{tabular}{c}
\raisebox{-.5\height}{
\begin{tikzpicture}[scale=1.5]
\coordinate (A) at  (-1,0);
\coordinate (B) at  (1,0);
\coordinate (C) at  (-.3,.7) ;
\coordinate (D) at  (.2,0) ;

\filldraw (A) circle ({0.6*1.5/1.5pt});
\filldraw (B) circle ({0.6*1.5/1.5pt});
\filldraw (C) circle ({0.6*1.5/1.5pt});
\filldraw (D) circle ({0.6*1.5/1.5pt});
\draw (A) node [below left] {$A$};
\draw (B) node [below right] {$B$};
\draw (C) node [above ] {$C$};
\draw (D) node [below] {$D$};

\draw  (A) -- (B)--(C)-- cycle;
\draw  (D)--(C);
\end{tikzpicture}}
\end{tabular}
\end{center}
the lengths of the line segments satisfy 
\[
\overline{AD}\cdot
\overline{CB}^2
+
\overline{DB}\cdot
\overline{AC}^2
=
\overline{AB}\cdot
\overline{CD}^2
+\overline{AD}\cdot\overline{DB}^2
+\overline{AD}^2\cdot\overline{DB}.
\]

\begin{proof}[Full Proof of Fact~\ref{prop:refl-sm-coco}]
By Proposition~\ref{prop:monotone-srg} and Theorems~\ref{thm:srg-scaling-translation} and \ref{thm:srg-inversion}, we have the geometry
\begin{center}
\begin{tabular}{ccccc}
\raisebox{-.5\height}{
\begin{tikzpicture}[scale=1.8]
\def\m{0.5};
\def\b{.8};
\begin{scope}
\clip ({1+\m},-1.2) rectangle (-1.2,1.2);
\draw[dashed] ({1+1/2/\b},0) circle ({1/2/\b});
\end{scope}
\draw [dashed] ({1+\m},{sqrt((1/2/\b)^2-(1/2/\b-\m)^2)}) -- ({1+\m},1.2);
\draw [dashed] ({1+\m},{-sqrt((1/2/\b)^2-(1/2/\b-\m)^2)}) -- ({1+\m},-1.2);

\begin{scope}
\clip ({1+\m},-1.2) rectangle (2.5,1.2);
\fill[fill=medgrey] ({1+1/2/\b},0) circle ({1/2/\b});
%\draw[line width=1.0pt] ({1+1/2/\b},0) circle ({1/2/\b});
% \fill[fill=medgrey] (1.2,-1.2) rectangle (2,1.2);
% \draw [line width=1.0pt] (1.2,-1.2) -- (1.2,1.2);
% \fill[fill=black!60] (1.3,0) circle (0.3);
% \draw[line width=1.0pt] (1.3,0) circle (0.3);
\end{scope}
\begin{scope}
\clip ({1+1/2/\b},0) circle ({1/2/\b});
%\draw [line width=1.0pt] ({1+\m},-1.2) -- ({1+\m},1.2);
\end{scope}
\filldraw ({1+\m},0)circle ({0.6*1.5/1.8pt});
\filldraw ({1+1/\b},0)  circle ({0.6*1.5/1.8pt});

\draw ({1+1/\b},0.0) node [above left] {$\scriptstyle1+\frac{\alpha}{\beta}$};
\draw [<->] (-1.2,0) -- (2.5,0);
\draw [<->] (0,-1.2) -- (0,1.2);
\draw [dashed] (0,0) circle (1);
\draw ({1+\m},0.05) node [above left,fill=white] {$\scriptstyle 1+\alpha\mu$};
\draw (1,-0pt) node [below left] {$\scriptstyle 1$};
\filldraw (1,0)circle ({0.6*1.5/1.8pt});
%\draw (.2,-.8) node [below,fill=white] {$\cG\left(I+\alpha(\cM_\mu\cap \cC_\beta)\right)$};
\draw (.1,-.6) node [below,fill=white] {$\scriptstyle\cG\left(I+\alpha\cA\right)$};
\draw [->] (.1,-.6) -- ({1+\m},-0.2);
\end{tikzpicture}}
&\!\!\!\!\!\!\!
$\stackrel{\bar{z}^{-1}}{\longrightarrow}$
&\!\!\!\!\!\!\!\!\!\!
\raisebox{-.5\height}{
\begin{tikzpicture}[scale=1.8]
\def\m{0.5};
\def\b{.8};

\begin{scope}
\clip ({1/2*(-\m/(1+\m)+sqrt(1/(1 + \m)/(1 + \m) + (4*\b*\m*(-1 + \b*\m))/(\b + \m + 2*\b*\m)/(\b + \m + 2*\b*\m)))+1/2},-1.2)  rectangle (-1.2,1.2);
\draw[dashed] ({1/(1+\m)/2},0) circle ({1/(1+\m)/2});
\end{scope}

\begin{scope}
\clip ({1/2*(-\m/(1+\m)+sqrt(1/(1 + \m)/(1 + \m) + (4*\b*\m*(-1 + \b*\m))/(\b + \m + 2*\b*\m)/(\b + \m + 2*\b*\m)))+1/2},-1.2)  rectangle (1.2,1.2);
\draw[dashed] ({(1+2*\b)/(2+2*\b)},0) circle ({(1)/(2+2*\b)});
\end{scope}

\begin{scope}
\clip ({1/(1+\m)/2},0) circle ({1/(1+\m)/2});
\fill[fill=medgrey] ({(1+2*\b)/(2+2*\b)},0) circle ({(1)/(2+2*\b)});
%\draw[line width=1.0pt] ({(1+2*\b)/(2+2*\b)},0) circle ({(1)/(2+2*\b)});
\end{scope}
\begin{scope}
\clip ({(1+2*\b)/(2+2*\b)},0) circle ({(1)/(2+2*\b)});
%\draw[line width=1.0pt] ({1/(1+\m)/2},0) circle ({1/(1+\m)/2});
\end{scope}
\draw [<->] (-1.2,0) -- (1.2,0);
\draw [<->] (0,-1.2) -- (0,1.2);
\draw [dashed] (0,0) circle (1);
\draw (1,-0pt) node [below right] {$\scriptstyle 1$};
\filldraw (1,0)circle ({0.6*1.5/1.8pt});

\filldraw ({1/(1+\m)},0)circle ({0.6*1.5/1.8pt});
\draw[->] (1.05,.4)--({1/(1+\m)},0);
\draw (1.15,.4) node [above] {$\scriptstyle\frac{1}{1+\alpha\mu}$};

\filldraw ({1/(1+1/\b)},0)circle ({0.6*1.5/1.8pt});
\draw[->] (.3,.5)--({1/(1+1/\b)},0);
\draw (.3,.45) node [above] {$\scriptstyle\frac{1}{1+\alpha/\beta}$};

\def\x{.5}
%\draw (.1,-.7) node [below,fill=white] {$\cG\left(\left(I+\alpha(\cM_\mu\cap \cC_\beta)\right)^{-1}\right)$};
%\draw (.1,-.7) node [below,fill=white] {$\cG\left(\left(I+\alpha\cA\right)^{-1}\right)$};
\draw (.1,-.6) node [below,fill=white] {$\scriptstyle\cG\left(J_{\alpha \cA}\right)$};
\draw [->] (.1,-.7) -- (\x,{-sqrt(((1)/(2+2*\b))^2-(\x-(1+2*\b)/(2+2*\b))^2)});
\end{tikzpicture}}\\
&\!\!\!\!\!\!\!\!\!\!\!\!\!\!\!\!
$\stackrel{2z-1}{\longrightarrow}$
&\!\!\!\!\!\!\!\!\!\!\!\!\!\!\!\!\!
\raisebox{-.5\height}{
\begin{tikzpicture}[scale=1.8]
\def\m{0.5};
\def\b{.8};
%\fill [fill=lightgrey] (0,0) circle ({sqrt((\b + \m - 2*\b*\m)/   (\b + \m + 2*\b*\m))});

\begin{scope}
\clip ({-\m/(1+\m)+sqrt(1/(1 + \m)/(1 + \m) + (4*\b*\m*(-1 + \b*\m))/(\b + \m + 2*\b*\m)/(\b + \m + 2*\b*\m))},-1.2)  rectangle (-1.2,1.2);
\draw[dashed] ({-\m/(1+\m)},0) circle ({1-\m/(1+\m)});
\end{scope}

\begin{scope}
\clip ({-\m/(1+\m)+sqrt(1/(1 + \m)/(1 + \m) + (4*\b*\m*(-1 + \b*\m))/(\b + \m + 2*\b*\m)/(\b + \m + 2*\b*\m))},-1.2)  rectangle (1.2,1.2);
\draw[dashed] ({1/(1+1/\b)},0) circle ({1-1/(1+1/\b)});
\end{scope}

\begin{scope}
\clip ({-\m/(1+\m)},0) circle ({1-\m/(1+\m)});
\fill[fill=medgrey] ({1/(1+1/\b)},0) circle ({1-1/(1+1/\b)});
%\draw[line width=1.0pt] ({1/(1+1/\b)},0) circle ({1-1/(1+1/\b)});
\end{scope}
\begin{scope}
\clip ({1/(1+1/\b)},0) circle ({1-1/(1+1/\b)});
%\draw[line width=1.0pt] ({-\m/(1+\m)},0) circle ({1-\m/(1+\m)});
\end{scope}
\draw [<->] (-1.2,0) -- (1.2,0);
\draw [<->] (0,-1.2) -- (0,1.2);
\draw [dashed] (0,0) circle (1);
\draw (1,-0pt) node [below right] {$\scriptstyle1$};
\filldraw (1,0)circle ({0.6*1.5/1.8pt});
\filldraw ({(1-\m)/(1+\m)},0)circle ({0.6*1.5/1.8pt});
\draw ({(1-\m)/(1+\m)-0.1},0) node [above right] {$\scriptstyle\frac{1-\alpha\mu}{1+\alpha\mu}$};
\filldraw ({(1-1/\b)/(1+1/\b)},0)circle ({0.6*1.5/1.8pt});
\draw ({(1-1/\b)/(1+1/\b)},0) node [above left] {$\scriptstyle\frac{\beta-\alpha}{\beta+\alpha}$};

\def\x{.05}
%\draw (-.1,-.9) node [below,fill=white] {$\cG\left(2\left(I+\alpha(\cM_\mu\cap \cC_\beta)\right)^{-1}-I\right)$};
\draw (-.05,-.6) node [below,fill=white] {$\scriptstyle\cG\left(2J_{\alpha \cA}-I\right)$};
\draw [->] (-.05,-.6) -- (\x,{-sqrt(((2)/(2+2*\b))^2-(\x-2*(1+2*\b)/(2+2*\b)+1)^2)});

%
%\def\y{.3}
%\draw [->] (.4,.7) -- (\y,{sqrt((sqrt((\b + \m - 2*\b*\m)/   (\b + \m + 2*\b*\m)))^2-(\y)^2)});
%\draw (.4,.7) node [above,fill=white] {$\cG\left(\cL\left(\sqrt{1-\frac{4\alpha\mu}{1+2\alpha\mu+\alpha^2\mu/\beta}}\right)\right)$};

\end{tikzpicture}}
\end{tabular}
\end{center}
A closer look gives us
%------------------------------------------------------------------------------
\begin{center}
\begin{tabular}{cc}
\raisebox{-.5\height}{
\begin{tikzpicture}[scale=2.5]
\def\m{0.5};
\def\b{.8};
\begin{scope}
\clip ({-\m/(1+\m)+sqrt(1/(1 + \m)/(1 + \m) + (4*\b*\m*(-1 + \b*\m))/(\b + \m + 2*\b*\m)/(\b + \m + 2*\b*\m))},-0.7)  rectangle (-1.2,0.9);
\draw[dashed] ({-\m/(1+\m)},0) circle ({1-\m/(1+\m)});
\end{scope}

\begin{scope}
\clip ({-\m/(1+\m)+sqrt(1/(1 + \m)/(1 + \m) + (4*\b*\m*(-1 + \b*\m))/(\b + \m + 2*\b*\m)/(\b + \m + 2*\b*\m))},-0.7)  rectangle (1.2,0.9);
\draw[dashed] ({1/(1+1/\b)},0) circle ({1-1/(1+1/\b)});
\end{scope}

\begin{scope}
\clip ({-\m/(1+\m)},0) circle ({1-\m/(1+\m)});
\fill[fill=medgrey] ({1/(1+1/\b)},0) circle ({1-1/(1+1/\b)});
%\draw[line width=1.0pt] ({1/(1+1/\b)},0) circle ({1-1/(1+1/\b)});
\end{scope}
\begin{scope}
\clip ({1/(1+1/\b)},0) circle ({1-1/(1+1/\b)});
%\draw[line width=1.0pt] ({-\m/(1+\m)},0) circle ({1-\m/(1+\m)});
\end{scope}
\draw [<->] (-1.2,0) -- (1.2,0);
\draw [<->] (0,-.7) -- (0,.7);
%\draw [dashed] (0,0) circle (1);
\draw (1,-0pt) node [below right] {$1$};
\draw (-1,-0pt) node [below left] {$-1$};
\filldraw (1,0)circle ({0.6*1.5/2.5pt});
\filldraw (-1,0)circle ({0.6*1.5/2.5pt});

\coordinate (O) at (0,0);
\filldraw (O) circle ({0.6*1.5/2.5pt});
\draw (O) node[below right] {$O$};

%\coordinate (C) at ({(1-\m)/(1+\m)},0);
%\filldraw (C) circle ({0.6*1.5/2.5pt});
%\draw (C) node[above right] {$C$};
%\draw (.4,-.3) node [below] {$\frac{1-\alpha\mu}{1+\alpha\mu}$};
%\draw[->] (.4,-.3)--({(1-\m)/(1+\m)},0);

%\filldraw ({(1-1/\b)/(1+1/\b)},0)circle ({0.6*1.5/2.5pt});
%\draw[->] (-.3,.4)--({(1-1/\b)/(1+1/\b)},0);
%\draw (-.3,.4) node [above] {$\frac{\beta-\alpha}{\beta+\alpha}$};

\coordinate (C) at ({1/(1+1/\b)},0);
\draw (C) node[above right] {$C$};
\filldraw (C)circle ({0.6*1.5/2.5pt});
\draw (.5,-.2) node [below] {$\frac{\beta}{\beta+\alpha}$};
\draw[->] (.5,-.2)--(C);

\coordinate (B) at ({-\m/(1+\m)},0);
\draw (B) node[above] {$B$};
\filldraw (B) circle ({0.6*1.5/2.5pt});
\draw (-.5,-.2) node [below] {$\frac{-\alpha\mu}{1+\alpha\mu}$};
\draw[->] (-.5,-.2)-- (B);

%\draw [line width=1.0pt] ({-\m/(1+\m)},0.2) -- ({-\m/(1+\m)+\b/(1 + \b) + \m/(1 + \m)},0.2);

%\draw [line width=1.0pt] ({-\m/(1+\m)+sqrt(1/(1 + \m)/(1 + \m) + (4*\b*\m*(-1 + \b*\m))/(\b + \m + 2*\b*\m)/(\b + \m + 2*\b*\m))},0) -- ({-\m/(1+\m)+sqrt(1/(1 + \m)/(1 + \m) + (4*\b*\m*(-1 + \b*\m))/(\b + \m + 2*\b*\m)/(\b + \m + 2*\b*\m))},{(2*sqrt(-(\b*\m*(-1 + \b*\m))))/(\b + \m + 2*\b*\m)});

\coordinate (A) at  ({-\m/(1+\m)+sqrt(1/(1 + \m)/(1 + \m) + (4*\b*\m*(-1 + \b*\m))/(\b + \m + 2*\b*\m)/(\b + \m + 2*\b*\m))},{(2*sqrt(-(\b*\m*(-1 + \b*\m))))/(\b + \m + 2*\b*\m)}) ;

%\draw [dashed] (0,0) circle ({sqrt((\b + \m - 2*\b*\m)/   (\b + \m + 2*\b*\m))});

\filldraw (A) circle ({0.6*1.5/2.5pt});
\draw (A) node[above] {$A$};

\draw (A) -- (B) -- (C) -- (A);
\draw (A) -- (O);

\coordinate(M1) at (-1,0);
\tkzMarkSegment[pos=.5,mark=||](A,B) ;
\tkzMarkSegment[pos=.5,mark=||](B,M1) ;

\coordinate(M2)  at (1,0);
\tkzMarkSegment[pos=.5,mark=|](A,C) ;
\tkzMarkSegment[pos=.5,mark=|](M2,C) ;

%\draw[->] (.5,.7)--({-\m/(1+\m)+sqrt(1/(1 + \m)/(1 + \m) + (4*\b*\m*(-1 + \b*\m))/(\b + \m + 2*\b*\m)/(\b + \m + 2*\b*\m))},{(2*sqrt(-(\b*\m*(-1 + \b*\m))))/(\b + \m + 2*\b*\m)});
%\draw (1,.7) node [above] {$
%\left(
%\sqrt{\frac{1}{(1+\alpha\mu)^2}-\frac{4\alpha^2\mu\beta(1-\mu\beta)}{(\beta+\alpha\mu(\alpha+2\beta))^2}}
%,\frac{2\alpha\sqrt{\mu\beta(1-\mu\beta)}}{\beta+\alpha\mu(\alpha+2\beta)}\right)
%$};

%\draw [decorate,decoration={brace,amplitude=4.5pt}] ({-\m/(1+\m)+sqrt(1/(1 + \m)/(1 + \m) + (4*\b*\m*(-1 + \b*\m))/(\b + \m + 2*\b*\m)/(\b + \m + 2*\b*\m))},{(2*sqrt(-(\b*\m*(-1 + \b*\m))))/(\b + \m + 2*\b*\m)}) -- (0,0);
%\draw [line width=1.0pt] (0,0) -- ({-\m/(1+\m)+sqrt(1/(1 + \m)/(1 + \m) + (4*\b*\m*(-1 + \b*\m))/(\b + \m + 2*\b*\m)/(\b + \m + 2*\b*\m))},{(2*sqrt(-(\b*\m*(-1 + \b*\m))))/(\b + \m + 2*\b*\m)});

%\draw[->](0.7,-0.7) -- ({(-\m/(1+\m)+sqrt(1/(1 + \m)/(1 + \m) + (4*\b*\m*(-1 + \b*\m))/(\b + \m + 2*\b*\m)/(\b + \m + 2*\b*\m)))/2+.07},{((2*sqrt(-(\b*\m*(-1 + \b*\m))))/(\b + \m + 2*\b*\m))/2-.03});
%\draw (0.8,-0.7) node [below] {
%$\sqrt{\frac{\beta+\alpha^2\mu-2\alpha\beta\mu}{\beta+\alpha\mu(\alpha+2\beta)}}$
%$\sqrt{1-\frac{4\alpha\mu}{1+2\alpha\mu+\alpha^2\mu/\beta}}$};
\end{tikzpicture}}
&
\!\!\!\!\!\!\!\!\!\!\!\!\!\!\!
\raisebox{-.5\height}{
\begin{tikzpicture}[scale=2.5]
\def\m{0.5};
\def\b{.8};

\begin{scope}
\clip ({-\m/(1+\m)+sqrt(1/(1 + \m)/(1 + \m) + (4*\b*\m*(-1 + \b*\m))/(\b + \m + 2*\b*\m)/(\b + \m + 2*\b*\m))},-0.7)  rectangle (-1.2,.9);
\draw[dashed] ({-\m/(1+\m)},0) circle ({1-\m/(1+\m)});
\end{scope}

\begin{scope}
\clip ({-\m/(1+\m)+sqrt(1/(1 + \m)/(1 + \m) + (4*\b*\m*(-1 + \b*\m))/(\b + \m + 2*\b*\m)/(\b + \m + 2*\b*\m))},-.7)  rectangle (1.2,.9);
\draw[dashed] ({1/(1+1/\b)},0) circle ({1-1/(1+1/\b)});
\end{scope}

\draw (-.7,.7) node [fill=white] {$R=\sqrt{1-\frac{4\alpha\mu}{1+2\alpha\mu+\alpha^2\mu/\beta}}$};

\fill [fill=lightgrey] (0,0) circle ({sqrt((\b + \m - 2*\b*\m)/   (\b + \m + 2*\b*\m))});

\begin{scope}
\clip ({-\m/(1+\m)},0) circle ({1-\m/(1+\m)});
\fill[fill=medgrey] ({1/(1+1/\b)},0) circle ({1-1/(1+1/\b)});
%\draw[line width=1.0pt] ({1/(1+1/\b)},0) circle ({1-1/(1+1/\b)});
\end{scope}
\begin{scope}
\clip ({1/(1+1/\b)},0) circle ({1-1/(1+1/\b)});
%\draw[line width=1.0pt] ({-\m/(1+\m)},0) circle ({1-\m/(1+\m)});
\end{scope}
\draw [<->] (-1.2,0) -- (1.2,0);
\draw  [<->](0,-.7) -- (0,.7);

\coordinate (A) at  ({-\m/(1+\m)+sqrt(1/(1 + \m)/(1 + \m) + (4*\b*\m*(-1 + \b*\m))/(\b + \m + 2*\b*\m)/(\b + \m + 2*\b*\m))},{(2*sqrt(-(\b*\m*(-1 + \b*\m))))/(\b + \m + 2*\b*\m)}) ;
\coordinate (Ap) at  ({-\m/(1+\m)+sqrt(1/(1 + \m)/(1 + \m) + (4*\b*\m*(-1 + \b*\m))/(\b + \m + 2*\b*\m)/(\b + \m + 2*\b*\m))},{-(2*sqrt(-(\b*\m*(-1 + \b*\m))))/(\b + \m + 2*\b*\m)}) ;
\filldraw (A) circle ({0.6*1.5/2.5pt});
\draw (A) node[above] {$A$};
\filldraw (Ap) circle ({0.6*1.5/2.5pt});
\draw (Ap) node[below] {$A'$};

\filldraw ({(1-\m)/(1+\m)},0)circle ({0.6*1.5/2.5pt});
\draw ({(1-\m)/(1+\m)-0.1},0) node [above right] {$E$};
\filldraw ({(1-1/\b)/(1+1/\b)},0)circle ({0.6*1.5/2.5pt});
\draw ({(1-1/\b)/(1+1/\b)},0) node [above left] {$D$};

\draw (0,0) node [below right] {$O$};
\filldraw (0,0)circle ({0.6*1.5/2.5pt});

\def\x{-.05}
%\draw (-.1,-.9) node [below,fill=white] {$\cG\left(2\left(I+\alpha(\cM_\mu\cap \cC_\beta)\right)^{-1}-I\right)$};
\draw (-.7,-.45) node [below,fill=white] {$\cG\left(2J_{\alpha \cA}-I\right)$};
\draw [->] (-.7,-.45) -- (\x,{-sqrt(((2)/(2+2*\b))^2-(\x-2*(1+2*\b)/(2+2*\b)+1)^2)});

\def\y{-.45}
\draw [->] (-.95,.45) -- (\y,{sqrt((sqrt((\b + \m - 2*\b*\m)/   (\b + \m + 2*\b*\m)))^2-(\y)^2)});
%\draw (.4,.7) node [above,fill=white] {$\cG\left(\cL\left(\sqrt{1-\frac{4\alpha\mu}{1+2\alpha\mu+\alpha^2\mu/\beta}}\right)\right)$};
\draw (-.95,.45) node [left,] {$\cG\left(\cL_R\right)$};
\end{tikzpicture}}
\end{tabular}
\end{center}
To clarify, $B$ is the center of the circle with radius $\overline{BA}$ and $C$ is the center of the circle with radius $\overline{CA}$.
By Stewart's theorem \cite{stewart}, we have
\begin{align*}
\overline{OA}^2&=\frac{\overline{OC}\cdot \overline{AB}^2+\overline{BO}\cdot\overline{CA}^2-\overline{BO}\cdot\overline{OC}\cdot\overline{BC}}{\overline{BC}}\\
&=\frac{
\tfrac{\beta }{\alpha +\beta }\left(1-\tfrac{\alpha  \mu
   }{1+\alpha  \mu }\right)^2+
   \tfrac {\alpha  \mu } {1+\alpha \mu}
   \left (1 - \tfrac {\beta } {\alpha + \beta } \right)^2
   -\tfrac{\beta }{\alpha +\beta }
   \tfrac {\alpha  \mu } {1+\alpha \mu }
   \left(\tfrac{\beta }{\alpha +\beta}+\tfrac{\alpha  \mu }{1+\alpha  \mu }\right)
   }{\tfrac{\beta }{\alpha +\beta}+\tfrac{\alpha  \mu }{1+\alpha  \mu }}\\
&   =1-\frac{4\alpha\mu}{1+2\alpha\mu+\alpha^2\mu/\beta}.
\end{align*}
Since $2$ non-identitcal circles intersect at at most $2$ points, and since $D$ is within circle $B$, arc $\arc{ADA'}$ is within circle $O$.
By the same reasoning, 
%Since $2$ non-identitcal circles intersect at at most $2$ points, and since $E$ is within circle $A$,
arc $\arc{A'EA}$ is within circle $O$.
Finally, the region bounded by $\arc{ADA'}\cup \arc{A'EA}$ (darker shade) is within circle $O$ (lighter shade).

The containment holds for $R$ and fails for smaller $R$.
Since $\cL_R$ is SRG-full by Theorem~\ref{thm:srg-full}, the containment of the SRG in $\ecomplex$ equivalent to the containment of the class.
%Since $\cL\left(\sqrt{1-\frac{4\alpha\mu}{1+2\alpha\mu+\alpha^2\mu/\beta}}\right)$ is SRG-representable, the inclusion of the SRG in $\ecomplex$ implies the inclusion of the class by Theorem~\ref{thm:srg-representation}.
\qed\end{proof}

%The composition of two firmly nonexpansive operators arise in alternating projections onto convex sets.
\begin{proof}[Proof of Fact~\ref{thm:fne-compose}]
Let
\begin{center}
\begin{tabular}{cc}
%$Q=\left\{\frac{1}{2}+\frac{r}{2}e^{i\varphi}\,|\,
%0\le r\le 1
%r\in [0,1],\,\varphi\in[0,2\pi]
%\right\}=$
$Q=$
\raisebox{-.5\height}{
\begin{tikzpicture}[scale=1]
\fill [fill=medgrey] (1/2,0) circle (1/2);
\draw [<->] (-.2,0) -- (1.2,0);
\draw [<->] (0,-.6) -- (0,.6);
\filldraw (1,0) circle ({0.6*1.5/1pt});
\draw (1,0) node[above right] {$1$};
\end{tikzpicture}}
&\qquad
%$C=\left\{\frac{1}{2}+\frac{1}{2}e^{i\varphi}\right\}=$
%\,\bigg|\,\varphi\in[0,2\pi]\right\}=$
$C=$
\raisebox{-.5\height}{
\begin{tikzpicture}[scale=1]
\draw  (1/2,0) circle (1/2);
\draw [<->] (-.2,0) -- (1.2,0);
\draw [<->] (0,-.6) -- (0,.6);
\filldraw (1,0) circle ({0.6*1.5/1pt});
\draw (1,0) node[above right] {$1$};
\end{tikzpicture}}
\end{tabular}
\end{center}
First, we show $Q=[0,1]C=[0,1]Q$.
%=\{tq\,|\,0\le t\le 1,\,q\in C\}.
To clarify, $[0,1]$ is the set of real numbers between $0$ and $1$ and 
$[0,1]C$ and $[0,1]Q$ are Minkowski products of sets of complex numbers.
Given any point $A\in Q\backslash C$, define $A'$ as the nonzero intersection of the line extending $\overline{OA}$ and circle $C$.
Since $A$ is on the line and inside the circle, the nonzero intersection $A'$ exists.
%Since a line and a circle intersect at zero, one, or two points, 
%the line and the circle intersect at $O$,
%and $A$ is on the line and inside the circle, second 
\begin{center}
\begin{tabular}{c}
\raisebox{-.5\height}{
\begin{tikzpicture}[scale=1.5]
%\draw [dashed] (0,0) circle (1);
\draw [] (1/2,0) circle (1/2);

%\filldraw  ({cos(\a/pi*180)/4+1/4},{sin(\a/pi*180)/4}) circle[radius={0.6*1.5/1.5pt}];

\draw [<->] (-.2,0) -- (1.2,0);
\draw [<->] (0,-.6) -- (0,.6);
%\draw [] (1,-\w) -- (1,\w);

\def\p{50};
\def\t{.6};
\coordinate (O) at (0,0);
\coordinate (Ap) at ({1/2+1/2*cos(\p)},{1/2*sin(\p)});
\coordinate (A) at ({\t*(1/2+1/2*cos(\p))},{\t*1/2*sin(\p)});

\filldraw (O) circle ({0.6*1.5/1.5pt});
\filldraw (Ap) circle ({0.6*1.5/1.5pt});
\filldraw (A) circle ({0.6*1.5/1.5pt});

\draw (A) node[above] {$A$};
\draw (Ap) node[above] {$A'$};
\draw (O) node[above left] {$O$};

\draw (1/2,-1/2) node[below] {$C$};

\draw (O) -- (Ap);
\end{tikzpicture}}

\end{tabular}
\end{center}
Since $A\in \overline{OA'}\subseteq [0,1]C$, we have
$Q\subseteq  [0,1]C$.
%Clearly $C\subseteq [0,1]C$, so $Q\subseteq [0,1]C$.
On the other hand, $Q\supseteq [0,1]C$ follows from noting that 
given any point $A'$ on $C$,
the line segment $\overline{OA'}$ is a chord of the circle $C$
and therefore is within the disk $Q$.
Therefore, $Q=[0,1]C$.
As a corollary, we have
$Q=[0,1]C=([0,1][0,1])C=[0,1]([0,1]C)=[0,1]Q$.

Next, define
\[
S=\bigcup_{0\le \varphi_1\le 2\pi}S_{\varphi_1},
\qquad
S_{\varphi_1}=
Q\left(\frac{1}{2}+\frac{1}{2}e^{i\varphi_1}\right).
\]
%\[
%S_{\varphi_1}=\left\{
%\left(\frac{1}{2}+\frac{r}{2}e^{i\varphi_2}\right)
%\left(\frac{1}{2}+\frac{1}{2}e^{i\varphi_1}\right)
%\,\bigg|\,r\in [0,1],\,\varphi_2\in [0,2\pi]
%\right\}
%\]

In geometric terms,  this construction takes a point on the circle $C$, draws the disk whose diameter is the line segment between this point and the origin, and takes the union of such disks.
\begin{center}
\begin{tabular}{c}
\raisebox{-.5\height}{
\begin{tikzpicture}[scale=3]
\def\a{pi*.66/3};
\def\b{pi*.66*2/3};
\def\c{pi*.66};
\def\w{1.2};
\def\u{1.2};

\fill[fill=lightgrey] ({cos(\a/pi*180)/4+1/4},{sin(\a/pi*180)/4}) circle ({cos(\a/pi*180/2)/2});
\coordinate (A1) at ({1},{tan(\a/pi*180/2)}) ;
\coordinate (A2) at ({1/2*cos(\a/pi*180)+1/2},{sin(\a/pi*180)/2});
\filldraw (A2) circle[radius={0.6*1.5/3pt}];
%\filldraw (A1)  circle[radius={0.6*1.5/3pt}];
\draw [line width=1.0pt] (0,0) -- (A2);
%\draw [line width=1.0pt] ({1-tan(\a/pi*180/2)*(\w-tan(\a/pi*180/2))},{\w}) -- ({1+tan(\a/pi*180/2)*(\w-tan(\a/pi*180/2))},{2*tan(\a/pi*180/2)-\w});

\fill[fill=medgrey] ({cos(\b/pi*180)/4+1/4},{sin(\b/pi*180)/4}) circle ({cos(\b/pi*180/2)/2});
\coordinate (B1) at ({1},{tan(\b/pi*180/2)}) ;
\coordinate (B2) at ({1/2*cos(\b/pi*180)+1/2},{sin(\b/pi*180)/2});
\filldraw (B2) circle[radius={0.6*1.5/3pt}];
%\filldraw (B1)  circle[radius={0.6*1.5/3pt}];
\draw [line width=1.0pt] (0,0) -- (B2);
%\draw [line width=1.0pt] ({1-tan(\b/pi*180/2)*(\w-tan(\b/pi*180/2))},{\w}) -- ({1+tan(\b/pi*180/2)*(\w-tan(\b/pi*180/2))},{2*tan(\b/pi*180/2)-\w});

\fill[fill=darkgrey] ({cos(\c/pi*180)/4+1/4},{sin(\c/pi*180)/4}) circle ({cos(\c/pi*180/2)/2});
\coordinate (C1) at ({1},{tan(\c/pi*180/2)}) ;
\coordinate (C2) at ({1/2*cos(\c/pi*180)+1/2},{sin(\c/pi*180)/2});
\filldraw (C2) circle[radius={0.6*1.5/3pt}];
%\filldraw (C1)  circle[radius={0.6*1.5/3pt}];
\draw [line width=1.0pt] (0,0) -- (C2);
%\draw[ line width=1.0pt] ({1-tan(\c/pi*180/2)*(\w-tan(\c/pi*180/2))},{\w}) -- ({1+tan(\c/pi*180/2)*(\w-tan(\c/pi*180/2))},{2*tan(\c/pi*180/2)-\w});

\draw [dashed] (0,0) circle (1);
\draw [] (1/2,0) circle (1/2);

%\filldraw  ({cos(\a/pi*180)/4+1/4},{sin(\a/pi*180)/4}) circle[radius={0.6*1.5/3pt}];

\filldraw (0,0) circle[radius={0.6*1.5/3pt}];

\draw [<->] (-1.2,0) -- (\u,0);
\draw [<->] (0,-\w) -- (0,\w);
%\draw [] (1,-\w) -- (1,\w);

\coordinate (O) at (0,0);
\coordinate (E) at (1,0);
%\tkzMarkRightAngle[size=.053](O,A2,E);
%\tkzMarkRightAngle[size=.053](O,B2,E);
%\tkzMarkRightAngle[size=.053](O,C2,E);

\def\r{0.15};
\filldraw (1/2,0) circle[radius={0.6*1.5/3pt}];
\draw (1/2,0) -- (C2);
\draw ({1/2+\r},{0}) arc (0:{\c/pi*180}:\r) node[midway,right]{$\varphi_1$};

\def\f{-0.12};
\draw[->] (-0.3,0.3) --(\f,{sqrt((cos(\c/pi*180/2)/2)^2-(cos(\c/pi*180)/4+1/4-\f)^2)+sin(\c/pi*180)/4});
\draw  (-0.3,0.3) node[left] {$S_{\varphi_1}$};
\draw (0.5,-0.5) node[below] {$C$};
\end{tikzpicture}}

\end{tabular}
\end{center}
The dashed circle is the unit circle.
The solid circle is $C$.
%Circle $C$, the solid circle, represents the points $(1/2)+(1/2)e^{i\varphi_1}$ for $\varphi_1\in[0,2\pi]$.
The shaded circles represent instances of $S_{\varphi_1}$.
We can characterize $\cG(\cN_{1/2}\cN_{1/2})$ by analyzing this construction
since
\begin{align*}
S&=QC=(Q [0,1]) C=
Q([0,1]C)\\
&=QQ=\cG(\cN_{1/2})\cG(\cN_{1/2})=\cG(\cN_{1/2}\cN_{1/2})
\end{align*}
by Proposition~\ref{prop:monotone-srg} and Theorem~\ref{thm:composition-srg}.
%, $\cG(\cN_{1/2})=Q$.
%Since $\cN_{1/2}$ is SRG-full and satisfies the right-arc property, 
% tells us that $\cG(\cN_{1/2}\cN_{1/2})=QQ=S$.

We now show $S=\left\{re^{i\varphi}\,|\,0\le r\le \cos^2(\varphi/2)\right\}$.
This fact is known and can be analytically derived through the envelope theorem \cite[Exercise 5.22]{riley2006}.
We provide a geometric proof, which was inspired by \cite[Exercise 4.15]{kline1977}.

Throughout this proof, we write $\cI\colon\ecomplex\rightarrow\ecomplex$ for the mapping $\cI(z)=\bar{z}^{-1}$.
We map $S$ into the inverted space, i.e.,
%Next, we analyze $S$ in the inverted space, i.e.,
we analyze 
\[
\cI(S)=\bigcup_{0\le \varphi_1\le 2\pi}\cI(S_{\varphi_1}).
\]
\newpage

%\begin{figure}
\begin{center}
\begin{tabular}{c}
\raisebox{-.5\height}{
\begin{tikzpicture}[scale=3]
\def\a{pi*.66/3};
\def\b{pi*.66*2/3};
\def\c{pi*.66};
\def\w{2};
\def\u{2.5};

\begin{scope}
\clip (-1.2,-1.2) rectangle (\u,\w);

\coordinate (A3) at  ({1-tan(\a/pi*180/2)*(\w-tan(\a/pi*180/2))},{\w});
\fill[fill=lightgrey] (A3)
-- (\u,{\w-(\u-(1-tan(\a/pi*180/2)*(\w-tan(\a/pi*180/2))))/tan(\a/pi*180/2)})
-- (\u,\w)
-- cycle;

\draw (2,-1) node {$\cup \{\binfty\}$};

%-- ({1+tan(\a/pi*180/2)*(\w-tan(\a/pi*180/2))},{2*tan(\a/pi*180/2)-\w})
%-- ({1+tan(\a/pi*180/2)*(\w-tan(\a/pi*180/2))},{\w})

\coordinate (B3) at  ({1-tan(\b/pi*180/2)*(\w-tan(\b/pi*180/2))},{\w});
\fill[fill=medgrey] (B3)
-- (\u,{\w-(\u-(1-tan(\b/pi*180/2)*(\w-tan(\b/pi*180/2))))/tan(\b/pi*180/2)}) 
-- (\u,\w)
-- cycle;
\draw (2,.6) node {$\cup \{\binfty\}$};

\coordinate (C3) at ({1-tan(\c/pi*180/2)*(\w-tan(\c/pi*180/2))},{\w});
\fill[fill=darkgrey] (C3)
--(\u,{\w-(\u-(1-tan(\c/pi*180/2)*(\w-tan(\c/pi*180/2))))/tan(\c/pi*180/2)}) 
-- (\u,\w)
-- cycle;

\draw (2,1.7) node {$\cup \{\binfty\}$};

\draw (1.9,1.35) node {$\cI(S_{\varphi_1})$};

\end{scope}

\fill[fill=lightgrey] ({cos(\a/pi*180)/4+1/4},{sin(\a/pi*180)/4}) circle ({cos(\a/pi*180/2)/2});
\coordinate (A1) at ({1},{tan(\a/pi*180/2)}) ;
\coordinate (A2) at ({1/2*cos(\a/pi*180)+1/2},{sin(\a/pi*180)/2});
\filldraw (A2) circle[radius={0.6*1.5/3pt}];
\filldraw (A1)  circle[radius={0.6*1.5/3pt}];
\draw [line width=1.0pt] (0,0) -- ({1},{tan(\a/pi*180/2)});
%\draw [line width=1.0pt] ({1-tan(\a/pi*180/2)*(\w-tan(\a/pi*180/2))},{\w}) -- ({1+tan(\a/pi*180/2)*(\w-tan(\a/pi*180/2))},{2*tan(\a/pi*180/2)-\w});

\fill[fill=medgrey] ({cos(\b/pi*180)/4+1/4},{sin(\b/pi*180)/4}) circle ({cos(\b/pi*180/2)/2});
\coordinate (B1) at ({1},{tan(\b/pi*180/2)}) ;
\coordinate (B2) at ({1/2*cos(\b/pi*180)+1/2},{sin(\b/pi*180)/2});
\filldraw (B2) circle[radius={0.6*1.5/3pt}];
\filldraw (B1)  circle[radius={0.6*1.5/3pt}];
\draw [line width=1.0pt] (0,0) -- ({1},{tan(\b/pi*180/2)});
%\draw [line width=1.0pt] ({1-tan(\b/pi*180/2)*(\w-tan(\b/pi*180/2))},{\w}) -- ({1+tan(\b/pi*180/2)*(\w-tan(\b/pi*180/2))},{2*tan(\b/pi*180/2)-\w});

\fill[fill=darkgrey] ({cos(\c/pi*180)/4+1/4},{sin(\c/pi*180)/4}) circle ({cos(\c/pi*180/2)/2});
\coordinate (C1) at ({1},{tan(\c/pi*180/2)}) ;
\coordinate (C2) at ({1/2*cos(\c/pi*180)+1/2},{sin(\c/pi*180)/2});
\filldraw (C2) circle[radius={0.6*1.5/3pt}];
\filldraw (C1)  circle[radius={0.6*1.5/3pt}];
\draw [line width=1.0pt] (0,0) -- ({1},{tan(\c/pi*180/2)});
%\draw[ line width=1.0pt] ({1-tan(\c/pi*180/2)*(\w-tan(\c/pi*180/2))},{\w}) -- ({1+tan(\c/pi*180/2)*(\w-tan(\c/pi*180/2))},{2*tan(\c/pi*180/2)-\w});

\draw [dashed] (0,0) circle (1);
\draw [] (1/2,0) circle (1/2);

%\filldraw  ({cos(\a/pi*180)/4+1/4},{sin(\a/pi*180)/4}) circle[radius={0.6*1.5/3pt}];

\filldraw (0,0) circle[radius={0.6*1.5/3pt}];

%\draw [dashed] (-1,0) circle (2);
% \filldraw (-3,0) circle[radius={0.6*1.5/3pt}];
% \draw (-3,0) node [above right] {$-3$};

\draw [<->] (-1.2,0) -- (\u,0);
\draw [<->] (0,-1.2) -- (0,\w);
\draw [] (1,-1.2) -- (1,\w);

\draw (0.5,-0.5) node[below] {$C$};

\draw (1,-1) node[right] {$\cI(C)$};

%\coordinate (C) at ($(C)!(A)!(B)$);
%%\draw (A)node[below left]{$A$}--(B)node[below right]{$B$}--(C)node[above left]{$C$}--cycle;
%%\draw[dashed] (A)--(D)node[above right]{$D$};

\tkzMarkRightAngle[size=.053](A2,A1,A3);
\tkzMarkRightAngle[size=.053](B2,B1,B3);
\tkzMarkRightAngle[size=.053](C2,C1,C3);

\coordinate (O) at (0,0);
\coordinate (E) at (1,0);
\tkzMarkRightAngle[size=.053](A1,A2,E);
\tkzMarkRightAngle[size=.053](B1,B2,E);
\tkzMarkRightAngle[size=.053](C1,C2,E);
%\tkzMarkRightAngle[size=.075]((0,0),B2,(1,0));
%\tkzMarkRightAngle[size=.075]((0,0),C2,(1,0));

\draw (O) node[below left] {$O$};
\draw (C2) node[above left] {$A$};
\draw (C1) node[above right] {$\cI(A)$};

\def\f{-0.12};
\draw[->] (-0.3,0.3) --(\f,{sqrt((cos(\c/pi*180/2)/2)^2-(cos(\c/pi*180)/4+1/4-\f)^2)+sin(\c/pi*180)/4});
\draw  (-0.3,0.3) node[left] {$S_{\varphi_1}$};
\end{tikzpicture}}

\end{tabular}
\end{center}
%\caption{XXX}
%\end{figure}
Again, $\cI(z)=\bar{z}^{-1}$.
The dashed circle, the unit circle, is mapped onto itself.
Circle $C$, the solid circle, is mapped to  $\cI(C)$, the verticle line going through $1$.
Each shaded circle $S_{\varphi_1}$ is mapped to a half-space $\cI(S_{\varphi_1})$.
Let point $A$ be the nonzero intersection between $C$ and the boundary of $S_{\varphi_1}$.
Then point $\cI(A)$ is the non-infinite intersection between $\cI(C)$ and the boundary of $\cI(S_{\varphi_1})$.
%$S_{\varphi_1}$ and $C$ intersect at $A$ and $O$
By construction,  $\overline{OA}$ is the diameter of $S_{\varphi_1}$.
The (infinite) line containing $O$, $A$, and $\cI(A)$ is mapped onto itself, excluding the origin.
Since $\cI$ is conformal, the right angle at $A$ between the boundary of $S_{\varphi_1}$ and the diameter $\overline{OA}$
is mapped to a right angle between boundary of $\cI(S_{\varphi_1})$ and $\overline{O\cI(A)}$.
\newpage

%\begin{figure}
\begin{center}
\begin{tabular}{c}
\raisebox{-.5\height}{
\begin{tikzpicture}[scale=3]
\def\a{pi*.17};
\def\b{pi*.45};
\def\c{pi*.66};
\def\w{2};
\def\u{2.5};

\begin{scope}
\clip (-1.2,-1) rectangle (\u,\w);

\coordinate (A3) at  ({1-tan(\a/pi*180/2)*(\w-tan(\a/pi*180/2))},{\w});
\fill[fill=lightgrey] (A3)
-- (\u,{\w-(\u-(1-tan(\a/pi*180/2)*(\w-tan(\a/pi*180/2))))/tan(\a/pi*180/2)})
-- (\u,\w)
-- cycle;

%-- ({1+tan(\a/pi*180/2)*(\w-tan(\a/pi*180/2))},{2*tan(\a/pi*180/2)-\w})
%-- ({1+tan(\a/pi*180/2)*(\w-tan(\a/pi*180/2))},{\w})

\coordinate (B3) at  ({1-tan(\b/pi*180/2)*(\w-tan(\b/pi*180/2))},{\w});
\fill[fill=medgrey] (B3)
-- (\u,{\w-(\u-(1-tan(\b/pi*180/2)*(\w-tan(\b/pi*180/2))))/tan(\b/pi*180/2)}) 
-- (\u,\w)
-- cycle;

%\coordinate (C3) at ({1-tan(\c/pi*180/2)*(\w-tan(\c/pi*180/2))},{\w});
%\fill[fill=darkgrey] (C3)
%--(\u,{\w-(\u-(1-tan(\c/pi*180/2)*(\w-tan(\c/pi*180/2))))/tan(\c/pi*180/2)}) 
%-- (\u,\w)
%-- cycle;

\end{scope}

\coordinate (A1) at ({1},{tan(\a/pi*180/2)}) ;
\coordinate (A2) at ({1/2*cos(\a/pi*180)+1/2},{sin(\a/pi*180)/2});
%\filldraw (A2) circle[radius={0.6*1.5/3pt}];
%\filldraw (A1)  circle[radius={0.6*1.5/3pt}];
\draw [line width=1.0pt] (0,0) -- ({2},{2*tan(\a/pi*180/2)});
%\draw [line width=1.0pt] ({1-tan(\a/pi*180/2)*(\w-tan(\a/pi*180/2))},{\w}) -- ({1+tan(\a/pi*180/2)*(\w-tan(\a/pi*180/2))},{2*tan(\a/pi*180/2)-\w});

\coordinate (B1) at ({1},{tan(\b/pi*180/2)}) ;
\coordinate (B2) at ({1/2*cos(\b/pi*180)+1/2},{sin(\b/pi*180)/2});
%\filldraw (B2) circle[radius={0.6*1.5/3pt}];
%\filldraw (B1)  circle[radius={0.6*1.5/3pt}];
\draw [line width=1.0pt] (0,0) -- ({2},{2*tan(\b/pi*180/2)});
%\draw [line width=1.0pt] ({1-tan(\b/pi*180/2)*(\w-tan(\b/pi*180/2))},{\w}) -- ({1+tan(\b/pi*180/2)*(\w-tan(\b/pi*180/2))},{2*tan(\b/pi*180/2)-\w});

\begin{scope}
\clip (-1.2,-1) rectangle (\u,\w);

\coordinate (C1) at ({1},{tan(\c/pi*180/2)}) ;
\coordinate (C2) at ({1/2*cos(\c/pi*180)+1/2},{sin(\c/pi*180)/2});
%\filldraw (C2) circle[radius={0.6*1.5/3pt}];
%\filldraw (C1)  circle[radius={0.6*1.5/3pt}];
%\draw [line width=1.0pt] (0,0) -- ({2},{2*tan(\c/pi*180/2)});
%\draw[ line width=1.0pt] ({1-tan(\c/pi*180/2)*(\w-tan(\c/pi*180/2))},{\w}) -- ({1+tan(\c/pi*180/2)*(\w-tan(\c/pi*180/2))},{2*tan(\c/pi*180/2)-\w});
\end{scope}
%\draw [] (1/2,0) circle (1/2);

%\filldraw  ({cos(\a/pi*180)/4+1/4},{sin(\a/pi*180)/4}) circle[radius={0.6*1.5/3pt}];

\filldraw (0,0) circle[radius={0.6*1.5/3pt}];
\filldraw (1,0) circle[radius={0.6*1.5/3pt}];

%\draw [dashed] (-1,0) circle (2);
% \filldraw (-3,0) circle[radius={0.6*1.5/3pt}];
% \draw (-3,0) node [above right] {$-3$};

\draw [<->] (-1.2,0) -- (\u,0);
\draw [<->] (0,-1) -- (0,\w);
\draw [] (1,-1) -- (1,\w);

%\coordinate (C) at ($(C)!(A)!(B)$);
%%\draw (A)node[below left]{$A$}--(B)node[below right]{$B$}--(C)node[above left]{$C$}--cycle;
%%\draw[dashed] (A)--(D)node[above right]{$D$};

\tkzMarkRightAngle[size=.053](A2,A1,A3);
\tkzMarkRightAngle[size=.053](B2,B1,B3);
%\tkzMarkRightAngle[size=.053](C2,C1,C3);

%\tkzMarkRightAngle[size=.075]((0,0),B2,(1,0));
%\tkzMarkRightAngle[size=.075]((0,0),C2,(1,0));

\draw[line width=1.0pt, domain=-1:\w,variable=\y] plot ({1-(\y)^2/4},{\y});

\coordinate (O) at (0,0);
\coordinate (E) at (2,-\w);

\coordinate (A7) at  ({1-tan(\a/pi*180/2)^2},{2*tan(\a/pi*180/2)}) ;
\coordinate (A8) at (2,{2*tan(\a/pi*180/2)}) ;

\coordinate (B7) at  ({1-tan(\b/pi*180/2)^2},{2*tan(\b/pi*180/2)}) ;
\coordinate (B8) at (2,{2*tan(\b/pi*180/2)}) ;

\filldraw (A7) circle[radius={0.6*1.5/3pt}];
\filldraw (B7) circle[radius={0.6*1.5/3pt}];

\tkzMarkRightAngle[size=.053](A7,A8,E);
\tkzMarkRightAngle[size=.053](B7,B8,E);

\draw (O) -- (A7) -- (A8) ;
\draw (O) -- (B7) -- (B8) ;

\tkzMarkSegment[pos=.5,mark=||](O,A7) ;
\tkzMarkSegment[pos=.5,mark=||](A7,A8) ;

\tkzMarkSegment[pos=.5,mark=|](O,B7) ;
\tkzMarkSegment[pos=.5,mark=|](B7,B8) ;

\draw (2.25,1.2) node {$\cup \{\binfty\}$};
\draw (2.25,-.85) node {$\cup \{\binfty\}$};

\draw (0,0) node [above left] {focus};
\draw (0,0) node [below left] {$O$};
\draw (1,0) node [above right] {vertex};
\draw (1,0) node [below right] {$1$};

\filldraw (2,0) circle[radius={0.6*1.5/3pt}];
%\draw (2,0) node [above right] {directrix};
\draw (2,0) node [below right] {$2$};
\draw [<->] (2,-1) -- (2,\w);

\filldraw (B1) circle[radius={0.6*1.5/3pt}];
\draw (B1) node[right] {$A'$};
\filldraw (B8) circle[radius={0.6*1.5/3pt}];
\draw (B8) node[right] {$E$};
\draw (B7) node[above ] {$B$};

%\draw [dashed] (-1,0) circle (2);
% \filldraw (-3,0) circle[radius={0.6*1.5/3pt}];
% \draw (-3,0) node [above right] {$-3$};

\draw[->] (-0.2,-.3) node[left]{$x=1-\frac{y^2}{4}$} -- ({1-(-.5)^2/4},-.5) ;
%\draw  (1.5,-1.5) node [] {$x=1-\frac{y^2}{4}$};
\draw (1,-.75) node[right] {$\cI(C)$};
\draw (2.,.3) node[right] {$D$};
\draw (2,.17) node[right] {(directrix)};

\coordinate (Bp) at  ({1-tan(\b/pi*180/2)^2+0.6*cos(90-\b/pi*180/2)},{2*tan(\b/pi*180/2)-0.6*sin(90-\b/pi*180/2)}) ;
\filldraw (Bp) circle[radius={0.6*1.5/3pt}];
\draw (Bp) node[right] {$B'$};
\end{tikzpicture}}

\end{tabular}
\end{center}
%\caption{XXX}
%\end{figure}
Next, we show that the union of the half-spaces is described by a parabola.
Define line $D$ as the vertical line going through $2$.
Consider any point $A'$ on the line $\cI(C)$.
Consider the line through $A'$ perpendicular to $\overline{OA'}$ and the half-space to the right of the line including $\infty$.
%(Note that the boundary is not horizontal.)
Define point $E$ as the intersection of line $D$ and the line extending $\overline{OA}$.
Draw a line through point $E$ perpendicular to line $D$, and define point $B$ as the intersection of this line with the boundary of the half-space.
Since $E$ is within the half-space and the boundary of the half-space is not horizontal, this intersection exists and $B$ is to the left of $E$.
Since $\overline{OA'}=\overline{A'E}$, we have $\triangle BA'O\cong\triangle BA'E$ by the side-angle-side (SAS) congruence,  and $\overline{OB}=\overline{BE}$.
The union of all such points corresponding to $B$ forms a parabola. with directrix $D$, focus $(0,0)$ and vertex $(1,0)$.

The boundary of the half-space is tangent to the parabola at $B$, i.e., the line intersects the parabola at no other point.
To see why, consider any other point $B'$ on the boundary of the half-space.
Then $\overline{OB'}=\overline{B'E}$ by SAS congruence.
However, $\overline{B'E}$ is not perpendicular to $D$, i.e., $\overline{B'E}$ is not a horizontal line.
%since $\overline{BE}$ is horizontal and $\overline{BB'}$ is non-horizontal.
Therefore  $\overline{OB'}$ is longer than the distance of $B'$ to $D$, and therefore $B'$ is not on the parabola.
%Therefore $B$ is the only point where the boundary of the half-space and the parabola intersect.
Since each half-space is tangent to the parabola at $B$, all points to the right of $B$ are in $\cI(S)$ and no points strictly to the left of the parabola are in $\cI(S)$.
Therefore $\cI(S)$ is characterized by the closed region to the right of the parabola including $\infty$.

\begin{center}
\begin{tabular}{ccc}
\raisebox{-.5\height}{
\begin{tikzpicture}[scale=.9]
\def\a{pi*.17};
\def\b{pi*.45};
\def\c{pi*.66};
\def\w{2.5};
\def\u{3.5};

\fill[fill=medgrey] (-\u,-\w) rectangle (1.5,\w);
\begin{scope}
\clip (-\u,-\w) rectangle (1.5,\w);
\fill[fill=white] (-1,0) circle (2);
\end{scope}

\fill[fill=darkgrey, domain=-\w:\w,variable=\y] (1.5,{-\w}) -- plot ({1-(\y)^2/4},{\y}) -- (1.5,{\w});
\filldraw (-1,0) circle[radius={0.6*1.5/0.9pt}];
\draw (-1,0) node [below ] {$-1$};
\filldraw (1,0) circle[radius={0.6*1.5/0.9pt}];
\draw (1,0) node [below left] {$1$};
\filldraw (-3,0) circle[radius={0.6*1.5/0.9pt}];
\draw (-3,0) node [below right] {$-3$};
%\fill[fill=darkgrey] ({1-(\w)^2/4},-\w) rectangle (\u,\w);
%\begin{scope}
%\clip[domain=-\w:\w,variable=\y] plot ({1-(\y)^2/4},{\y});
%\fill[fill=medgrey] (-1.5,-\w) rectangle (\u,\w);
%\end{scope}
%
%
%\begin{scope}
%\clip[domain=-\w:\w,variable=\y] plot ({1-(\y)^2/4},{\y});
%\fill[fill=medgrey] (-1.5,-\w) rectangle (\u,\w);
%\end{scope}
%\fill[fill=white] (-1,0) circle (2);
%

\draw [<->] (-\u,0) -- (1.5,0);
\draw [<->] (0,-\w) -- (0,\w);
%
%\draw[line width=1.0pt, domain=-\w:\w,variable=\y] plot ({1-(\y)^2/4},{\y});

\begin{scope}
\clip (-1,0) circle (2);
\draw[->] (-.65,.15) node[above, text width=3cm,fill=white]{\footnotesize $x=1-\frac{y^2}{4}$ and \phantom{xxxx} $x=\sqrt{4-y^2}-1$ \phantom{xx}  have matching curvature} -- (1,0) ;
\end{scope}

\draw (.95,2.2) node {$\cup \{\binfty\}$};

\draw (-2.7,2.2) node {$\cup \{\binfty\}$};

\end{tikzpicture}}
&$\stackrel{\cI}{\longrightarrow}$
&
\raisebox{-.5\height}{
\begin{tikzpicture}[scale=1.5]
\def\a{pi*.17};
\def\b{pi*.45};
\def\c{pi*.66};
\def\w{1.2};
\def\u{1.2};

\draw [dashed] (0,0) circle (1);
\fill[fill=medgrey] (1/3,0) circle (2/3);
\fill[fill=darkgrey,samples=200,smooth] plot[domain=-pi:pi] (xy polar cs:angle=\x r,radius= {cos(\x r/2)^2});
%\draw [decorate,decoration={brace,amplitude=4.5pt}] (1/3,0) -- (1,0) ;
%\draw (1/3,0.1) node [above] {$1/3$};
\filldraw (-1/3,0) circle[radius={0.6*1.5/1.5pt}];
\draw (-1/3,0) node [above left] {$-\frac{1}{3}$};
%
%
%

%\draw[domain=-10:10, variable=\t] plot ({4*(\t-(\t)^2)/((4+(\t)^2)^2)},{16*\t/((4+(\t)^2)^2)});

%\fill[fill=darkgrey, domain=-\w:\w,variable=\y] (\u,{-\w}) -- plot ({1-(\y)^2/4},{\y}) -- (\u,{\w});

\filldraw (1,0) circle[radius={0.6*1.5/1.5pt}];
\draw (1,0) node [above right] {$1$};
%\fill[fill=darkgrey] ({1-(\w)^2/4},-\w) rectangle (\u,\w);
%\begin{scope}
%\clip[domain=-\w:\w,variable=\y] plot ({1-(\y)^2/4},{\y});
%\fill[fill=medgrey] (-1.5,-\w) rectangle (\u,\w);
%\end{scope}
%
%
%\begin{scope}
%\clip[domain=-\w:\w,variable=\y] plot ({1-(\y)^2/4},{\y});
%\fill[fill=medgrey] (-1.5,-\w) rectangle (\u,\w);
%\end{scope}
%\fill[fill=white] (-1,0) circle (2);
%

\draw [<->] (-1.2,0) -- (\u,0);
\draw [<->] (0,-\w) -- (0,\w);
%
%\draw[line width=1.0pt, domain=-\w:\w,variable=\y] plot ({1-(\y)^2/4},{\y});
\end{tikzpicture}}
\end{tabular}
\end{center}
%The largest circle symmetric about the horizontal axis touching the point $(1,0)$ has radius $2$. This fact is easily verified with calculus.
The region exterior to the circle centered at $-1$ with radius $2$ contains the region towards the right of the parabola.
This is easily verified with calculus.
%The largest circle symmetric about the horizontal axis touching the point $(1,0)$ has radius $2$. This fact is easily verified with calculus.
The circle with the lighter shade corresponds to $\cG(\cN_{2/3})$ by  Proposition~\ref{prop:monotone-srg}.
Since $\cN_{2/3}$ is SRG-full by Theorem~\ref{thm:srg-full}, strict containment of the SRG in $\ecomplex$ implies strict containment of the class.
The inverse curve of the parabola with the focus as the center of inversion is known as the \emph{cardioid}
and it has the polar coordinate representation $r(\varphi)\le \cos^2(\varphi/2)$. The expression of the Theorem is the region bounded by this curve.
\qed\end{proof}

%\[
%\left\{
%\frac{1}{x^2+y^2}(x,y)
%\,\Big|\,x\ge 1-\frac{y^2}{4}
%\right\}
%\]
%which translates to a polar coordinate representation

%https://math.stackexchange.com/questions/824027/non-brute-force-proof-of-parabola-tangent-property

\end{document}